\documentclass[12pt,a4paper]{amsart}

\usepackage[utf8]{inputenc}

\usepackage{amsfonts,amsmath, amsthm, amssymb, paralist, xspace, graphicx, url, amscd, euscript, mathrsfs, epic, eepic,color,longtable, mathtools}

\usepackage[normalem]{ulem}

\usepackage[all]{xy}

\usepackage[colorlinks=true]{hyperref}

\newcommand{\ul}{\underline}

\newcommand{\ddata}{\beta} 

\newcommand{\cM}{\mathcal{M}} 

\newcommand{\cK}{\mathcal{K}} 

\newcommand{\oM}{\overline{\cM}} 
\newcommand{\oN}{\overline{\cN}} 

\newcommand{\NN}{\mathbb{N}} 
\newcommand{\ZZ}{\mathbb{Z}} 
\newcommand{\CC}{\mathbb C}

\newcommand{\fM}{\mathfrak{M}} 
\newcommand{\fC}{\mathfrak{C}} 

\newcommand{\fU}{\mathfrak{U}} 

\newcommand{\fV}{\mathfrak{V}} 

\newcommand{\fH}{\mathfrak{H}} 

\newcommand{\cL}{\mathcal{L}} 

\newcommand{\vgamma}{{\pmb{\gamma}}} %

\newcommand\scrM{\mathscr{M}} 

\newcommand{\SF}{\mathscr{S}} 

\newcommand{\scrU}{\mathscr{U}} %

\newcommand{\USF}{\mathscr{U}} 

\newcommand{\bL}{{\bf{L}}} 

\newcommand{\vb}{\mathrm{Vb}} 

\newcommand{\cP}{\mathcal{P}} 
\newcommand{\cC}{\mathcal{C}} 
\newcommand{\cZ}{\mathcal{Z}} 

\newcommand{\cE}{\mathcal{E}} 

\newcommand{\Log}{\operatorname{Log}} 
\newcommand{\cS}{\mathcal S} 
\newcommand{\cA}{\mathcal{A}} 

\newcommand{\cY}{\mathcal{Y}} 

\newcommand{\TT}{\mathbb{T}} 

\newcommand{\LL}{\mathbb{L}} 

\newcommand{\EE}{\mathbb{E}} 

\newcommand{\FF}{\mathbb{F}} 

\newcommand{\obs}{\operatorname{Obs}} 

\newcommand{\cok}{\operatorname{cok}} 

\newcommand{\bC}{{\mathbf C}} 

\newcommand{\bs}{\mathbf{s}} 

\newcommand{\bt}{\mathbf{t}} 

\newcommand{\fm}{\mathfrak{m}} 

\newcommand{\fb}{\mathfrak{b}} 

\newcommand{\fc}{\mathfrak{c}} 

\newcommand{\tf}{\tilde{f}} 

\newcommand{\cO}{\mathcal O} 
\newcommand{\cN}{\mathcal{N}} 

\newcommand{\cT}{\mathcal T} 

\newcommand{\cU}{\mathcal{U}}

\newcommand{\PP}{\mathbb{P}}

\newcommand{\Gm}{\mathbb{G}_{m}}

\newcommand{\GIT}{\mathbin{\mkern-3mu/\mkern-6mu/\mkern-3mu}}

\newcommand{\vir}{\mathrm{vir}}
\newcommand{\red}{\mathrm{red}}

\newcommand{\tw}{\mathrm{tw}}

\newcommand{\gd}{\mathbf{gd}} 

\newcommand{\spec}{\operatorname{Spec}} 

\DeclareMathOperator{\Aut}{Aut}

\DeclareMathOperator{\diff}{d}

\newcommand{\poleq}{\preccurlyeq} 

\newtheorem{proposition}{Proposition}[section]
\newtheorem{corollary}[proposition]{Corollary}
\newtheorem{lemma}[proposition]{Lemma}

\newtheorem{theorem}[proposition]{Theorem}
\theoremstyle{definition}
\newtheorem{assumption}[proposition]{Assumption}
\newtheorem{definition}[proposition]{Definition}
\theoremstyle{remark}
\newtheorem{remark}[proposition]{Remark}
\newtheorem{notation}[proposition]{Notation}
\newtheorem{example}[proposition]{Example}

\newcommand{\step}[2]{\smallskip\noindent{\bf Step #1: #2.}}

\begin{document}

\title[Towards Logarithmic GLSM]{Towards Logarithmic GLSM: The $r$-spin case}
\author{Qile Chen \and Felix Janda \and Yongbin Ruan \and Adrien Sauvaget}
\renewcommand{\shortauthors}{Q. Chen, F. Janda, Y. Ruan and A. Sauvaget}

\address[Q. Chen]{Department of Mathematics\\
Boston College\\
Chestnut Hill, MA 02467\\
U.S.A.}
\email{qile.chen@bc.edu}

\address[F. Janda]{Department of Mathematics\\
University of Notre Dame\\
Notre Dame, IN 46556\\
U.S.A.}
\email{fjanda@nd.edu}

\address[Y. Ruan]{Institute for Advanced Study in Mathematics\\
Zhejiang University\\
Hangzhou\\
China}
\email{yongbin.ruan@yahoo.com}

\address[A. Sauvaget]{CNRS, Universit\'e de Cergy--Pontoise \\
Laboratoire de Math\'ematiques AGM, UMR 8088 \\
95302 Cergy--Pontoise Cedex\\
France}
\email{adrien.sauvaget@math.cnrs.fr}

\date{\today}

\begin{abstract}
  In this article, we establish the logarithmic foundation for compactifying the moduli stacks of the gauged linear sigma model using
  stable log maps of \cite{Ch14, AC14, GS13}.
  We then illustrate our method via the key example of Witten's
  $r$-spin class to construct a proper moduli stack with a reduced perfect obstruction theory whose virtual cycle recovers the $r$-spin virtual cycle of Chang--Li--Li \cite{CLL15}. Indeed, our construction of the reduced virtual cycle is built upon the work of \cite{CLL15} by appropriately extending and modifying the Kiem-Li cosection \cite{KiLi13} along certain logarithmic boundary. In the subsequent article \cite{CJR19P}, we push the technique to a general situation. 
  
  One motivation of our construction is to fit the gauged linear sigma model in the broader setting of Gromov-Witten theory so that powerful tools such as virtual localization \cite{GrPa99} can be applied. A project \cite{CJR20P, CJR21P} along this line is currently in progress leading to applications including computing loci of holomorphic differentials \cite{CJRSZ20P}, and calculating higher genus Gromov--Witten
  invariants of quintic threefolds \cite{GJR17P, GJR18P}.
\end{abstract}



\maketitle

\setcounter{tocdepth}{1}
\tableofcontents

\section{Introduction}

\subsection{Gauged Linear Sigma models}

One of the major advances in the subject of Gromov--Witten theory is
the development of the so called FJRW-theory by the third author and
his collaborators.
The Gromov--Witten theory of a Calabi--Yau hypersurface of a weighted
projective space is conjectured to be equivalent to its FJRW-dual via
the LG/CY correspondence, a famous duality from physics.
In physics, the Gromov--Witten theory corresponds to a nonlinear sigma
model while FJRW-theory corresponds to a Landau--Ginzburg model.
Back in 1993, Witten gave a physical derivation of the LG/CY
correspondence by constructing a family of theories which was known as the 
\emph{gauged linear sigma model} or GLSM \cite{Wi93}.
By varying the parameters of GLSM, Witten argued that GLSM converges
to a nonlinear sigma model at a certain limit and a
Landau-Ginzburg orbifold at a different limit.
Hence, they are related by analytic continuation.

Several years ago, GLSM (with the restriction to compact-type
insertions) was put on a firm mathematical footing by Fan, Jarvis and
the third author \cite{FJR18}.
Let us briefly describe the construction.
The input data of a GLSM is an LG-space
\begin{equation*}
  W\colon V \GIT_{\theta} G\rightarrow \CC
\end{equation*}
for a GIT quotient $V \GIT_{\theta} G$ with a $\CC^*$-action
$\CC^*_R\curvearrowright V$ (called the R-charge) such that $W$ is
homogenous of degree one.
Moreover, we assume that the critical locus
$\mathrm{Crit}_W=\{\diff W=0\}\subset V \GIT_{\theta}G$ is compact.
The most famous  example is
\begin{equation*}
  W = p(x^5_1+x^5_2+x^5_3+x^5_4+x^5_5)\colon \CC^5\times \CC\rightarrow \CC
\end{equation*}
with $\CC^*$-action of weight $(1,1,1,1,1,-5)$.
Here $(x_1, x_2, x_3, x_4, x_5)$ are the coordinates of $\CC^5$ and $p$
is the coordinate of $\CC$.
Furthermore, the R-charge has weight $(0,0,0,0,0, 1)$.
The GIT-quotient $(\CC^5\times \CC)\GIT_{\theta} \CC^*$ has two chambers
or phases depending on the character
\begin{equation*}
  \theta(z)=z^n\colon \CC^*\rightarrow \CC^*.
\end{equation*}
If $\theta>0$ (i.e., $n>0$), then the unstable locus is
$(0,0,0,0,0)\times \CC$ and we have the GIT quotient
$((\CC^5-\{(0,0,0,0,0)\})\times \CC)\GIT_{\theta} \CC^*\cong
\cO_{\PP^4}(-5).$ When $\theta<0$, the unstable locus is
$\CC^5\times \{0\}$ and the GIT quotient is
$(\CC^5\times \CC^*)\GIT_{\theta}\CC^*\cong [\CC^5/\ZZ_5]$.
This GLSM is supposed to be equivalent to the Gromov--Witten theory of
the quintic 3-fold $X_5=\{ x^5_1+x^5_2+x^5_3+x^5_4+x^5_5=0\}$ in the
chamber $\theta>0$ and FJRW-theory of the LG orbifold
$$F = x^5_1+x^5_2+x^5_3+x^5_4+x^5_5\colon [\CC^5/\ZZ_5]\rightarrow \CC$$
in the chamber $\theta<0$.
Let us use this example to illustrate Fan--Jarvis--Ruan's algebraic
GLSM theory.
The geometric data for the above GLSM is
\begin{equation*}
  \scrM = \{(\cC, \cL, (s_1, s_2,s_3, s_4,s_5)\in H^0(\cL^{\oplus 5}), p\in H^0(\cL^{-5}\otimes \omega_{\log})): \dots\}
\end{equation*}
satisfying a certain stability condition where $\cC$ is a pre-stable
curve and $\cL$ is a line bundle over $\cC$.
For $\theta>0$, the stability condition implies that
$(s_1, s_2, s_3, s_4, s_5)$ define a stable quasimap into $\PP^4$ and
we obtain a variant of Chang--Li's $p$-field moduli space \cite{CL12}.
For $\theta<0$, the stability condition implies that the zeros of $p$
form an effective divisor $D$, and that $p$ defines a weighted 5-spin
structure $\cL^5\cong \omega_{\log,\cC}(-D)$.
In both cases, $\scrM$ is a DM-stack with two-term perfect obstruction
theory and has a virtual cycle in the Chow group.
However, it is not proper (compact).
To obtain a virtual cycle which we can integrate, (under the
assumption that all insertions are of compact-type%
\footnote{We refer the reader to \cite{KiLi18P, CKL18P} for
  interesting new developments related to the more general broad
  insertions along the cosection approach.}%
) we use $\diff W$ to define a cosection
$$\sigma\colon \obs_{\scrM} \rightarrow \cO_{\scrM}$$
and apply Kiem---Li's cosection localization technique \cite{KiLi13}
to define a localized virtual cycle $[\scrM]^\vir_\sigma$ with support
on the \emph{compact} sub-locus
$\scrM(\sigma)\subset \scrM$ satisfying the condition
$(s_1, s_2, s_3, s_4, s_5, p)\in Crit_W$.

The above construction is beautiful.
However, it is not directly useful for computational purposes.
In many ways, we would like to have an alternative construction which
is more friendly towards effective computation.
To that end, we would like to avoid using a cosection.

In the same paper, Kiem--Li showed that if $\scrM$ is a compact moduli
space with a two-term perfect obstruction theory and a cosection
$\sigma$, then
\begin{equation*}
  \deg([\scrM]^\vir) = \deg([\scrM]^\vir_\sigma)
\end{equation*}
This suggests that one should try to compactify the GLSM moduli space
$\scrM$ in a way that its cosection extends without
additional degeneracy loci.
The main purpose of this and its subsequent articles is to construct
such a compactification.

\subsection{The logarithmic approach}

\subsubsection{Stable maps relative to boundary divisors}

The theory of stable maps relative to a smooth boundary divisor was
introduced in symplectic geometry by Li--Ruan \cite{LR01} and
Ionel--Parker \cite{IP03, IP04} in the 90s.
The algebraic version using expansions was first developed
by Jun Li in his work \cite{Li01, Li02}, including a proof of a
degeneration formula in Gromov--Witten theory.
Since then, the degeneration formula has become one of the main tools in Gromov--Witten theory.
A combination of expansions with logarithmic geometry was introduced
by Kim \cite{Kim10}, and with orbifold structures was introduced by
Abramovich--Fantechi \cite{AF16}.

The idea of using logarithmic structures {\em without expansion} was first proposed by Bernd Siebert in 2001 \cite{Siebert}. This has led to the theory of stable log maps of Abramovich--Chen--Gross--Siebert \cite{AC14, Ch14, GS13} in the general logarithmic setting with toroidal boundary
divisors. A different approach using exploded manifolds was introduced by Brett
Parker \cite{Par11, Par12, Par15}.

In this and the subsequent articles \cite{CJR19P, CJR20P, CJR21P}, we will
apply the techniques of stable log maps to compactify the gauged
linear sigma model (GLSM) of Fan-Jarvis-Ruan \cite{FJR18}, and study
their virtual cycles.

\subsubsection{Log maps}

A \emph{stable log map} to a separated log Deligne--Mumford stack $Y$
is a morphism of log stacks $f \colon \cC \to Y$ over a log scheme $S$
where $\cC \to S$ is a \emph{twisted log curve} and the underlying
twisted map $\underline{f}$ obtained by removing log structures is
stable in the usual sense. For our purpose, we will only consider the
case that $\cM_Y$ is of \emph{Deligne--Faltings type of rank one}.
This amounts to saying that the logarithmic boundary of $Y$ is a
Cartier divisor, see Section \ref{sss:rank-one}.

The central object of log maps is the stack $\scrM(Y,\beta)$ parameterizing stable log maps to $Y$ with a given collection of discrete data $\beta$ (Section \ref{ss:log-moduli}).
The case where $Y$ is a log scheme has been developed in \cite{AC14, Ch14, GS13}.
The same method applies to the case of log Deligne--Mumford targets.
Due to a lack of references, in Section \ref{sec:logmap} we collect results of stable log maps with Deligne--Mumford targets needed in our construction.

\subsubsection{Modular principalization of the boundary}\label{sss:principalization}

A stable log map is \emph{degenerate} if it maps a component of the
source curve to the boundary of $Y$.
Denote by $\Delta \subset \scrM(Y,\beta)$ the locus consisting of
degenerate fibers.
In general, it is a virtual toroidal divisor, in the sense that it is the pullback of a Weil divisor from a log smooth stack $\fM(\cA,\beta')$ via a canonical morphism, see \eqref{equ:forgetgeometry}.
This turns out to be a major difficulty for the construction of a reduced perfect obstruction
theory of the compactified GLSM.
The key to overcoming this difficulty is the following \emph{modular
  principalization} of $\Delta$.

Let $f\colon \cC \to Y$ be a stable log map over a geometric log point
$S$.
For each irreducible component $Z \subset \cC$ we may associate an
unique element $e_{Z} \in \oM_{S} := \cM_S/\cO^*_S$ called the
\emph{degeneracy} of $Z$ (Section~\ref{sss:map-at-generic-pt}).
As elements of $\oM_{S}$, they carry a natural partial ordering such
that $e_{Z_1} \poleq e_{Z_2}$ iff $(e_{Z_2} - e_{Z_1}) \in \oM_S$.
Intuitively $e_Z$ measures the ``speed'' of $Z$ falling into the
boundary of $Y$, and $e_{Z_1} \poleq e_{Z_2}$ means that $Z_2$
degenerates ``faster'' than $e_{Z_1}$.
The stable log map $f$ is said to have \emph{uniform maximal
  degeneracy} if the set of degeneracies has a unique maximal element.
It turns out that having uniform maximal degeneracy is an open
condition and is stable under base change.
Let $\scrU(Y,\beta) \subset \scrM(Y,\beta)$ be the sub-category
fibered over log schemes consisting of objects with the uniform
maximal degeneracy.
In Section \ref{sec:UMD}, we establish the following:

\begin{theorem}[Theorem \ref{thm:max-moduli}] \label{thm:max-moduli-intro}
The canonical morphism $\scrU(Y,\beta) \to \scrM(Y,\beta)$ is a proper, representable and log \'etale morphism of log Deligne--Mumford stacks.
\end{theorem}

The maximal degeneracy defines naturally a \emph{virtual} Cartier
divisor $\Delta_{\max} \subset \scrU(Y,\beta)$ in the sense of \cite[\S 3]{KR18} whose support is
precisely the locus of degenerate log maps, see
Section~\ref{sss:boundary-torsor}.
To be more precise, $\Delta_{\max}$ is the vanishing locus of a global section of a line bundle $\bL_{\max}^{\vee}$.

\begin{remark}
  The category $\scrU(Y,\beta)$ is indeed the largest sub-category of
  $\scrM(Y,\beta)$ to which our construction of reduced perfect
  obstruction theory of compactified GLSM applies.
  Consequently, our construction applies to subcategories of
  $\scrU(Y,\beta)$ including the aligned logarithmic structures of
  \cite[Section~8.1]{ACFW13}.
  The general construction of this paper allows us to work with
  various subcategories of $\scrU(Y,\beta)$ to carry out the
  computation of the GLSM virtual cycle.
  This will be a task of \cite{CJR20P, CJR21P}.
\end{remark}

\subsection{The $r$-spin case}

Since the technique is relatively involved, for the reader's benefit it
makes sense to work out in full detail a first nontrivial simple
example.
This is another main purpose of the current article.
Our example of choice is the r-spin theory which corresponds to the
GLSM of
$$W=x^rp\colon [(\CC\times \CC)/\CC^*] \rightarrow \CC,$$
where the coordinates on $\CC \times \CC$ are $(x, p)$, the weight of
action is $(1, -r)$ and the $R$-charge is $(0,1)$.
Similarly to the case of quintic 3-folds, this model has two chambers
as well.
The relevant chamber for $r$-spin curve theory is the Landau--Ginzburg
chamber $\theta<0$, where the stable locus is $\CC\times \CC^*$.
Furthermore, we choose a stability condition such that $p$ has no
zero.
By the previous discussion, $p$ can be interpreted as defining an
isomorphism $\cL^r\cong \omega_{\log, \cC}$ and the GLSM moduli space
is
$$\USF^{\circ}_{g,k}=\{(\cC, \cL, s\in H^0(\cL), \cL^r\cong \omega_{\log, \cC})\}.$$

Let $(\cC/S, \cL)$ be an $r$-spin curve consisting of a log curve $\cC \to S$ and an $r$-spin bundle $\cL$ over $\ul{\cC}/\ul{S}$.
Denote by $0_{\cP}$ and $\infty_{\cP}$ the zero and infinity sections of $\PP := \PP(\cL\oplus\cO_{\ul{\cC}})$ respectively, and by $\cM_{\cP_{\infty}}$ the log structure on $\PP$ associated to $\infty_{\cP}$.
Consider the log stack $\cP = (\PP, \cM_{\cC}|_{\PP}\oplus_{\cO^*}\cM_{\cP_{\infty}})$ with the projection $\cP \to \cC$.
A {\em log field} is a section $f \colon \cC \to \cP$.
It is {\em stable} if $\omega^{\log}_{\cC/S}\otimes f^*\cO(0_{\cP})^{k}$ is positive for $k \gg 0$.

Denote by $\SF_{\beta}^{1/r}$ the stack of stable $r$-spin curves with a log field with discrete data $\beta = (g, \vgamma, \bf{c})$ consisting of the genus $g$, the monodromy $\vgamma$ of the spin bundle along markings, and the contact order $\bf{c}$ along each marking with $\infty_{\cP}$. We first achieve the compactification:

\begin{theorem}[Theorem \ref{thm:spin-fields-moduli}]
 $\SF_{\beta}^{1/r}$ is represented by a proper log Deligne--Mumford stack.
\end{theorem}

\begin{remark}
  The compactification of the moduli of abelian and meromorphic
  differentials using log stable maps has been studied previously in
  \cite{CC16, Gu16}.
  The compactification considered in this paper (in the case $r = 1$)
  is different from loc.\ cit.\ in that we do not put the log
  structure on $\cP$ induced by the zero section.
\end{remark}

\begin{remark}
It is worth emphasizing that the properness of  $\SF_{\beta}^{1/r}$ is interestingly a non-trivial fact. As shown in Section \ref{sss:properness-failure}, limit(s) of a one-parameter family of meromorphic sections of  spin bundles may not exist regardless of the stability conditions. Log structures play an important role in the existence of the underlying limiting section!
\end{remark}

Note that a log field $f \colon \cC \to \cP$ is equivalent to a log
map $f' \colon \cC \to (\PP, \cM_{\cP_{\infty}})$ whose underlying
morphism of schemes is a section of $\PP \to \ul{\cC}$.
Since $\cM_{\cP_{\infty}}$ is Deligne--Faltings type of rank one, we
may consider the stack $\USF_{\beta}^{1/r}$ of stable $r$-spin curve
with a log field with uniform maximal degeneracy with respect to
$\cM_{\cP_{\infty}}$.
Theorem \ref{thm:max-moduli-intro} implies that $\USF_{\beta}^{1/r}$
is a proper log Deligne--Mumford stack as well.

Next, we consider the virtual cycles.
The stack $\SF_{\beta}^{1/r}$ admits a canonical two term perfect
obstruction theory and hence a virtual cycle
$[\SF_{\beta}^{1/r}]^{\vir}$ which pulls back to the canonical perfect
obstruction theory and the virtual cycle $[\USF_{\beta}^{1/r}]^{\vir}$
of $\USF_{\beta}^{1/r}$.
But this canonical virtual cycle is different than the cosection
localized virtual cycle in general.
The main result of the paper is the following:

\begin{theorem}[Proposition \ref{prop:reduced-obs} and \ref{prop:comparison}]
  \label{thm:main}
  Under the condition that all markings are narrow and of trivial
  contact order, the space $\USF_{\beta}^{1/r}$ carries an alternative
  ``reduced'' two term perfect obstruction theory together with a
  cosection $\sigma^{\red}_{\USF/\fU}$ on $\USF_{\beta}^{1/r}$ that
  has no additional degeneracy loci.
  Furthermore, denote by $[\USF_{\beta}^{1/r}]^{\red}$ the virtual
  cycle of the reduced perfect obstruction theory, then
  \begin{equation*}
    i_*[\USF^{\circ}]^{\vir}_\sigma = [\USF_{\beta}^{1/r}]^{\red}
  \end{equation*}
  where
  $i\colon \overline{\scrM}^{1/r}_{g,\vgamma}\rightarrow
  \USF_{\beta}^{1/r}$ is the inclusion of the zero section, and
  $\USF^{\circ} = \USF_{\beta}^{1/r} \setminus \Delta_{\max}$.
\end{theorem}

\begin{remark}
  We remark that the reduced perfect obstruction theory has the same
  virtual dimension as the canonical one.
  Therefore, it is \emph{not} a traditional reduced virtual cycle,
  which changes the virtual dimension.
  Instead, the perfect obstruction theory is only ``reduced'' along
  the boundary $\Delta_{\max}$.
  In fact, the two perfect obstruction theories are related by a
  triangle \eqref{tri:reduce-X} where the third complex is determined
  by a virtual Cartier divisor supported along $\Delta_{\max}$.
\end{remark}

\begin{remark}
  In general, the canonical cosection of Chang--Li--Li \cite[(3.5)]{CLL15} over the open substack
  $\USF^{\circ} \subset \SF_{\beta}^{1/r}$ defining
  $[\USF^{\circ}]^{\vir}_\sigma$, extends to a meromorphic cosection
  over $\SF_{\beta}^{1/r}$.
  However, the behavior of this meromorphic cosection along the
  boundary $\SF_{\beta}^{1/r} \setminus \USF^{\circ}$, which is
  crucial for studying the virtual cycle, is hard to understand.
  The key to solve this issue is the modular principalization
  $\USF_{\beta}^{1/r} \to \SF_{\beta}^{1/r}$ of the boundary
  $\SF_{\beta}^{1/r} \setminus \USF^{\circ}$ in \S
  \ref{sss:principalization}, which endows $\Delta_{\max}$ with a
  virtual Cartier divisor structure in $\USF_{\beta}^{1/r}$.
  Consequently, we prove in \S \ref{ss:extending-cosection} that
  the canonical cosection over $\USF^{\circ}$  extends to
  $\USF_{\beta}^{1/r}$ with precisely $r$-th order poles along
  $\Delta_{\max}$ (\eqref{equ:relative-cosection} and Lemma
  \ref{lem:old-compatible}) but no additional degeneracy loci
  (Proposition \ref{prop:cosection-boundary-surjective}).
  These properties lead to the reduced theory in \S
  \ref{ss:cosection-factorization} and \S
  \ref{ss:relative-reduced-POT}.

  Comparing to the canonical theory, the reduced one carries a {\em reduced cosection} which extends the canonical one over $\USF^{\circ}$ across the boundary $\Delta_{\max}$ with neither additional poles nor additional degeneracy loci.
  This is what allows for comparing the virtual cycles
  $[\USF^{\circ}]^{\vir}_\sigma$ and $[\USF_{\beta}^{1/r}]^{\red}$ in
  \S \ref{ss:absolute-reduced-POT}. 
\end{remark}

\begin{remark}
As further shown in the subsequent paper \cite[Theorem~1.10]{CJR19P}, the two virtual cycles $[\USF_{\beta}^{1/r}]^{\vir}$ and $[\USF_{\beta}^{1/r}]^{\red}$ are related by $[\USF_{\beta}^{1/r}]^{\vir} = [\USF_{\beta}^{1/r}]^{\red} + r\cdot [\Delta_{\max}]^{\red}$, where $[\Delta_{\max}]^{\red}$ is a {\em reduced virtual cycle} of the boundary $\Delta_{\max}$. The cycle $[\Delta_{\max}]^{\red}$ is non-zero in general, and will be studied in detail in \cite{CJR20P, CJR21P}. The fact that the canonical virtual cycle $[\USF_{\beta}^{1/r}]^{\vir}$ does not equal the cosection localized virtual cycle $[\USF^{\circ}]^{\vir}_\sigma$ led us to search for the reduced theory.
\end{remark}

\begin{remark}
  The GLSM moduli space $\USF^{\circ}$, as well as its
  compactification $\USF_{\beta}^{1/r}$ admit $\CC^*$-actions induced
  by scaling the (log) field.
  Unfortunately, Chang--Kiem--Li's \cite{CKL17} localization theorem for
  cosection-localized virtual cycles does not apply to the cycle
  $[\USF^{\circ}]^{\vir}_\sigma$.
  This is why Theorem~\ref{thm:main} is significant, since it allows
  the application of Graber--Pandharipande's virtual localization
  formula \cite{GrPa99}, thus leading to a new way for understanding
  Witten's $r$-spin class (even before push-forward to the moduli
  space of curves).
  We leave the discussion of the $\CC^*$-action, equivariance of
  obstruction theories, etc. (in the more general setting of
  \cite{CJR19P}) to the future work \cite{CJR20P}.
  The localization formula has been a main motivation of this work
  (see also Section~\ref{ss:effective-rspin}).
\end{remark}

\subsection{History of the $r$-spin virtual cycle}

There was a long line of works constructing both the moduli space of
$r$-spin structures and its virtual cycle.
Spin curves were proposed by Witten \cite{Wi93} in an effort to
generalize his famous conjecture that the intersection theory of the
moduli space of stable curves is governed by the KdV-hierarchy.
The compactification was first constructed by Jarvis \cite{Ja98} using
torsion-free sheaves and later by Abramovich--Jarvis \cite{AbJa03}
using line bundles on twisted curves.

The first construction of the virtual cycle is due to
Polishchuk--Vaintrob \cite{PoVa01}.
From the modern point of view, their construction is better viewed as
a quantum K-theoretic construction from which one can obtain a virtual
cycle by taking some kind of Chern character (see~\cite{Ch08}).

The picture was clarified significantly by Fan--Jarvis--Ruan with a
vast generalization (FJRW-theory) of $r$-spin theory.
The input data of FJRW theory is a non-degenerate quasi-homogeneous
polynomial $W$ together with a so called \emph{admissible} finite
automorphism group $G$ of $W$.
The $r$-spin theories are simply the case of $W=z^r$ and
$G=\ZZ/r \ZZ$.
The state space of the $r$-spin theory corresponds to the monodromy at
the marked point, and is indexed by an integer $0\leq m<r$.
The insertion $m>0$ corresponds to the so called \emph{narrow} sector
in FJRW-theory and the corresponding virtual cycle was constructed as
a localized topological Euler class.
The role of $m=0$ was clarified in general FJRW-theory as a new type
of insertions called \emph{broad}. They showed that broad
insertions are irrelevant in $r$-spin theory but a source of difficulty
in general case.
Fan--Jarvis--Ruan's construction is analytic in nature although there
is an algebraic construction of Polishchuk and Vaintrob using matrix
factorizations \cite{PoVa16}.
However, it is not clear that these two are equivalent in the most
general case.

The last piece of the puzzle before the present work was provided by
Chang--Li--Li in \cite{CLL15}, where they gave yet another algebraic
geometric construction of FJRW virtual cycle for narrow sectors.
This is the construction that we use in this article.
Furthermore, they proved that all constructions of
Polishchuk--Vaintrob, Chiodo, Fan--Jarvis--Ruan and Chang--Li--Li are
equivalent.

Finally, the $A_r$-generalization of Witten's integrable hierarchies
conjecture was proved by Faber--Shadrin--Zvonkine \cite{FSZ10} while
the $D_n, E_{6,7,8}$-generalization was proved by Fan--Jarvis--Ruan
\cite{FJR13}.

\subsection{Effective $r$-spin structures}
\label{ss:effective-rspin}

A key input that led us to propose this new construction of the
$r$-spin virtual cycle is a conjectural formula of
$[\overline{\scrM}^{1/r}_{g,n}]^\vir$ by the second author.
This formula was motivated by the recent study of the cycle of the
locus of holomorphic differentials and of double ramification cycles.
We outline here this train of thought.

We consider the open sub-stack
$\scrM_{g,\vgamma}^{1/r}\subset\overline{\scrM}_{g,\vgamma}^{1/r}$ of
$r$-spin structures on smooth orbifold curves.
An $r$-spin structure $(\cC/S,\cL)\in \scrM_{g,\vgamma}^{1/r}$ is
called \emph{effective} if $h^0(\cL)>0$.
We denote by $S_0\subset \scrM_{g,\vgamma}^{1/r}$ the locus of
effective $r$-spin structures and by
$\overline{S}_0 \subset \overline{\scrM}_{g,\vgamma}^{1/r}$ its
closure.
A.~Polishchuk studied the geometry of effective $r$-spin structures
(see~\cite{Po06}) and asked the following question: Can we express the
$r$-spin virtual cycle to $\scrM_{g,\vgamma}^{1/r}$ in terms of the
cycle $[\overline{S}_0]$ and other natural cycles?

This problem was left aside until a precise  conjecture was recently stated (see~\cite[Conjecture~A.1]{PPZ16P}). This conjecture can be re-stated as follows:
for large values of  $r$, we have
\begin{equation*}\label{conj:holodiff}
\epsilon_*\left( \frac{1}{r}[\overline{\scrM}_{g,\vgamma}^{1/r}]^\vir + [\overline{S}_0] \right)  = \alpha(r)\in A^*(\overline{\scrM}_{g,n})
\end{equation*}
where $\alpha(r)$ is a polynomial in $r$ (here
 $\epsilon\colon \overline{\scrM}_{g,\vgamma}^{1/r}\to \overline{\scrM}_{g,n}$ stands for the forgetful map of the spin structure).

\begin{remark}
  The conventions for the value of
  $[\overline{\scrM}_{g,\vgamma}^{1/r}]^{\vir}$ are different
  in~\cite{CLL15},~\cite{Po06}, and~\cite{PPZ16P}.
\end{remark}

This conjecture is very similar to a conjectural expression by Pixton
for the double ramification (DR) cycles that was proved by
Pandharipande, Pixton, Zvonkine, and the second author
(see~\cite{PPZ16P}).
The main tool of their proof is the virtual localization formula of
Graber and Pandharipande (see~\cite{GrPa99}).

In order to prove the new conjecture of~\cite{PPZ16P}, the second
author built a (conjectural) localization formula by analogy with the
proof of the expression of DR cycles.
In this conjectural localization formula, the role of DR cycles is
replaced by cycles of effective $r$-spin structures.
The second author checked the consistency of this formula by various
computations in low genera.

From this point, our main problem was to construct the space where the
conjectural localization formula should hold.
The effort to pin down the geometry underlining this formula led to
use the machinery of log geometry in this article.
In work in progress \cite{CJR20P, CJR21P}, the first three authors, will prove
a general localization formula for log GLSM, and in \cite{CJRSZ20P},
with Pandharipande, Pixton, Schmitt and Zvonkine, we will show that it
implies \cite[Conjecture~A.1]{PPZ16P}.

\subsection{Plan}

The paper is organized as follows.
In Section~\ref{sec:logmap}, we discuss the general set-up of log
stable maps in the orbifold setting.
In Section~\ref{sec:UMD}, we introduce the new notion of log
structures of ``uniform maximal degeneracy'', which is crucial for the
construction of the reduced virtual cycle.
This is applied in Section~\ref{sec:logfields}, to construct the
compactification of the moduli space of $r$-spin curves with a
field. Finally, in Section~\ref{sec:virtual}, we construct the reduced
perfect obstruction theories and cosections, and we prove
Theorem~\ref{thm:main}.

\subsection*{Acknowledgement}

The first author was partially supported by NSF grant DMS 1560830 and
DMS 1700682.
The second author was partially supported by an AMS Simons Travel
Grant and NSF grants DMS-1901748 and DMS-1638352..
The third author was partially supported by NSF grant DMS 1405245 and
NSF FRG grant DMS 1159265.

The second author is indebted to Dimitri Zvonkine for many discussions
at the Institut de Mathématiques de Jussieu that have directly led to
a conjectural formula for Witten's class, whose proof was the main
motivation for starting the current project.
Discussions at the ``RTG conference on Witten's class and related
topics'' at the University of Michigan formed the start of this
collaboration.
The second and third authors would like to thank Shuai Guo for the
collaboration which provided motivation for the current work. 
The authors would like to thank Rahul Pandharipande for the continuous
support and the wonderful ``Workshop on higher genus'' at ETH Zürich
where our entire program was presented at the first time.
We would also like to thank Dawei Chen, Alessandro Chiodo, J\'er\'emy
Gu\'er\'e, Davesh Maulik, Jonathan Wise and Dimitri Zvonkine for
useful discussions.
Finally, the second and third authors would like to thank MSRI for the
hospitality where the paper was finished during a visit supported by
NSF grant DMS-1440140.

\section{Moduli of twisted stable log maps}
\label{sec:logmap}

In this section, we introduce the setup of stable log maps needed for
compactifying GLSM.
It was defined with prestable source curves in \cite{AC14, Ch14,
  GS13}.
We take the opportunity to extend it to the orbifold setting.


\subsection{Twisted log maps}

\subsubsection{Twisted curves}\label{sss:twisted-curve}

Recall from \cite[Definition~4.1.2]{AV02} that a {\em twisted
  $n$-pointed curve of genus $g$} over $\ul{S}$ consists of
the data
\[
(\ul{\cC} \to \ul{C} \to \ul{S}, \{\sigma_i\}_{i=1}^n)
\]
where
\begin{enumerate}
 \item $\ul{\cC}$ is a proper Deligne--Mumford stack, and is \'etale locally a nodal curve over $\ul{S}$;
 \item $\sigma_i \subset \ul{\cC}$ are disjoint closed substacks in the smooth locus of $\ul{\cC} \to \ul{S}$;
 \item $\sigma_i \to \ul{S}$ are \'etale gerbes banded by the multiplicative group $\mu_{r_i}$ for some non-negative integer $r_i$;
 \item the morphism $\ul{\cC} \to \ul{C}$ is the coarse moduli morphism;
 \item along each stacky singular locus of $\ul{\cC} \to \ul{S}$, the group action of $\mu_{r_i}$ is balanced;
 \item $\ul{\cC} \to \ul{C}$ is an isomorphism over $\ul{\cC}_{gen}$, where $\ul{\cC}_{gen}$ is the complement of the markings $\sigma_i$ and the stacky singular locus of $\ul{\cC} \to \ul{S}$.
\end{enumerate}

Given a twisted curve as above, by \cite[Proposition~4.1.1]{AV02} the
coarse space $\ul{C} \to \ul{S}$ is an $n$-pointed, genus $g$ ordinary
pre-stable curve over $\ul{S}$ with the markings determined by the
images of $\{\sigma_i\}$.
When there is no danger of confusion, we simply write
$\ul{\cC} \to \ul{S}$ for a family of twisted curves.

Twisted curves can only have stacky structure along markings and
nodes.
Though the stacky structures can be described equivalently in terms of
log structures as in \cite{Ol07}, to be compatible with the existing
literature on r-spin curves, we will recall their local structures
below following \cite{AV02}.

\subsubsection{Stacky structure along nodes}
Let $\ul{\cC} \to \ul{C} \to \ul{S}$ be a family of twisted curves,
and ${\bar q} \to \ul{C}$ be a geometric point of a node.
Shrinking $\ul{S}$ if necessary, there exists an \'etale neighborhood
$\ul{U} \to \ul{C}$ of $q$ with an \'etale morphism
\[
\ul{U} \to \spec \big( \cO_{\ul{S}}[x,y]/(xy = t)\big)
\]
for some $t \in \cO_{\ul{S}}$. The pullback $\ul{\cC}\times_{\ul{C}}\ul{U}$ is the stack quotient
\begin{equation}\label{equ:node-local}
[\spec\big( \cO_{\ul{U}}[\tilde{x},\tilde{w}]/(\tilde{x}\tilde{y} = t', \tilde{x}^r = x, \tilde{y}^{r} = y) \big)/\mu_{r}]
\end{equation}
for some $t' \in \cO_{\ul{S}}$. Here for a generator $\gamma \in \mu_{r}$, the $\mu_{r}$-action is given by $\gamma(\tilde{x}) = \zeta \tilde{x}$ and $\gamma(\tilde{y}) = \zeta' \tilde{y}$ for some primitive $r$-th roots of unity $\zeta$ and $\zeta'$. The {\em balanced} condition implies that $\zeta' = \zeta^{-1}$.

\subsubsection{Stacky structure along markings}
Let $p \to \ul{C}$ be a geometric point of a marking corresponding to
$\sigma_i$.
Shrinking $\ul{S}$ if necessary, there exists an \'etale neighborhood
$\ul{U} \to \ul{C}$ of $p$ with an \'etale morphism
\[
\ul{U} \to \spec \cO_{\ul{S}}[z].
\]
The pullback $\ul{\cC}\times_{\ul{C}}\ul{U}$ is the stack quotient
\begin{equation}\label{equ:marking-local}
\big[\spec\big( \cO_{\ul{U}}[\tilde{z}]/(\tilde{z}^{r_i} = z) \big) \big/ \mu_{r_i} \big],
\end{equation}
where for each $\zeta \in \mu_{r_i}$ the action is given by $\tilde{z} \mapsto \zeta\tilde{z}$.

\subsubsection{Logarithmic twisted curves}\label{sss:log-twisted-curve}
A {\em log twisted $n$-pointed curve} over a fine and saturated log scheme $S$ in the sense of \cite[Definition~1.7]{Ol07} consists of
\[
(\pi\colon \cC \to C \to S, \{\sigma_i\}_{i=1}^n)
\]
such that
\begin{enumerate}
\item The underlying data $(\ul{\cC} \to \ul{C} \to \ul{S}, \{\sigma_i\}_{i=1}^n)$ is a twisted $n$-pointed curve over $\ul{S}$.
\item $\pi$ is a log smooth and integral morphism of fine and saturated log stacks.
\item If $\ul{U} \subset \ul{\cC}$ is the non-critical locus of $\ul{\pi}$, then $\oM_{\cC}|_{\ul{U}} \cong \pi^*\oM_{S}\oplus\bigoplus_{i=1}^{n}\NN_{\sigma_i}$ where $\NN_{\sigma_i}$ is the constant sheaf over $\sigma_i$ with fiber $\NN$.
\end{enumerate}

We remark that the log structure $\cM_{\cC}$ along each marking has a component given by the divisorial log structure associated to the marking. This corresponds to the component $\NN_{\sigma_i}$ above. However, the stacky structure is allowed to be trivial along markings.

For simplicity, we may refer to $\pi\colon \cC \to S$ as a log twisted curve.
The {\em pullback} of a log twisted curve $\pi\colon \cC \to S$ along an arbitrary morphism of fine and saturated log schemes $T \to S$ is the log twisted curve $\pi_T\colon \cC_T:= \cC\times_S T \to T$.

\subsubsection{The combinatorial structure of log twisted curves}\label{curvecombinatorial}

Consider a log twisted curve
$(\pi\colon \cC \to S, \{\sigma_i\}_{i=1}^n)$, a geometric point
$p\to \ul{\cC}$ and its image
$s=\ul{\pi}(p) \in \ul{S}$.
The morphism
$\bar{\pi}^{\flat}\colon \ul{\pi}^{*}\oM_{S} \to \oM_{\cC}$ of sheaves
of monoids can be described on the level of stalks as for classical
log curves by
\[
\bar{\pi}_{p}^{\flat} \colon \ul{\pi}^{*} \oM_{S, s} \to \oM_{\cC, p} \simeq \left\{
\begin{array}{c l}
\ul{\pi}^*\oM_{S,s}\oplus\NN, & \text{if $\ul{p}$ is a marked point}\\
\ul{\pi}^*\oM_{S,s}\oplus_{\NN} \NN^2, & \text{if $\ul{p}$ is a node}\\
\ul{\pi}^*\oM_{S,s}, & \text{otherwise},
\end{array} \right.
\]
where $\bar{\pi}_{p}^{\flat}$ is the inclusion of the first factor.
Recall that at the $i$-th marking the factor $\NN$ is generated by
generator of the divisorial log-structure associated to $\sigma_i$,
while at a node, the direct sum is determined by
\begin{equation}\label{node-characteristic}
\NN \to \oM_{S,s}, \ \ \ 1 \mapsto \rho_q,
\end{equation}
and the diagonal map $\NN \to \NN^{2}$.
Indeed, the diagonal map is induced by the relation
$t' = \tilde{x}\tilde{y}$ in the local chart (\ref{equ:node-local}).
The two generators $(1,0), (0,1) \in \NN^2$ correspond to the local
coordinates $\tilde{x}, \tilde{y}$ of the two branches of the node,
and $\rho_{q}$ corresponds to the local section $t'$.

\subsubsection{The stack of log twisted curves}\label{sss:curve-stack}

Denote by $\fM^{\tw}_{g,n}$ the category of genus $g$ log twisted
curves with $n$ marked points over the category of log schemes.
By \cite[Theorem~1.9]{Ol07}, the fibered category $\fM^{\tw}_{g,n}$ is
represented by a log algebraic stack.
Indeed, the underlying stack $\ul{\fM}^{\tw}_{g,n}$ is the stack
parameterizing twisted curves with the same discrete data.
The boundary of $\ul{\fM}^{\tw}_{g,n}$ parameterizing singular fibers
is a normal crossings divisor whose associated divisorial log
structure defines the log structure of $\fM^{\tw}_{g,n}$.

\subsubsection{Log stable maps with twisted source curves}
We fix a log algebraic stack $Y$ as the target.

\begin{definition}
A {\em log map} to $Y$ over a fine and saturated log scheme $S$ consists of the data
\[
(\pi\colon \cC \to S, f\colon \cC \to Y)
\]
where $\cC \to S$ is a log twisted curve over $S$, and $f$ is a
morphism of log stacks.
The \emph{pullback} of a log map along an arbitrary morphism of log
schemes is defined via the pullback of log twisted curves as usual.

When $Y$ is a separated log Deligne--Mumford stack, a log map is
\emph{stable} if the underlying twisted map is stable in the usual
sense.
In particular, a stable log map is representable.
\end{definition}
For simplicity, we may write $f\colon \cC \to Y$ for a log map.

\subsubsection{Deligne--Faltings targets of rank one}\label{sss:rank-one}
Throughout this paper, we will mainly focus on the following type of targets.

\begin{definition}
  A log algebraic stack $Y$ is {\em Deligne--Faltings type of rank one}
  if there is a morphism of sheaves of monoids $\NN_{Y} \to \oM_Y$
  which locally lifts to a chart (in the sense of \cite[Definition~2.9 (1)]{KKato})
  of $\cM_Y$. Here $\NN_{Y}$ denotes the constant sheaf over $Y$ with
  fiber $\NN$.
\end{definition}

Consider the log algebraic stack  $\cA$ with the underlying stack $$\big[\spec(k[\NN])/\spec(k[\ZZ]) \big]$$ and the log structure induced by the affine toric variety $\spec(k[\NN])$. Let $\infty_{\cA}\subset \cA$ be the boundary divisor associated to the log structure $\cM_\cA$. The log stack $\cA$ has the universal property that if $Y$ is  Deligne--Faltings type of rank one, then there is a canonical strict morphism $Y \to \cA$.

\subsection{The combinatorial structure of twisted log maps}\label{ss:log-combinatorial}
The combinatorial structure of log maps with twisted source curves is
similar to the case without twists as in \cite{GS13,AC14, Ch14}.
We introduce it following \cite[Section~2.3]{ACGS17}.
For our purposes, we assume $Y$ is Deligne--Faltings type of rank one.

\subsubsection{The induced morphism of sheaves of monoids}
Let $(\pi\colon \cC \to S, f\colon \cC \to Y)$ be a log map over $S$. First consider the case where $\ul{S}$ is a geometric point with $\oM_{S} = Q$. Denote by $\cM := \ul{f}^*\cM_{Y}$. Thus, $\cM$ is a Deligne--Faltings log structure on $\ul{\cC}$ of rank one. This leads to a pair of morphisms of sheaves of monoids
\[
(\bar{\pi}^{\flat}\colon Q \to \oM_{\cC}, \bar{f}^{\flat}\colon \oM \to \oM_{\cC}).
\]
where we view $Q$ as the constant sheaf of monoids on $\ul{\cC}$. The morphism $\bar{\pi}^{\flat}$ is described in Section \ref{curvecombinatorial}. We describe the behavior of $\bar{f}^{\flat}$ at generic points, marked points, and nodes of $\ul{\cC}$ as follows.

\subsubsection{The stalks of $\oM$}
Since $\cM$ is Deligne--Faltings type of rank one, for any point $s \to \ul{\cC}$ the sheaf $\oM_{s}$ is a constant sheaf of monoids with fiber either $\NN$ or the trivial one $\{0\}$.

\subsubsection{The structure of $\bar{f}^{\flat}$ at generic points}\label{sss:map-at-generic-pt}

If $s = \eta$ is a generic point of an irreducible component
$Z \subset \ul{\cC}$, then we have a local morphism of monoids
$\bar{f}^{\flat}_{\eta}\colon \oM_{\eta} \to Q$.\footnote{A morphism
  of monoids $h\colon P \to Q$ is {\em local} if
  $h^{-1}(Q^{\times}) = P^{\times}$.}

If $\oM_{\eta} = \NN$, then we call $Z$ a {\em degenerate component},
and $e_{Z} := \bar{f}^{\flat}_{\eta}(1) \in Q$ the {\em degeneracy} of
$f$ along $Z$.

If $\oM_{\eta} = \{0\}$, then we call $Z$ a {\em non-degenerate
  component}, and set the {\em degeneracy} of $Z$ to be
$e_{Z} = 0 \in Q$.

\subsubsection{The structure of $\bar{f}^{\flat}$ at marked points}\label{sss:map-at-marking}
If $s = p$ is a point lying on the marking $\sigma_i$, 
then we have a local morphism of monoids $\bar{f}^{\flat}_{p}\colon \oM_{p} \to Q\oplus\NN$. Consider the composition
 \[
 c_p\colon \oM_{p} \stackrel{\bar{f}^{\flat}_{p}}{\longrightarrow} Q\oplus\NN \stackrel{pr_2}{\longrightarrow} \NN
 \]
If $\oM_{p} = \NN$, the morphism $c_{p}$ is determined by $c_{p}(1) \in \NN$. We call $c_p$ or equivalently $c_{p}(1)$ the {\em contact order} at $p$. The marked point $p$ has the {\em trivial contact order} if $c_p(1) = 0$.

Let $\eta$ be the generic point of the component $Z$ containing $p$,
and assume that $Z$ is degenerate.
Since the generization morphism $\chi_{\eta,p}\colon Q\oplus\NN \to Q$
(see \cite[Lemma~3.5~iii]{Ol03}) is just the projection to the first
factor, we obtain
\[
\bar{f}^{\flat}_p \colon \NN \to Q\oplus\NN, \ \ \ 1 \mapsto e_Z + c_{p}(1)\cdot(0,1).
\]

\subsubsection{The structure of $\bar{f}^{\flat}$ at nodal points}\label{sss:nodecombinatorial}
Suppose $s = q \to \ul{\cC}$ is a nodal point contained in the closures of two generic points $\eta_1, \eta_2$ of the two branches meeting at $q$. Using the description of nodes in Section \ref{curvecombinatorial}, we have a local morphism
\[\bar{f}^{\flat}_{q}\colon \oM_{q} \to Q\oplus_{\NN}\NN^2.\]
Let $(1,0), (0,1) \in \NN^2$ correspond to the two local coordinates
around $q$ of the two branches of $\eta_1$ and $\eta_2$ respectively.

If $\oM_{q} = \NN$, after possibly renaming the branches at $q$, we
may assume that
\begin{equation}\label{nodecombinatorial}
\bar{f}^{\flat}_{q}(1) = e + c_{q}\cdot (1,0)
\end{equation}
for some $c_q \in \NN$ and $e \in Q$. We call $c_{q}$ the {\em contact order} of the node $q$. Observe the commutative diagram
\begin{equation}\label{nodegeneralization}
\xymatrix{
\oM_{q} \ar[r]^{\bar{f}^{\flat}_q} \ar[d] & \oM_{\cC, q} \ar[d] \\
\oM_{\eta_i} \ar[r]^{\bar{f}^{\flat}_{\eta_i}} & \oM_{\cC, \eta_i}.
}
\end{equation}
where the vertical arrows are the generization morphisms.
Applying the commutativity of the above diagram with $i=1$ to (\ref{nodecombinatorial}), we obtain that
\begin{equation}\label{nodedegeneracy}
\bar{f}^{\flat}_{q}(1) = e_{Z_1} + c_{q}\cdot (1,0)
\end{equation}
where $e_{Z_i}$ is the degeneracy of the component $Z_i$ containing $\eta_i$. Using $i=2$, we have
\begin{equation}\label{nodalequation}
e_{Z_1} + c_{q}\rho_q = e_{Z_2}.
\end{equation}
This is the nodal equation as in \cite[(3.3.2)]{Ch14}.

If $\oM_{q} = \{0\}$, then $\bar{f}^{\flat}_{q}$ is necessarily trivial and $c_q = 0$.
Since the commutativity of \eqref{nodegeneralization} holds in this
case as well, taking generization, we obtain
$e_{Z_1} = e_{Z_2} = 0$. In particular, Equation \eqref{nodalequation}
holds for all nodes.

\subsubsection{The natural partial ordering}\label{sss:partial-order}
For a twisted curve $\ul{\cC}$ over a geometric point, recall that its {\em dual intersection graph} $\ul{G}$ consisting of the set of vertices $V(\ul{G})$ corresponding to irreducible components, the set of edges $E(\ul{G})$ corresponding to nodes, and the set of half-edges $L(\ul{G})$ corresponding to marked points.

Let $q \to \ul{\cC}$ be a node joining two irreducible components $Z_1, Z_2$. Using \eqref{nodalequation}, we introduce the partial ordering $\poleq$ as follows:
\begin{enumerate}
 \item If $c_q > 0$, we write $v_1 \poleq v_2$.
 \item If $c_q = 0$, we write  $v_1 \poleq v_2$ and $v_2 \poleq v_1$, or equivalently $v_1 \sim v_2$.
\end{enumerate}
Then $\poleq$ extends to a partial order on the set $V(\ul{G})$, called the {\em minimal partial order}.

The minimal partial order yields an orientation of $\ul{G}$ as
follows.
Let $l \in E(\ul{G})$ be the corresponding edge joining vertices
$v_1, v_2$ associated to $Z_1, Z_2$ respectively.
The edge $l$ is oriented from $v_1$ to $v_2$ if $v_1 \poleq v_2$, and
the edge is oriented both ways if $v_1 \sim v_2$.
We remark that such $\ul{G}$ contain no one direction oriented loops,
see \cite[Corollary 3.3.7]{Ch14}.

If $c_q > 0$, we say that $q$ is an \emph{incoming node of $Z_1$} or
an \emph{outgoing node of $Z_2$}.
When $c_q = 0$, the node is neither an incoming nor outcoming component
of any component.
The \emph{incoming special points} of a component $Z \subset \cC$ are
all incoming special points, the \emph{outgoing special points} are
the markings on $Z$ and all outgoing nodes.

\subsubsection{The logarithmic combinatorial type}\label{logcombinatorialtype}
We introduce the {\em log combinatorial type} of the log map $(\cC \to S, f\colon \cC \to Y)$ over a geometric point $\ul{S}$ following \cite[Section~3.4]{Ch14} and \cite[Section~4.1.1]{AC14}:
\begin{equation}\label{equ:combinatorial-type}
G = \big(\ul{G},  V(G) = V^{n}(G) \cup V^{d}(G), \poleq, (c_i)_{i\in L(G)}, (c_l)_{l\in E(G)} \big)
\end{equation}
where
\begin{enumerate}[(a)]
 \item $\ul{G}$ is the dual intersection graph of the underlying curve $\ul{\cC}$.

 \item $V^{n}(G) \cup V^{d}(G)$ is a partition of $V(G)$ where $V^{d}(G)$ consists of vertices of degenerate components.

 \item $\poleq$ is the minimal partial order defined in Section \ref{sss:partial-order}.

 \item Associate to a leg $i\in L(G)$ the contact order $c_i \in \NN$ of the corresponding marking $\sigma_i$.

 \item Associate to an edge $l\in E(G)$ the contact order $c_l \in \NN$ of the corresponding node.
\end{enumerate}

\begin{remark}
Our definition of log combinatorial types is similar to the definition of types in \cite[Definition~1.10]{GS13} and \cite[Section~2.3.7]{ACGS17}. Since we work with Deligne--Faltings type targets, we are able to include more combinatorial information such as the partition and partial order on $G$.
\end{remark}

These combinatorial data behave well under generization:

\begin{proposition}\label{prop:combinatorial-generalization}
Let $f\colon \cC \to Y$ be a log map over an arbitrary log scheme $S$. Then
\begin{enumerate}
 \item The contact order $c_i$ along the $i$th marking $\sigma_i$ is a constant over each connected component of $S$.
 \item Let $W \subset \cC$ be a connected locus of nodes in $\cC$. Then the contact order of the nodes is constant along $W$.
\end{enumerate}
\end{proposition}
\begin{proof}
The proof is identical to the case of \cite[Lemma 3.2.4, 3.2.9]{Ch14}.
\end{proof}

\subsection{Minimality}\label{sss:minimality}

\subsubsection{The monoid}\label{sss:min-monoid}
We recall the construction of minimal monoids in \cite{Ch14, AC14, GS13}. Consider a log map $(\cC \to S, f\colon \cC \to Y)$ over a geometric point $\ul{S}$ with the log combinatorial type $G$. We introduce a variable $\rho_l$ for each edge $l \in E(G)$, and a variable $e_v$ for each vertex $v \in V(G)$. Denote by $h_l$ the relation
$ e_{v'} = e_v + c_l\cdot\rho_l
$
for each edge $l$ with the two ends $v \poleq v'$ and contact order $c_l$. Denote by $h_v$ the relation
$
e_v = 0
$
for each $v \in V^{n}(G)$. Consider the abelian group
\[
\mathcal{G} = \big(\bigoplus_{v \in V(G)} \ZZ e_v \bigoplus_{l \in E(G)} \ZZ \rho_l \big) \big/ \langle h_v, h_l \ | \ v\in V^{n}(G), \ l \in E(G) \rangle
\]
Let $\mathcal{G}^{t} \subset \mathcal{G}$ be the torsion subgroup. Consider the composition
\[
\big(\bigoplus_{v \in V(G)} \NN e_v \bigoplus_{l \in E(G)} \NN \rho_l\big) \to \mathcal{G} \to \mathcal{G}/\mathcal{G}^{t}
\]
Let $\oM(G)$ be the smallest submonoid that is saturated in $\mathcal{G}/\mathcal{G}^{t}$, and contains the image of the above composition. We call $\oM(G)$ the {\em minimal monoid} associated to $G$, or associated to the log map.\footnote{The monoid $\oM(G)$ is called the {\em basic monoid} in \cite{GS13}.}

\begin{proposition}\label{prop:minimality}
There is a canonical map of monoids $\phi\colon \oM(G) \to \oM_S$ induced by sending $e_v$ to the degeneracy of the component associated to $v$, and sending $\rho_l$ to the element $\rho_{q}$ as in Equation (\ref{node-characteristic}) associated to $l$. In particular, the monoid $\oM(G)$ is fine, saturated, and sharp.
\end{proposition}
\begin{proof}
This follows from the proof of \cite[Proposition 3.4.2]{Ch14}.
\end{proof}

For later use, we observe the following.

\begin{corollary}\label{cor:separate-non-distinct-log}
  There is a unique monoid $\oM(G)'$ such that
  $\oM(G) = \oM(G)' \oplus \NN^{d}$ where $d$ is the number of edges
  in $E(G)$ with trivial contact orders.
  In particular, the image of $e_{v}$ is contained in $\oM(G)'$ for
  all $v \in V(G)$.
\end{corollary}
\begin{proof}
  When $c_{l} = 0$, the element $\rho_{l}$ is not involved in the relation $h_{l}$. The collection of such $\rho_l$ generates the factor $\NN^{d}$.
\end{proof}

\subsubsection{Minimal objects}

As in \cite{GS13, Ch14, AC14}, we define the minimal objects using the canonical morphism $\phi$.

\begin{definition}
  A log map $(\cC \to S, f\colon \cC \to Y)$ over $S$ is called
  \emph{minimal} or \emph{basic}\footnote{The terminology used in \cite{GS13} is {\em
      basic}.} if for each of its geometric fibers, the induced
  canonical morphism in Proposition \ref{prop:minimality} is an
  isomorphism.
\end{definition}

The definition is justified by the openness of minimality.

\begin{proposition}\label{prop:minimal-open}
For any family of log maps $(\cC \to S, f\colon \cC \to Y)$ over a log scheme $S$, if the fiber $f_s\colon \cC_s \to Y$ over a geometric point $s \to S$ is minimal, then there is an \'etale neighborhood $U \to S$ of $s$ such that the fiber $f_{U}\colon \cC_{U} \to Y$ is minimal.
\end{proposition}
\begin{proof}
This follows from the proof of \cite[Proposition 3.5.2]{Ch14} and \cite[Proposition 1.22]{GS13}.
\end{proof}

Minimal objects have the following universal property which is the key to the construction of the moduli stack.
\begin{proposition}\label{prop:minimal-universal}
For any log map $f\colon \cC \to Y$ over a log scheme $S$, there exists a minimal log map $f_m\colon \cC_m \to Y$ over $S_m$ and a morphism of log schemes $\Phi\colon S \to S_m$ such that
\begin{enumerate}
 \item The underlying morphism $\ul{\Phi}$ is an isomorphism.
 \item $f\colon \cC \to Y$ is the pullback of $f_m\colon \cC_m \to Y$ along $\Phi$.
\end{enumerate}
Furthermore, the pair $(f_m, \Phi)$ is unique up to a unique isomorphism.
\end{proposition}
\begin{proof}
The proof is identical to the situation of log maps with no orbifold twists on the source curves. We refer to \cite[Proposition 1.24]{GS13} and \cite[Proposition 4.1.1]{Ch14} for details.
\end{proof}

\subsubsection{Finiteness of automorphisms}\label{sss:finite-auto}
Let $f\colon \cC \to Y$ be a log map over $S$ with $\ul{S}$ a geometric point. An {\em automorphism} of a stable log map is a pair $(\psi\colon \cC \to \cC, \theta\colon S \to S)$ of compatible automorphisms of log schemes such that $\psi\circ f = f$. Denote by $\Aut(f)$ the automorphism group of the log map $f$, and by $\Aut(\ul{f})$ the automorphism group of the corresponding underlying map. We have the following property:

\begin{proposition}\label{prop:finite-auto}
Suppose the log map $f\colon \cC \to Y$ over $S$ is stable and minimal. Then the natural group morphism $\Aut(f) \to \Aut(\ul{f})$ is injective. In particular, the group $\Aut(f)$ is finite.
\end{proposition}
\begin{proof}
The proof is identical to the case of \cite[Proposition 1.25]{GS13} and \cite[Lemma 3.8.3]{Ch14}.
\end{proof}

\subsection{The stacks of twisted log maps}\label{ss:log-moduli}

Fix a separated log Deligne--Mumford stack $Y$ as the target with
$\cM_Y$ of Deligne--Faltings type of rank one.
Consider the {\em discrete data}
\begin{equation}\label{discretedata}
  \beta = (g,n, {\bf c} = \{c_i\}_{i=1}^{n}, A)
\end{equation}
for twisted log maps in $Y$ where $g$ is the genus, $n$ is the number of markings,
$c_i$ is the contact order of the $i$-th marking, and $A \in H_2(\ul{Y})$ is a curve class.

Let $\beta' = (g,n, \bf{c})$ be the {\em reduced discrete data} obtained by removing the curve class, and $\ul{\beta} = (g,n, A)$ the {\em underlying discrete data} by removing the contact orders.

Denote by $\scrM(Y, \beta)$ the category of stable log maps to $Y$ with the discrete data $\beta$ fibered over the category of log schemes, and $\scrM(\ul{Y}, \ul{\beta})$ the stack of usual twisted stable maps to $\ul{Y}$. For our purposes, we view $\scrM(\ul{Y}, \ul{\beta})$ as a log stack equipped with the canonical log structure given by its universal curves. Composing with the forgetful morphism $Y \to \ul{Y}$, we obtain a canonical morphism
\begin{equation}\label{equ:forgetlog}
\scrM(Y,\beta) \to \scrM(\ul{Y}, \ul{\beta}).
\end{equation}

\begin{theorem}\label{thm:algebraicity}
The morphism \eqref{equ:forgetlog} is representable by log Deligne--Mumford stacks locally of finite type.
\end{theorem}

The above theorem has been established when both domain curves and the target are schemes \cite{Ch14, GS13}, and the same method applies in the orbifold case as well.
However for later use, we will follow the universal target strategy of \cite{AW, Wi16} below.

For any log map $f\colon \cC \to Y$ over $W$, the composition $\cC \to Y \to \cA$ is a log map  to $\cA$ over $W$, where $Y \to \cA$ is the canonical strict morphism.
Denote by $\fM(\cA,\beta')$ the category of log maps to $\cA$ with the reduced discrete data $\beta'$. The above composition defines a canonical morphism
\begin{equation}\label{equ:forgetgeometry}
\scrM(Y,\beta) \to \fM(\cA,\beta').
\end{equation}

On the other hand, consider the stack $\fM_{g,n}(\ul{\cA})$
parameterizing (not necessarily representable) usual maps to
$\ul{\cA}$ from genus $g$, $n$-marked log twisted curves.
It is an algebraic stack locally of finite type by \cite[Theorem
1.2]{HR14}.
We further view $\fM_{g,n}(\ul{\cA})$ as a log stack equipped with the
canonical log structure induced by its universal twisted curve.

\begin{proposition}\label{prop:uni-min-stack}
The canonical morphism
\[
\fM(\cA,\beta') \to \fM_{g,n}(\ul{\cA})
\]
induced by the forgetful morphism $\cA \to \ul{\cA}$ is representable by log Deligne--Mumford stacks locally of finite type.
In particular, the fibered category $\fM(\cA,\beta')$ is representable by log algebraic stacks locally of finite type.
\end{proposition}

\begin{proof}
The proof is identical to the case of \cite[Corollary 1.1.1]{Wi16}. 
\end{proof}

\begin{proof}[Proof of Theorem \ref{thm:algebraicity}]
The underlying map $\ul{Y} \to \ul{\cA}$ of $Y \to \cA$ induces a strict morphism of log stacks
\[
\scrM(\ul{Y}, \ul{\beta}) \to \fM_{g,n}(\ul{\cA}),
\]
where both stacks are equipped with the canonical log structures from their universal curves. The two morphisms (\ref{equ:forgetlog}) and (\ref{equ:forgetgeometry}) induce
\[
\scrM(Y,\beta) \to \scrM(\ul{Y},\ul{\beta})\times_{\fM_{g,n}(\ul{\cA})}\fM(\cA,\beta'),
\]
where the fiber product is in the fine and saturated category. The above morphism is an isomorphism. Indeed, the datum of a log map to $Y$ is equivalent to the datum of an underlying map to $\ul{Y}$ and a  log map to $\cA$ with compatible compositions to $\ul{\cA}$. Thus, the algebraicity of Theorem \ref{thm:algebraicity} follows from Proposition \ref{prop:uni-min-stack}. The Deligne--Mumford property is a consequence of Proposition \ref{prop:finite-auto}.
\end{proof}

The following log smoothness result will be used later.

\begin{proposition}\label{prop:uni-stack-smooth}
The tautological morphism
\[
\fM(\cA,\beta') \to \fM^{\tw}_{g,n}
\]
by taking the source log curves, is log \'etale.
In particular, the stack $\fM(\cA,\beta')$ is log smooth and
equi-dimensional.
\end{proposition}
\begin{proof}
This is identical to the proof of \cite[Proposition 3.2]{AW}.
\end{proof}

\subsection{Relative boundedness of twisted log maps}
The boundedness of stable log maps without orbifold structures has been proved in \cite{Ch14, AC14, GS13} under certain assumptions, and in \cite[Theorem~1.1.1]{ACMW17} in full generality by reducing to the case of \cite{AC14}.
For our purposes, we will only consider the Deligne--Faltings case of rank one in the orbifold situation.

Consider the forgetful morphism of log algebraic stacks
\[
{\bf F}\colon \fM(Y,\beta) \to \fM(\ul{Y},\ul{\beta})
\]
where $\fM(\ul{Y},\ul{\beta})$ has the canonical log structure from its universal curve. For each strict morphism $W \to \fM(\ul{Y},\beta)$, consider the projection
$${\bf F}_W\colon \fM(Y,\beta)_W := \fM(Y,\beta) \times_{\fM(\ul{Y},\ul{\beta})} W \to W$$

\begin{definition}\label{def:combinatorially-finite}
For a strict morphism $W \to \fM(\ul{Y},\beta)$, the discrete data $\beta$ is called {\em combinatorially finite over $W$} if the collection of log combinatorial types of log maps over $\fM(Y,\beta)_W$ is finite.
\end{definition}

\begin{remark}
  If $W = \scrM(\ul{Y}, \ul{\beta})$, then
  $\fM(Y,\beta)_W = \scrM(Y,\beta)$.
  Thus the above definition is compatible with the combinatorial
  finiteness of \cite[Definition 3.3]{GS13}.
\end{remark}

\begin{proposition}\label{prop:boundedness}
Suppose $\beta$ is combinatorially finite over $W$ for a strict morphism $W \to \fM(\ul{Y},\ul\beta)$. Then ${\bf F}_W$ is of finite type.
\end{proposition}

\begin{proof}
  This follows from the same proof as in \cite[Section 5.4]{Ch14} or \cite[Section 3.2]{GS13}.
\end{proof}

\subsection{The relative weak valuative criterion}
We fix a discrete valuation ring $R$ with the maximal ideal $\fm$ and the residue field $R/\fm$. Let $K$ be the quotient field of $R$. We have the following version of valuative criterion necessary for properness.

\begin{proposition}\label{prop:log-map-valuative}
Consider a commutative diagram of solid arrows of underlying stacks
\begin{equation}\label{diag:valuative}
\xymatrix{
\spec K \ar[rr] \ar[d] && \ul{\fM(Y,\beta)} \ar[d]^{\ul{\bf F}} \\
\spec R \ar[rr] \ar@{-->}[rru] && \ul{\fM(\ul{Y},\ul{\beta})}
}
\end{equation}
Possibly after replacing $R$ by a finite extension of DVRs, and $K$ by
the induced extension of the quotient field, there exists a dashed
arrow making the above diagram commutative.
Furthermore, such a dashed arrow is unique up to a unique isomorphism.
\end{proposition}
\begin{proof}
This follows from the same proof as in \cite[Section 6]{Ch14} and \cite[Section 6]{GS13}.
Indeed, the bottom arrow of \eqref{diag:valuative} provides a family of underlying pre-stable maps over $\spec R$.
It remains to construct the extension on the level of log structures. 

The first step is to extend the log combinatorial type to the closed fiber.
This can be done identically as in \cite[Section 6.2]{Ch14} or \cite[Section 4.1]{GS13} by studying \'etale locally on the source curve.
The second step is to construct a log curve over $S$ with $\ul{S} = \spec R$.
This step can be carried out identically as in \cite[Section 4.2]{GS13}, since it only uses the complement of markings and nodes, and orbifold structures play no role.
Finally, the morphism between log structures of the curve and target can be constructed identically as in \cite[Section 4.3]{GS13} and \cite[Section 6.3]{Ch14} by first constructing the log map \'etale locally on the curve, then gluing them using the canonicity of the local construction.
\end{proof}

\section{Stable log maps with uniform maximal degeneracy}
\label{sec:UMD}


In this section, we introduce a configuration of log structures which
is the key to the construction of the reduced perfect obstruction
theory, and subsequently Witten's r-spin class.

We again fix  the target $Y$ with the log structure $\cM_Y$ of rank one Deligne--Faltings type.

\subsection{Uniform maximal degeneracy}

\subsubsection{Maximal degeneracies}\label{sss:max-deg}
Consider a log map $f\colon \cC \to Y$ over a geometric log point $S$.
Denote by $G$ the log combinatorial type of $f$, and by $\oM(G)$ the minimal monoid.
Let $\phi\colon\oM(G) \to \oM_S$ be the canonical morphism as in  Proposition \ref{prop:minimality}.

Consider the natural partial order $\poleq_{\oM_S}$ on $\oM_S$ such
that $e_1 \poleq e_2$ iff $(e_2 - e_1) \in \oM_S$. The partial order
$\poleq_{\oM_S}$ induces a refinement of $\poleq$ of $G$ in the sense
that $v_1 \poleq v_2$ in $V(G)$ implies $e_{v_1} \poleq e_{v_2}$ in
$\oM_S$.
When $\phi$ is an isomorphism, this refinement is the trivial
refinement.

\begin{definition}
  A degeneracy $\phi(e_v) \in \oM_S$ is called \emph{maximal} if
  $\phi(e_v)$ is maximal in the set of all degeneracies
  under $\poleq_{\oM_S}$.
  The corresponding vertex $v \in V(G)$ is called a \emph{maximally
    degenerate vertex} of $f$.
\end{definition}

As $\poleq_{\oM_S}$ is a partial order, there could be more than one maximal degeneracy in $\oM_S$. On the other hand, different vertices are allowed to have the same degeneracy in $\oM_S$.

\begin{definition}\label{def:uniform-deg}
  The log map $f\colon \cC \to Y$ over $S$ is said to have \emph{uniform maximal degeneracy}
  if the set of degeneracies has a maximum under $\poleq_{\oM_S}$.
  A family of log maps is said to have \emph{uniform maximal
    degeneracy} if each geometric fiber has uniform maximal
  degeneracy.
\end{definition}

Since the set $V(G)$ is finite, the maximal degeneracy, if it exists, has to be the degeneracy of some vertex. The above definition for families is justified by the following.

\begin{proposition}\label{prop:um-open}
For any family of log maps $f\colon \cC \to Y$ over a log scheme $S$, if the fiber $f_s\colon\cC_s \to Y$ over a geometric point $s \to S$ has uniform maximal degeneracy, then there is an open neighborhood $U \subset S$ of $s$ such that the pullback family $f_{U}\colon \cC_{U} \to Y$ over $U$ has uniform maximal degeneracy.
\end{proposition}

Proposition \ref{prop:um-open} can be checked \'etale locally, and
follows immediately from Lemma \ref{lem:generize-degeneracy} and
\ref{lem:generize-po} below.

\subsubsection{Generization of degeneracies and partial orders}
Consider a pre-stable log map $f\colon \cC \to Y$ over a log scheme
$S$ together with a chart $h\colon \oM_{S,s} \to \cM_S$ where
$s \to S$ is a geometric point.
Such a chart always exists after possibly passing to an \'etale cover.
Here $f$ does not necessarily have uniform maximal degeneracy.
The chart $h$ allows us to view any $e \in \oM_{S,s}$ as a section of $\oM_S$ via the composition $\oM_{S,s} \to \cM_S \to \oM_S$. This section $e$ can be then specialized to any geometric point $t \in S$ with the fiber denoted by $e_{t} \in \oM_{S,t}$. Let $G$ the log combinatorial type of $f_s$.

\begin{lemma}\label{lem:generize-degeneracy}
With notation as above, suppose $e \in \oM_{S,s}$ is the degeneracy of $v \in V(G)$. Then there is an \'etale neighborhood $U \to S$ of $s$ such that for any geometric point $t \in U$, the fiber $e_t \in \oM_{S,t}$ is a degeneracy.
\end{lemma}
\begin{proof}
Shrinking $S$ if necessary, we may choose a section $\sigma\colon \ul{S} \to \ul{\cC}$ such that $\sigma(\ul{S})$ is contained in the smooth non-marked locus of $\cC \to S$, and intersects the component of $\cC_s$ corresponding to $v$. Consider the pullback morphism
\[
\sigma^*(\bar{f}^{\flat})\colon (\sigma \circ f)^*\oM_Y \to \sigma^*\oM_{\cC} = \oM_S.
\]
The equality on the right hand side follows from the assumption that $\sigma(\ul{S})$ avoids all nodes and markings.

Since $\oM_{Y}$ is of Deligne--Faltings type of rank one, we may choose a morphism $\NN \to \oM_Y$ which locally lifts to a chart. Denote again by $1 \in \oM_Y$ the image of $1 \in \NN$ via this morphism. By the discussion in Section \ref{sss:map-at-generic-pt}, the fiber of the image $\sigma^*(\bar{f}^{\flat})(1)_t \in \oM_{S,t}$ over each geometric point $t \in S$ is the degeneracy of the component of $\cC_t$ intersecting $\sigma(\ul{S})$. In particular, we have $\sigma^*(\bar{f}^{\flat})(1)_s = e$.
\end{proof}

Conversely, every degeneracy of a nearby fiber is the generization of some degeneracy from the central fiber:

\begin{corollary}
With notation as above, there is an \'etale neighborhood $U \to S$ of $s$ such that for any geometric point $t \in U$ and any degeneracy $e' \in \oM_{S,t}$, there is a degeneracy $e \in \oM_{S,s}$ such that $e_t = e'$.
\end{corollary}
\begin{proof}
  With notation as in the proof of Lemma \ref{lem:generize-degeneracy}, we
  may further shrink $S$ and choose a finite set of extra markings
  $\{\sigma_i \to \ul{\cC}\}$ avoiding nodes and the original
  markings, whose union $\cup \sigma_i(\ul{S})$ intersects each
  irreducible component of each geometric fiber of $\cC \to S$.
\end{proof}

The partial order $\poleq_{\oM_{S,t}}$ is well-behaved under generization:

\begin{lemma}\label{lem:generize-po}
With notation as above, consider a pair of elements $e_1, e_2 \in \oM_{S,s}$ with $e_1 \poleq_{\oM_{S,s}} e_2$. Then we have $e_{1,t} \poleq_{\oM_{S,t}} e_{2,t}$ in $\oM_{S,t}$ for any geometric point $t \to S$.
\end{lemma}
\begin{proof}
By assumption, we have $(e_2 - e_1) \in \oM_{S,s}$, hence
$$(e_2 - e_1)_t = (e_{2,t} - e_{1,t}) \in \oM_{S,t}.$$
\end{proof}

\begin{corollary}\label{cor:max-deg-exists}
  Suppose $f\colon \cC \to Y$ is a family of log maps over $S$ with
  uniform maximal degeneracy.
  Then there is a global section $e_{\max} \in \Gamma(S, \oM_{S})$,
  which restricts to the maximal degeneracy over each geometric fiber
  over $S$. 
\end{corollary}
\begin{proof}
  Lemma \ref{lem:generize-degeneracy} and Lemma \ref{lem:generize-po}
  imply that the maximal degeneracy over each geometric fiber glues to
  the global section $e_{\max}$.
\end{proof}

\begin{example}
  \label{ex:UMD}
  Let $\ul\cC$ be a twisted curve over a geometric point, which we
  make into a log curve $\cC$ by adding divisorial log structures at
  the markings.
  Then, the projection $f\colon \cC \times \cA \to \cA$ is an example
  of a family of log maps over $S = \cA$ with uniform maximal
  degeneracy.
  Indeed, over each of the two geometric points of $S$, all components
  of $\cC$ have the same degeneracy -- over the unique closed point
  $0_\cA$, all are degenerate, and over $\cA \setminus \{0_\cA\}$, all
  are non-degenerate.
  In this example, $e_{\max}$ is the unique generator of
  $\Gamma(S, \oM_S) = \Gamma(\cA, \oM_\cA) = \NN$.
\end{example}

\subsubsection{The category of log maps with uniform maximal degeneracy}\label{sss:um-category}

Let $Y$ be a log stack of rank one Deligne--Faltings type.
We introduce the fibered category $\fU(Y,\beta')$ of pre-stable log
maps to $Y$ with uniform maximal degeneracy and reduced discrete data
$\beta'$ over the category of fine and saturated log schemes.
If furthermore $Y$ is a separated log Deligne--Mumford stack, denote
by $\scrU(Y,\beta) \subset \fU(Y,\beta')$ the sub-category of stable
log maps with discrete data $\beta$ as in \eqref{discretedata}.

By the universality as in Proposition \ref{prop:minimal-universal}, there are tautological morphisms of fibered categories as inclusions of subcategories:
\begin{equation}\label{equ:forget-max}
\scrU(Y,\beta) \to \scrM(Y,\beta) \ \ \mbox{and} \ \ \fU(Y,\beta') \to \fM(Y,\beta').
\end{equation}
We next introduce the minimality of the subcategory $\fU(Y,\beta')$.

\subsection{Minimality with uniform maximal degeneracy}

\subsubsection{Log combinatorial type with uniform maximal degeneracy}\label{sss:log-type-um}
Let $f\colon \cC \to Y$ be a pre-stable log map over $S$ with uniform maximal degeneracy. First assume that $\ul{S}$ is a geometric point.

Let $G$ be the log combinatorial type of $f$, and
$\phi\colon \oM(G) \to \oM_S$ the canonical morphism.
Denote by $V_{\max} \subset V(G)$ the subset of vertices having the
maximal degeneracy in $\oM_S$.
We call $(G, V_{\max})$ the {\em log combinatorial type with uniform
  maximal degeneracy}.

\subsubsection{Minimal monoids with uniform maximal degeneracy}\label{sss:um-monoid}
Consider the torsion-free abelian group
\[
\big( \oM(G)^{gp}\big/ \sim \big)^{tf}
\]
where $\sim$ is given by the relations $(e_{v_1} - e_{v_2}) = 0$ for any $v_1, v_2 \in V_{\max}$. By abuse of notation, we may use $e_v$ for the image of the degeneracy of the vertex $v$ in $\big( \oM(G)^{gp}\big/ \sim \big)^{tf}$. Thus, for any $v \in V_{\max}$ their degeneracies in $\big( \oM(G)^{gp}\big/ \sim \big)^{tf}$ are identical, denoted by $e_{\max}$.

Let $\oM(G,V_{\max})$ be the saturated submonoid in $\big( \oM(G)^{gp}\big/ \sim \big)^{tf}$ generated by
\begin{enumerate}
\item the image of $\oM(G) \to \big( \oM(G)^{gp}\big/ \sim \big)^{tf}$, and
\item the elements $(e_{\max} - e_v)$ for any $v \in V(G)$.
\end{enumerate}
By the above construction, we obtain a natural morphism of monoids $\oM(G) \to \oM(G,V_{\max})$. On the other hand, we have a canonical morphism of monoids $\phi\colon \oM(G) \to \oM_S$ by Proposition \ref{prop:minimality}. Putting these together, we observe the following canonical factorization:

\begin{proposition}\label{prop:minimize-umd}
There is a canonical morphism of monoids
\begin{equation}\label{equ:minimize-umd}
\phi_{\max}\colon \oM(G,V_{\max}) \to \oM_S.
\end{equation}
such that the morphism $\phi\colon \oM(G) \to \oM_S$ factors through $\phi_{\max}$.
\end{proposition}

\begin{corollary}\label{cor:separate-non-distinct-log-U}
  There is a canonical splitting
  \begin{equation*}
    \oM(G, V_{\max}) = \oM(G,V_{\max})' \oplus \NN^{d}
  \end{equation*}
  where $d$ is the number of edges in $E(G)$ whose contact order is
  zero.
  Furthermore, the image of $e_{v}$ is contained in
  $\oM(G)'$ for all $v \in V(G)$.
\end{corollary}
\begin{proof}
This follows directly from Corollary \ref{cor:separate-non-distinct-log} and the construction of $\oM(G, V_{\max})$.
\end{proof}

\begin{definition}
We call $\oM(G,V_{\max})$ the {\em minimal monoid with uniform maximal degeneracy} associated to $(G, V_{\max})$, or simply the {\em minimal monoid} associated to $(G,V_{\max})$.
\end{definition}

\begin{definition}\label{def:umd-minimal}
  A stable log map $f\colon \cC \to Y$ over $S$ with $\ul{S}$ a
  geometric point is called \emph{minimal with uniform maximal
    degeneracy} if \eqref{equ:minimize-umd} is an isomorphism.
  A family of log maps is called \emph{minimal with uniform maximal
    degeneracy} if each of its geometric fibers is so.
\end{definition}

\subsubsection{Openness of minimality with uniform maximal degeneracy}

The definition of minimal objects in families with uniform maximal degeneracy  is justified by the following analogue of Proposition \ref{prop:minimal-open}:

\begin{proposition}\label{prop:min-umd-open}
For any family of log maps $f\colon \cC \to Y$ over a log scheme $S$, if the fiber $f_{s}\colon\cC_{s} \to Y$ over a geometric point $s \to S$ is minimal with uniform maximal degeneracy, then there is an open neighborhood $U \subset S$ of $s$ such that the family $f_{U}\colon \cC_{U} \to Y$ is minimal with uniform maximal degeneracy.
\end{proposition}
\begin{proof}
Since the statement can be checked \'etale locally on $S$, by Proposition \ref{prop:um-open}, replacing $S$ by an \'etale neighborhood of $s$, we may assume that $f\colon \cC \to Y$ over $S$ has uniform maximal degeneracy. For each geometric point $t\in S$, denote by $(G_t, V_{\max,t})$ the log combinatorial type of the fiber $f_t\colon \cC_t \to Y$ over $t$, see Section \ref{sss:log-type-um}. 

Let $f_m\colon \cC_m \to Y$ over $S_m$ be the associated minimal objects as in Proposition \ref{prop:minimal-universal} such that $f$ is the pullback of $f_m$ along a morphism $S \to S_m$. Shrinking $\ul{S}$ if necessary, we choose two charts $\oM_{S,s} \to \cM_{S}$ and $\oM_{S_m, s} \to \cM_{S_m}$. We view elements of $\oM_{S,s}$ and $\oM_{S_m,s}$ as global sections of $\oM_{S}$ and $\oM_{S_m}$ via the respective compositions:
\[
\oM_{S,s} \to \cM_{S} \to \oM_S \ \ \ \mbox{and} \ \ \ \oM_{S_m, s} \to \cM_{S_m} \to \oM_{S_m}.
\]

For each geometric point $t \in S$ we have a commutative diagram of solid arrows
\[
\xymatrix{
\oM_{S_m, s} \ar[r] \ar[d] & \oM_{S_m,t} \ar[d] \\
\oM(G_s, V_{\max,s}) \ar@{=}[d] \ar@{-->}[r]^{\chi} & \oM(G_t, V_{\max,t}) \ar[d]^{\phi_{\max,t}} \\
\oM_{S,s} \ar[r]^{\chi_{s,t}} & \oM_{S,t}.
}
\]
where the top and bottom horizontal arrows are the generization morphisms given by the two charts above, the compositions of the vertical arrows are given by the morphism $S \to S_m$, and the factorization through $\oM(G_t, V_{\max,t})$ follows from Proposition \ref{prop:minimize-umd}. By the construction in Section \ref{sss:um-monoid}, the arrow on the top induces the dashed arrow $\chi$ making the above diagram commutative.
Indeed, to see the commutativity of the lower square, observe that the maximal degeneracy $e_{\max,s}$ at $s$ specializes to the maximal degeneracy $e_{\max, t}$ at $t$. Thus, the relations $\sim$ and elements of the form $(e_{\max,s} - e_{v,s})$ as in Section \ref{sss:um-monoid}  over $s$ generize to the corresponding relations and elements over $t$, which leads to the factorization through $\chi$.

First observe that the lower commutative square in the above diagram implies that $\phi_{\max,t}$ is surjective.
Indeed, the groupification of the generization morphism $\oM_{S,s}^{gp} \to \oM_{S,t}^{gp}$ is surjective. Since it factors through $\oM(G_t,V_{\max,t})^{gp}$, the morphism $\phi_{\max,t}^{gp}$ is also surjective. Furthermore, $\oM_{S,t}$ is the saturation of the submonoid in $\oM_{S,t}^{gp}$ generated by the image of $\oM_{S,s}$ which is precisely the image $\phi_{\max,t}(\oM(G_t,V_{\max,t}))$.

To see that $\phi_{\max,t}$ is injective, it remains to prove the injectivity of $\phi_{\max,t}^{gp}$. Consider the set
\begin{equation}\label{equ:log-kernel}
F = \{e \in \oM_{S,s} \ | \ \chi_{s,t}(e) = 0\}.
\end{equation}
By \cite[Lemma 3.5]{Ol03}, the group $F^{gp}$ is the kernel of the morphism $\oM_{S,s}^{gp} \to \oM_{S,t}^{gp}$. Let $K$ be the kernel of $\oM(G_s, V_{\max,s})^{gp} \to \oM(G_t,V_{\max,t})^{gp}$, hence $K \subset F^{gp}$. We will prove $F^{gp} = K$ by showing that the composition $F \hookrightarrow \oM(G_s, V_{\max,s}) \stackrel{\chi}{\to} \oM(G_t, V_{\max,t})$ is trivial.

Indeed, consider the fine submonoid $\oN \subset \oM(G_s, V_{\max,s})^{gp}$ generated by the degeneracy $e_v$ for each $v \in V(G_s)$, the element $\rho_l$ for each $l\in E(G)$, and the element $e_{\max} - e_v$ for each $v \in V(G)$. Let $e \in \oM(G_s, V_{\max,s})$ be one of the above three types. Observe that $\chi(e) = 0$ if $\chi_{s,t}(e) = 0$ by the construction in Section \ref{sss:um-monoid}, and hence $\chi(\oN\cap F) = 0$.
Since $\oM(G_s, V_{\max,s})$ is the saturation of $\oN$ in $\oM(G_s, V_{\max,s})^{gp}$, $F$ is the saturation of $\oN\cap F$. We conclude that $\chi(F) = 0$.
\end{proof}

\begin{remark}
The proof in Proposition \ref{prop:min-umd-open} indeed proves that the log structure minimal in the sense of Definition \ref{def:umd-minimal} is coherent \cite[(2.1)]{KKato}. As shown in \cite[Theorem B.2]{Wi16}, the coherence is a sufficient condition for the openness of minimality in general.
\end{remark}

\subsubsection{The universality}

The minimal objects in $\fU(Y,\beta')$ have a universal property
similar to the case of Proposition \ref{prop:minimal-universal}:

\begin{proposition}\label{prop:min-umd-universal}
For any log map $f\colon \cC \to Y$ over a log scheme $S$ with uniform maximal degeneracy, there exists a log map $f_{mu}\colon \cC_{mu} \to Y$ over $S_{mu}$ which is minimal with uniform maximal degeneracy, and a morphism of log schemes $\Phi_{u}\colon S \to S_{mu}$ such that
\begin{enumerate}
 \item The underlying morphism $\ul{\Phi}_u$ is an isomorphism.
 \item $f\colon \cC \to Y$ is the pullback of $f_{mu}\colon \cC_{mu} \to Y$ along $\Phi_u$.
\end{enumerate}
Furthermore, the pair $(f_{mu}, \Phi_u)$ is unique up to a unique isomorphism.
\end{proposition}
\begin{proof}
Let $f_m\colon \cC_m \to Y$ over $S_m$ be the associated minimal object as in Proposition \ref{prop:minimal-universal}, so that $f$ is the pullback of $f_m$ along $\Phi\colon S \to S_m$ with $\ul{\Phi}$ the identity of $\ul{S}$.

Since the statement is local on $S$, we are free to shrink $S$ if needed. Thus, we may assume there are charts
\[
h_{S_m}\colon \oM_{S_m,s} \to \cM_{S_m} \ \ \mbox{and} \ \ h_S\colon \oM_{S,s} \to \cM_{S}
\]
for some geometric point $s \to S$.
Denote by $(G, V_{\max})$ the log combinatorial type of the fiber
$f_s$ over $s$.
By Proposition~\ref{prop:minimize-umd}, the morphism
$\phi\colon \oM(G) = \oM_{S_m,s} \to \oM_{S,s}$ factors through
$\phi_{\max}\colon Q:=\oM(G,V_{\max}) \to \oM_{S,s}$.
Write $\tilde{\phi}\colon \oM(G) \to Q$ for the canonical morphism.

Denote by $\cM_{S_{mu}}$ the log structure on $\ul{S}$ associated to
the pre-log structure defined by
$h\colon Q \to \oM_{S,s} \stackrel{h_S}{\to} \cM_{S}$.
Thus, there is a morphism of log structures
$\cM_{S_{mu}} = Q\oplus_{h^{-1}\cO^*_{\ul{S}}}\cO^*_{\ul{S}} \to
\cM_S$.
Then the following assignments on the right define a unique dashed
arrow on the left which makes the diagram of log structures
commutative:
\[
\xymatrix{
 & \cM_{S_m} \ar@{-->}[ld] \ar[rd] &   &  &h_{S_m}(e) \ar@{|-->}[ld] \ar@{|->}[rd] & \\
\cM_{S_{mu}} \ar[rr] && \cM_{S}      &  h \circ \tilde{\phi}(e) + v \ar@{|->}[rr] && h_S\circ \phi(e) + u
}
\]
Here $u \in \cO^*$ and $v \in \cO^*$ are the unique, invertible
sections making the diagram commutative.
This defines a morphism of log schemes
$S_{mu} := (\ul{S}, \cM_{S_{mu}}) \to S_m$ through which $S \to S_m$
factors.
Further observe that such a morphism depends on the choice of charts
$h_S$ and $h_{S_m}$.
However, different choices of charts induce a unique isomorphism of
$S_{mu}$ compatible with the arrows to and from $S_m$ and $S$
respectively.

Pulling back the log map over $S_m$, we obtain a log map
$f_{mu}\colon \cC \to Y$ over $S_{mu}$ which further pulls back to $f$
over $S$.
Note that the geometric fiber $f_{mu,s}$ is minimal with uniform
maximal degeneracy over $s$.
Further shrinking $\ul{S}$ and using
Proposition~\ref{prop:min-umd-open}, we obtain a family of log maps
over $S_{mu}$ minimal with uniform maximal degeneracy as needed.
\end{proof}

\subsubsection{Finiteness of automorphisms}

Consider a log map $f\colon \cC \to Y$ over $S$ with $\ul{S}$ a geometric point. Suppose $f$ is minimal with uniform maximal degeneracy. Let $f_m\colon \cC \to Y$ over $S_m$ be the minimal log map given by Proposition \ref{prop:finite-auto} such that $f$ is the pullback of $f_m$ along a morphism $\Phi\colon S \to S_m$. Let $\Aut(f)$ and $\Aut(f_m)$ be the automorphism groups introduced in Section \ref{sss:finite-auto}. They are related as follows:

\begin{proposition}\label{prop:finite-auto-umd}
With notation as above, there is an injective homomorphism of groups $\Aut(f) \to \Aut(f_m)$. In particular, $\Aut(f)$ is finite if $f$ is stable.
\end{proposition}
\begin{proof}
We first construct this group homomorphism. Consider an element $(\psi\colon \cC \to \cC, \theta\colon S \to S)$ in $\Aut(f)$. Note that $f$ can be obtained as the pullback of $f_m$ via either $S \stackrel{\Phi}{\longrightarrow} S_m$ or the composition $S \stackrel{\theta}{\longrightarrow} S \stackrel{\Phi}{\longrightarrow} S_m$. By the canonicity in Proposition \ref{prop:minimal-universal}, there is a unique isomorphism $(\psi_m\colon\cC_m \to \cC_m, \theta_m\colon S_m \to S_m)$ in $\Aut(f_m)$ that fits in the commutative diagram:
\[
\xymatrix{
S \ar[r]^{\theta} \ar[d]_{\Phi} & S \ar[d]^{\Phi} \\
S_m \ar[r]^{\theta_m} & S_m
}
\]
The arrow $\Aut(f) \to \Aut(f_m)$ is then defined by $(\psi,\theta) \mapsto (\psi_m,\theta_m)$.

To see the injectivity, observe that the morphism
$\cM^{gp}_{S_m} \to \cM^{gp}_{S}$ is surjective by the construction of
Section \ref{sss:um-monoid}.
Thus $\theta_m$ being the identity implies that $\theta$ is also the
identity.
\end{proof}

\subsection{The stack}

\subsubsection{The statements}
Consider the fibered categories of log maps with uniform maximal
degeneracies as in Section~\ref{sss:um-category}.
We now establish their algebraicity and properness.
By Proposition~\ref{prop:uni-min-stack}, \ref{prop:boundedness} and
\ref{prop:log-map-valuative}, it suffices to build these properties
upon the stack of log maps.
We first consider the case of the universal target.

\begin{theorem}\label{thm:max-uni-moduli}
The tautological morphism as in (\ref{equ:forget-max})
\[\fU(\cA, \beta') \to \fM(\cA, \beta')\]
is proper, birational, log \'etale and representable by log algebraic spaces. 
In particular, the fibered category $\fU(\cA, \beta')$ is represented by a log smooth log algebraic stack locally of finite type.
\end{theorem}

Then consider the cartesian diagram
\[
\xymatrix{
\scrU(Y,\beta) \ar[r] \ar[d] & \fU(Y,\beta') \ar[r] \ar[d] & \fU(\cA, \beta') \ar[d] \\
\scrM(Y,\beta) \ar[r] & \fM(Y,\beta') \ar[r] & \fM(\cA, \beta')
}
\]
where the vertical arrows are given by (\ref{equ:forget-max}), and the
two horizontal arrows of the right square are induced by the canonical
strict morphism $Y \to \cA$.
Note that imposing a curve class and requiring the underlying maps be
stable are both representable by open embeddings.
The following is an immediate consequence of the above theorem.

\begin{theorem}\label{thm:max-moduli}
The canonical morphism $\scrU(Y,\beta) \to \scrM(Y,\beta)$ is a proper, representable and log \'etale morphism of log Deligne--Mumford stacks. In particular, $\scrU(Y,\beta)$ is of finite type if $\scrM(Y, \beta)$ is so.
\end{theorem}

We now give the proof of Theorem \ref{thm:max-uni-moduli}, which splits to two parts.

\subsubsection{Representability, boundedness and log \'etaleness}
For simplicity, write $\fM := \fM(\cA,\beta')$ and $\fU := \fU(\cA,\beta')$.

Consider Olsson's log stack $\Log_{\fM}$, which associates to each
strict morphism $T \to \fM$ the category of morphisms of fine log
structures $\cM_T \to \cM$ over $\ul{T}$.
By Proposition~\ref{prop:min-umd-universal}, we may view $\fU$ as the
category fibered over the category of schemes parameterizing log maps
minimal with uniform maximal degeneracy.
By Proposition~\ref{prop:min-umd-open}, the tautological morphism
$\fU \to \Log_{\fM}$ is an open embedding.
Since $\Log_{\fM}$ is algebraic, $\fU$ is a log algebraic stack
equipped with the universal minimal log structure.
By Proposition~\ref{prop:finite-auto-umd}, the morphism $\fU \to \fM$
is representable.
The log \'etaleness of $\fU \to \fM$ follows from \cite[Theorem 4.6
(ii), (iii)]{Ol03}. By Proposition \ref{prop:uni-stack-smooth}, the
stack $\fU$ is log \'etale.

To prove that $\fU \to \fM$ is of finite type, consider a strict
morphism $T \to \fM$ from a log scheme $T$ of finite type, and write
$U := T\times_{\fM}\fU$.
Since being of finite type is a property local on the target, it
suffices to show that $U$ is of finite type.

Denote by $\Lambda$ the collection of log combinatorial types of log maps over $T$. Since $T$ is of finite type, the set $\Lambda$ is finite. Let $\Lambda_{um} = \{(G,V_{\max}) \ | \ G \in \Lambda \}$ be the collection of log combinatorial types of log maps over $U$ as  in Section \ref{sss:log-type-um}. The set $\Lambda_{um}$ is again finite as the number of choices of $V_{\max} \subset V(G)$ for a fixed $G \in \Lambda$ is finite.

For a fine and saturated monoid $P$, we introduce the log stack $\cA_{P}$ with the underlying stack $\big[\spec(k[\NN])/\spec(k[P^{gp}]) \big]$ and the log structure induced by the affine toric variety $\spec(k[P])$.

For each $(G,V_{\max}) \in \Lambda_{um}$, the canonical morphism (\ref{equ:minimize-umd}) induces a morphism of log stacks $\cA_{\oM(G,V_{\max})} \to \cA_{\oM(G)}$. Consider
\[
\cA_{\oM(G,V_{\max}), T} = T\times_{\Log}\cA_{\oM(G,V_{\max})}
\]
where $T \to \Log$ is the canonical strict morphism, and the morphism on the right is the composition $\cA_{\oM(G,V_{\max})} \to \cA_{\oM(G)} \to \Log$. By \cite[Corollary 5.25]{Ol03}, there is an \'etale morphism
\[
\cA_{\oM(G,V_{\max}), T} \to \Log_T.
\]
By the construction of $\fU$, $U$ is an open sub-stack of $\Log_T$. By Definition \ref{def:umd-minimal} and Proposition \ref{prop:min-umd-open}, $U$ is covered by the image of the finite union:
\[
\bigcup_{(G,V_{\max}) \in \Lambda_{um}} \cA_{\oM(G,V_{\max}), T} \to \Log_T.
\]
Thus $U$ is of finite type.

\subsubsection{Properness}

Since $\fU \to \fM$ is representable and of finite type, for properness it suffices to prove the weak valuative criterion.

\step{1}{The set-up of the weak valuative criterion}

Let $R$ be a discrete valuation ring, $\fm \subset R$ be its maximal ideal, and $K$ be its quotient field. Consider a commutative diagram of solid arrows of the underlying stacks
\[
\xymatrix{
\spec K \ar[rr] \ar[d]  && \ul{\fU} \ar[d] \\
\spec R \ar[rr] \ar@{-->}[rru]&& \ul{\fM}
}
\]
It suffices to show that possibly after replacing $R$ by a finite
extension of discrete valuation rings, and $K$ by the corresponding
finite extension of quotient fields, there exists a unique dashed
arrow making the above diagram commutative.

Let $f$ be a minimal log map over $S = (\spec R, \cM_S)$ given by the bottom arrow of the above diagram. Denote by $s, \eta \in S$ the closed and generic points with the log structure pulled back from $S$ respectively.  Let $f_{\eta_u}$ be the log map over $\eta_u = (\ul{\eta}, \cM_{\eta_{u}})$ minimal with uniform maximal degeneracy given by the top arrow. There is a canonical morphism $\eta_u \to \eta$ such that $f_{\eta_u}$ is the pullback of $f_{\eta}$. We will construct the dashed arrow by extending $f_{\eta_u}$ to a log map over $\spec R$ which is the pullback of $f$, and is minimal with uniform maximal degeneracy.

\step{2}{Determine the combinatorial type of the closed fiber}

Passing to a finite extension of $R$ and $K$, denote by $G$ the log
combinatorial type of the closed fiber $f_{s}$ of $f$, and by
$(G_{\eta}, V_{\max,\eta_u})$ the log combinatorial type of
$f_{\eta_u}$.
We next determine the log combinatorial type $(G, V_{\max})$ of
possible extensions of $f_{\eta_u}$.

We may assume that there exists a chart $h\colon \oM(G) \to \cM_{S}$ after taking a further base change.
For each $v\in V(G)$, denote by $e_v \in \oM(G)$ the corresponding degeneracy. Denote by $\gd$ the composition
\begin{equation}\label{equ:close-general}
\oM(G)  \stackrel{h}{\longrightarrow} \cM_S \longrightarrow \cM_{\eta} \longrightarrow \cM_{\eta_u}.
\end{equation}
By Lemma \ref{lem:generize-degeneracy}, the general fiber of $\gd(e_v)$ corresponds to a degeneracy of some vertex $v_{\eta} \in V(G_{\eta})$. Consider the subset $V' \subset V(G)$ consisting of vertices $v$ such that $\gd(e_v)_\eta$ corresponds to the degeneracy of vertices in $V_{\max,\eta_u}$. We define a partial order on $V'$ as follows.

For any $v_1, v_2 \in V'$, observe $\gd(e_{v_2}) - \gd(e_{v_1}) \in K^{\times}$ as it is a difference of maximal degeneracies over $\eta$.  We define
\[v_1 \poleq_u v_2 \ \ \mbox{if} \ \ \big(\gd(e_{v_2}) - \gd(e_{v_1})
  \big) \in R.\]
Denote by $V_{\max} \subset V'$ the collection of maximal elements
under this partial order $\poleq_u$.

We show that $(G,V_{\max})$ is necessarily the log combinatorial type of any possible extension $f_{S_u}$ of $f_{\eta_u}$ over $S_u = (\spec R, \cM_{S_u})$ with uniform maximal degeneracy. Given such an extension, let $V'_{\max}$ be the collection of maximally degenerated vertices of the closed fiber of $f_{S_u}$. By Lemma \ref{lem:generize-degeneracy} and \ref{lem:generize-po}, we have the inclusion $V'_{\max} \subset V'$.

Consider the canonical morphism $\psi\colon S_u \to S$ along which $f$ pulls back to $f_{S_u}$. Thus $\gd$ can be also given by the composition
\[
\oM(G)  \stackrel{h}{\longrightarrow} \cM_S \stackrel{\psi^{\flat}}{\longrightarrow} \cM_{S_u} \longrightarrow \cM_{\eta_u}.
\]
Suppose $v_2 \in V'_{\max}$. Then since $\psi^{\flat}\circ h(e_{v_2}) - \psi^{\flat}\circ h(e_{v_1}) \in \cM_{S_u}$ for any $v_1 \in V'$, we have $\gd(e_{v_2}) - \gd(e_{v_1}) \in R$. This implies $V'_{\max} \subset V_{\max}$. The other direction $V_{\max} \subset V'_{\max}$ is similar.

\step{3}{Principalize degeneracies of elements in $V_{\max}$}

Let $\cK_0 \subset \cM_{S}$ be the log ideal generated by $\{h(e_v) \ | \ v \in V_{\max}\}$. Let $\hat{S}_0 \to S$ be the log blow-up along $\cK_0$, and $f_{\hat{S}_0}$ be the pullback of $f$. We show that $\eta_u \to S$ factors through $\hat{S}_0 \to S$ uniquely.

Indeed, let $(G_{\eta}, V_{\max,\eta_u})$ be the log combinatorial type of $f_{\eta_u}$. By Lemma \ref{lem:generize-degeneracy} and \ref{lem:generize-po}, $\gd(e_v)$ corresponds to the maximal degeneracy of $f_{\eta_u}$ for any $v \in V_{\max}$. Thus $\cK_0$ pulls back to a locally principal log ideal over $\eta_u$ via $\eta_u \to S$. It follows from the universal property of log blow-ups that there is a unique morphism $\eta_u \to \hat{S}_0$ lifting $\eta_u \to S$ .

Since the underlying of $\hat{S}_0 \to S$ is projective, the underlying morphism of $\eta_u \to \hat{S}_0$ extends to a strict morphism $S_0 \to \hat{S}_0$ with the underlying $\ul{S}_0 = \spec R$. In particular, we obtain a morphism $\psi_0\colon \eta_u \to S_0$. Denote by $f_{S_0}$ the pullback of $f_{\hat{S}_0}$over $S_0$. Consider the composition
\[
\gd_0\colon \oM(G) \stackrel{h}{\longrightarrow} \cM_S \longrightarrow \cM_{S_0}
\]
We show that the elements in $V_{\max}$ have the same degeneracy associated to the closed fiber of $f_{S_0}$ by showing that
\begin{equation}\label{equ:prin-max-deg-0}
  \gd_0(e_{v_2}) - \gd_0(e_{v_1}) \in R^{\times}, \ \  \mbox{for any} \ \ v_1, v_2 \in V_{\max}.
\end{equation}

Indeed, observe
\begin{equation}\label{equ:prin-max-deg}
  \gd(e_{v_2}) - \gd(e_{v_1}) \in R^{\times} , \ \ \mbox{for any} \ \ v_1, v_2 \in V_{\max}.
\end{equation}
Since $S_0 \to S$ factors through $\hat{S}_0 \to S$, we have $\gd_0(e_{v_2}) - \gd_0(e_{v_1}) \in \cM_{S_0}$. Since $\gd = \psi^{\flat}_{0}\circ\gd_0$, the claim follows from the fact that
\[
\psi^{\flat}_{0}\big( \gd_0(e_{v_2}) - \gd_0(e_{v_1}) \big) = \gd(e_{v_2}) - \gd(e_{v_1}) \in R^{\times}.
\]

\step{4}{Maximize the degeneracy of elements in $V_{\max}$}

Fix $v_0 \in V_{\max}$.
Consider the finite set
$
V(G) \setminus V_{\max} = \{v_1, \cdots, v_k\}.
$
Define $\cK_{i} \subset \cM_{S_0}$ to be the log ideal generated by $\{ \gd_0(e_{v_i}), \gd_{0}(e_{v_0})\}$ for $i = 1, 2, \cdots, k$. By (\ref{equ:prin-max-deg}) the log ideal $\cK_i$ is independent of the choice of $v_0 \in V_{\max}$. Consider the diagram
\[
\xymatrix{
\hat{S}_k \ar[r] & \hat{S}_{k-1} \ar[r] & \cdots  \ar[r] & \hat{S}_1 \ar[r] & S_0 \\
&&&& \eta_u \ar[u]_{\psi_0} \ar@{-->}[lu] \ar@{-->}[lllu] \ar@/^/@{-->}[llllu]
}
\]
where $\hat{S}_{i+1} \to \hat{S}_i$ is the log blow-up of the pullback of $\cK_i$ via $\hat{S}_i \to S_0$.

Since $\gd(e_0) - \gd(e_{v_i}) \in \cM_{\eta_u}$, the log ideal $\cK_i$ pulls back to a locally principal log ideal over $\eta_u$ via $\psi_0$. Thus we obtain a sequence of dashed arrows $\hat{\psi}_i\colon \eta_u \to \hat{S}_i$ lifting $\psi_0$ as in the above diagram.

Since log blow-ups are projective, we obtain a strict morphism $S_k \to \hat{S}_k$ with underlying $\ul{S}_k = \spec R$ extending the underlying morphism of $\psi_{0}$. Thus for each $i$ we have morphisms $\psi_i\colon \eta_u \to S_i$ and $S_i \to S_0$. Let $f_{S_k}\colon C_{S_k} \to \cA$ over $S_{k}$ be the pullback of $f_{S_0}$.

Consider the composition
$\gd_k\colon \oM(G) \stackrel{h}{\longrightarrow} \cM_S \longrightarrow \cM_{S_k}.$
Since the pullback of $\cK_i$ is locally principal over $S_k$, either
$\gd_k(e_{v_0}) - \gd_k(e_{v_i})$
or
$\gd_k(e_{v_i}) - \gd_k(e_{v_0})$
belongs to $\cM_{S_u}$. We next show that the latter is not possible.

Indeed the construction in Step 2 implies that
$\gd(e_{v_0}) - \gd(e_{v_i}) \in \cM_{\eta_{u}}\setminus R^{\times}.$
Since $\gd = \psi^{\flat}_k\circ\gd_k$, we necessarily have that $\gd_k(e_{v_0}) - \gd_k(e_{v_i}) \in \cM_{S_k}\setminus R^{\times}$ for any $i=1,\cdots, k$. Thus $f_{S_k}$ over $S_k$ has uniform maximal degeneracy by Proposition \ref{prop:um-open}.

\step{5}{Verify the extension and uniqueness}

We show that $f_{S_k}$ is the unique extension of $f_{\eta_u}$ as needed. First observe that the pullback of $f_{S_k}$ along $\psi_{k}$ is the log map $f_{\eta_u}$ minimal with uniform degeneracy. Thus the universality of Proposition \ref{prop:min-umd-universal} implies that $\psi_k$ induces an isomorphism between the generic fiber $f_{S_k,\ul{\eta}}$ of $f_{S_k}$ and $f_{\eta_u}$. Using Proposition \ref{prop:min-umd-universal} again, we obtain a log map $f_{S_u}$ over $S_{u}$ which is minimal with uniform maximal degeneracy, and a morphism $S_{k} \to S_u$ with the identity underlying morphism, along which $f_{S_u}$ pulls back to $f_{S_k}$. This provides the desired extension of $f_{\eta_u}$.

To see the uniqueness, let $f_{S_u}$ over $S_u$ be any extension of $f_{\eta_u}$. Note that there is a canonical morphism $S_{u} \to S$ along which $f$ pulls back to $f_{S_u}$. Since the log combinatorial type $(G,V_{\max})$ is unique as shown in Step 2, the log ideal $\cK_0$ as in Step 3 pulls back to a locally principal log ideal over $S_u$, hence there is a unique morphism $S_u \to S_0$ such that $f_{S_u}$ is the pullback of $f_{S_0}$.

By Condition (2) of Section \ref{sss:um-monoid}, the log ideal $\cK_i$ as in Step 4 pulls back to a locally principal log ideal over $S_u$, hence a unique morphism $S_u \to S_k$ such that $f_{S_u}$ is the pullback of $f_{S_k}$. Applying the universality of Proposition \ref{prop:min-umd-universal} one more time, we obtain an isomorphism $S_u \to S_k$ compatible with pullback of log maps.

\smallskip

Finally, note that $\fU(\cA, \beta') \to \fM(\cA, \beta')$ is an
isomorphism over the open dense substacks with the trivial log
structure on both the source and the target.
Hence it is birational.
This completes the proof of Theorem \ref{thm:max-uni-moduli}. \qed

\subsection{The logarithmic twist}\label{ss:log-twists}

We introduce notions which will be used to extend the cosection across the boundary.

Consider the stack $\fU := \fU(\cA, \beta')$ with its universal pre-stable log map
$
f_{\fU}\colon \cC_{\fU} \to \cA
$
and the projection $\pi_{\fU}\colon \cC_{\fU} \to \fU.$

\subsubsection{The boundary torsor of $\fU$}\label{sss:boundary-torsor}

Consider the global section $e_{\max} \in \Gamma(\fU,\oM_{\fU})$ of
Corollary \ref{cor:max-deg-exists}, which is the maximal degeneracy
over each geometric point.  
Consider the $\cO_{\fU}^{*}$-torsor over $\fU$
\begin{equation}\label{equ:boundary-torsor}
\cT_{\max} := e_{\max}\times_{\oM_{\fU}}\cM_{\fU}
\end{equation}
and the corresponding line bundle
$
\bL_{\max} \supset \cT_{\max}.
$
The composition
\[
\cT_{\max} \to \cM_{\fU} \to \cO_{\fU}
\]
induces a morphism of line bundles
\begin{equation}\label{equ:boundary-complex}
\bL_{\max} \stackrel{}{\longrightarrow} \cO_{\fU}.
\end{equation}
Since $\fU$ is log smooth by Theorem \ref{thm:max-uni-moduli}, the dual of the above defines a section of $\bL_{\max}^{\vee}$ whose vanishing locus is a Cartier divisor $\Delta_{\max} \subset \fU$ such that $\bL_{\max}^{\vee} \cong \cO_{\fU}(\Delta_{\max})$.

\begin{example}
  Example~\ref{ex:UMD} defines a morphism $\cA \to \fU$.
  The pullback of $\Delta_{\max}$ is $0_\cA$ --- the closed point of $\cA$, and the pullback of
  $\bL_{\max}$ is $\cO_\cA(-0_\cA)$.
\end{example}

\subsubsection{The torsor from the target}

By Section \ref{sss:rank-one}, the characteristic sheaf $\oM_{\cA}$ admits a global section
$
\delta_{\infty} \in \Gamma(\cA, \oM_{\cA})
$
whose image in $\oM_{\cA}$ is a local generator. Consider the $\cO^*$-torsor over $\cA$:
\begin{equation}\label{equ:target-torsor}
\cT_{\infty} := \delta_{\infty}\times_{\oM_{\cA}}\cM_{\cA}
\end{equation}
and the corresponding line bundle $\cO_{\cA}(-\infty_{\cA}) \supset \cT_{\infty}$. The composition
\[
\cT_{\infty} \to \cM_{\cA} \to \cO_{\cA}
\]
corresponds to the canonical embedding
$
\cO_{\cA}(-\infty_{\cA}) \to \cO_{\cA}.
$

\subsubsection{The universal twist}

We construct the {\em log twist} as follows.

\begin{lemma}\label{lem:torsor-twist}
Suppose all contact orders of markings in $\beta'$ are trivial. Then $f_{\fU}^{\flat}$ induces a morphism compatible with the $\cO^*_{\cC_{\fU}}$-action
\begin{equation*}
\tf_{\fU}^{\flat}\colon (\pi^{*}_\fU\cT_{\max})\otimes (f^{*}_{\fU}\cT^{\vee}_{\infty}) \to \cM_{\cC_{\fU}}, \ \ a\otimes (-b) \mapsto a  - f_{\fU}^{\flat}(b)
\end{equation*}
where $\cT^{\vee}_{\infty}$ is the dual torsor of $\cT_{\infty}$.
\end{lemma}
\begin{proof}
Consider the sequence of inclusions
\[
\pi^{*}_\fU\cT_{\max} \subset \pi^{*}_\fU\cM_{\fU} \subset \cM_{\cC_{\fU}},
\]
and the composition
\[
f^{*}_{\fU}\cT^{\vee}_{\infty} \subset f^{*}_{\fU}\cM^{gp}_{\cA} \to \cM^{gp}_{\cC_{\fU}},
\]
where the last arrow induced by $f_{\fU}^{\flat}$.
Putting these together, we obtain
\[
(\pi^{*}_\fU\cT_{\max})\otimes (f^{*}_{\fU}\cT^{\vee}_{\infty}) \to \cM^{gp}_{\cC_{\fU}}, \ \ \ \ a\otimes (-b) \mapsto a  - f_{\fU}^{\flat}(b).
\]

To see this morphism factors through $\cM_{\cC_{\fU}}$, it suffices to show the image of the composition
\[
(\pi^{*}_\fU\cT_{\max})\otimes (f^{*}_{\fU}\cT^{\vee}_{\infty}) \to \cM^{gp}_{\cC_{\fU}} \to \oM^{gp}_{\cC_{\fU}}
\]
is contained in $\oM_{\cC_{\fU}}$. Note the image is of the form $e_{\max} - \bar{f}_{\fU}^{\flat}(\delta_{\infty})$. Since $e_{\max}$ is the maximal degeneracy and the contact orders are all trivial, we have $e_{\max} - \bar{f}_{\fU}^{\flat}(\delta_{\infty}) \in \oM_{\cC_{\fU}}$ by the description in Section \ref{ss:log-combinatorial}.
\end{proof}

\begin{proposition}\label{prop:log-twist}
Suppose the contact orders in $\beta'$ are all trivial. Then there is a natural morphism of line bundles over $\cC_{\fU}$
\begin{equation}\label{equ:log-twist}
\tf_{\fU}\colon \pi^{*}_{\fU}\bL_{\max}\otimes f^{*}_{\fU}\cO(\infty_{\cA}) \to \cO_{\cC_{\fU}}
\end{equation}
such that $\tf_{\fU}$ vanishes along non-maximally degenerate components, and is surjective everywhere else.
\end{proposition}
\begin{proof}
The morphism $\tf_{\fU}$ is obtained by composing $\tf_{\fU}^{\flat}$ as in Lemma \ref{lem:torsor-twist} with the structural morphism $\cM_{C_{\fU}} \to \cO_{C_{\fU}}$, and using the corresponding line bundles $\cT_{\infty} \subset \cO_{\cA}(-\infty)$ and $\cT_{\max} \subset \bL_{\max}$.

Consider a non-maximally degenerate component $Z$ with degeneracy $e_Z$. Then over the generic point of $Z$ we have $$e_{\max} - \bar{f}_{\fU}^{\flat}(\delta_{\infty})=e_{\max} - e_Z \in \cM_{\fU}\setminus \{ 0 \}$$ as $e_{\max}$ is the maximal degeneracy. Since the target of $\tf_{\fU}$ is the trivial line bundle, we conclude that $\tf_{\fU}$ vanishes over the non-maximally degenerate components.

Then observe that $e_{\max} - \bar{f}_{\fU}^{\flat}(\delta_{\infty}) = 0$ in $\oM_{C_{\fU}}$ over the maximally degenerate components except those nodes joining maximally degenerate components with non-maximally degenerate components.
\end{proof}

\subsection{A partial expansion}\label{ss:partial-expansion}

Denote by $\cN_{\max} \subset \cM_{\fU}$ the sub-log structure
generated by $\cT_{\max}\subset \cM_{\fU}$.
Then $e_{\max} \subset \Gamma(\fU, \oM_{\fU})$ is a global section
whose image in $\oN_{\max}$ is a local generator.

Denote by $\cA_{\max} := \cA$ the log stack with the boundary divisor $\Delta$ given by the origin. The inclusion $\cN_{\max} \hookrightarrow \cM_{\fU}$ defines a morphism of log stacks $\fm\colon \fU \to \cA_{\max}$ with $\fm^{-1}(\Delta) = \Delta_{\max}$.

Let
$
\fb\colon \cA^{e} \to \cA\times\cA_{\max}
$
be the blow-up of $\infty_{\cA}\times\Delta$ with the naturally induced log structure. Indeed, there a unique open dense point in $\cA^{e}$ with the trivial log structure whose complement is a simple normal crossings divisor in $\cA^e$.  The divisorial log structure associated to this simple normal crossings divisor is $\cM_{\cA^e}$. Furthermore $\fb$ is log \'etale.

Let $\cE_{\fb} \subset \cA^e$ be the exceptional divisor of $\fb$ and
$\infty_{\cA^{e}} \subset \cA^e$ be the proper transform of
$\infty_{\cA}\times\cA_{\max} \subset \cA\times \cA_{\max}$.

\begin{lemma}\label{lem:lift-max-deg}
Suppose the contact orders in $\beta'$ are all trivial. Then there is a commutative diagram of log stacks
\[
\xymatrix{
\cC_{\fU} \ar[rr]^{f^{e}_{\fU}} \ar[rd]_{f_{\fU}\times\fm} && \cA^{e} \ar[ld]^{\fb} \\
 &\cA\times\cA_{\max}&
}
\]
such that
\begin{enumerate}
 \item The inverse image $(f^e_{\fU})^{-1}(\infty_{\cA^e})$ is empty.
 \item For any geometric point $w \to \fU$, an irreducible component $Z \subset \cC_{w}$ over $w$ dominates $\cE_{\fb}$ via $f^e_{\fU}$ if and only if $Z$ is maximally degenerate with non-trivial degeneracy.
\end{enumerate}
\end{lemma}
\begin{proof}
We first construct the morphism $f^e_{\fU}$. Denote by
$$\cK \subset \cM_{\cA\times\cA_{\max}} := \cM_{\cA}\oplus_{\cO^*}\cM_{\cA_{\max}}$$
the log ideal generated by $\cT_{\max}$ and $\cT_{\infty}$.
Consider the log ideal
$(f_{\fU}\times \fm)^{\bullet}\cK \subset \cM_{\cC_{\fU}}$ generated by $(f_{\fU}\times \fm)^{-1}\cK$. Thus $(f_{\fU}\times \fm)^{\bullet}\cK$ is the log
ideal generated by $\cT_{\max}$ and
$f_{\fU}^{\flat}(\cT_{\infty})$.
Since $\fb$ is the log blow-up of $\cK$, to show that
$f_{\fU}\times \fm$ lifts to $f^e$, it suffices to show that
$(f_{\fU}\times \fm)^{\bullet}\cK$ is locally principal, which follows
from Lemma \ref{lem:torsor-twist}.

Now consider geometric points $w \to \Delta_{\max}$ and $x \to \cC_w$. Denote by $e'_{\max}$ and $\delta'$ the corresponding local generators of $\cT_{\max}$ and $f_{\fU}^{\flat}(\cT_{\infty})$ in a neighborhood $W \subset \cC_{w}$ of $x$ respectively. Then by Lemma \ref{lem:torsor-twist} we have $e'_{\max} - \delta' \in \cM_{C_{\fU}}$. Let $\alpha(e'_{\max} - \delta') \in \cO_{W}$ be the corresponding image.

By construction of $\fb$, locally in the smooth topology we can choose a coordinate of $\cE_{\fb}\setminus \infty_{\cA^{e}}$ mapping to $\alpha(e'_{\max} - \delta')$ via $(f^e_{\fU})^*$. Thus $f^e_{\fU}(W)$ dominates $\cE_{\fb} \setminus \infty_{\cA^{e}}$ if and only if $\alpha(e'_{\max} - \delta') \neq 0$ on $W$. Statement (2) follows from the fact that $\alpha(e'_{\max} - \delta')$ vanishes only along non-maximally degenerate components of $\cC_w$.

To see (1), observe that $(f^e_{\fU,w})^{-1}(\infty_{\cA^e})$ is supported on the poles of the section $\alpha(e'_{\max} - \delta')$ over the maximally degenerate components of $\cC_w$. But $\alpha(e'_{\max} - \delta')$ has no poles by Lemma \ref{lem:torsor-twist}.
\end{proof}

We give another description of \eqref{equ:log-twist}. Since $\cE_{\fc} := \pi^*_{\fU}\Delta_{\max} - \cE_{\fb}$ is effective, there is a natural inclusion $(f^e_{\fU})^*\cO(\cE_{\fb}) \to \pi^*_{\fU}\cO(\Delta_{\max})$, hence
\begin{equation}\label{equ:log-twisted-expansion}
\pi^*_{\fU}\bL_{\max}\otimes (f^e_{\fU})^*\cO(\cE_{\fb})  \cong \pi^*_{\fU}\cO(-\Delta_{\max})\otimes (f^e_{\fU})^*\cO(\cE_{\fb})  \to \cO_{\cC_{\fU}}.
\end{equation}

\begin{lemma}\label{lem:compare-universal-log-twist}
The two morphisms \eqref{equ:log-twisted-expansion} and \eqref{equ:log-twist} are identical.
\end{lemma}
\begin{proof}
Since $\fb^*[\infty_{\cA}\times\cA_{\max}] = [\cE_{\fb}] + [\infty_{\cA^e}]$, pulling back \eqref{equ:log-twist} via $\fb$, we have
\[
\pi^*_{\fU}\bL_{\max}\otimes (f^e_{\fU})^*\cO(\cE_{\fb} + \infty_{\cA^e}) \to \cO_{\cC_{\fU}}.
\]
By Lemma \ref{lem:lift-max-deg}, the above morphism becomes
\[
\pi^*_{\fU}\bL_{\max}\otimes (f^e_{\fU})^*\cO(\cE_{\fb}) \to \cO_{\cC_{\fU}},
\]
which is (\ref{equ:log-twisted-expansion}).
\end{proof}

\section{Logarithmic fields}
\label{sec:logfields}


\subsection{$r$-spin curves and their moduli}
The case of stable $r$-spin curves has been studied in \cite{Ja98, Ja00, AbJa03, Ch08}. Following the strategy of \cite{AbJa03}, we extend $r$-spin structures to twisted pre-stable curves.

\subsubsection{$r$-spin structures}
\begin{definition}\label{def:rspin}
An $n$-marked, genus $g$, $r$-spin curve over a scheme $S$ consists of the data
\[
(\cC \to S, \cL, \cL^{r} \cong \omega^{\log}_{\cC/S})
\]
where
\begin{enumerate}
 \item $\cC \to S$ is a family of genus $g$, $n$-marked twisted pre-stable curves.

 \item $\cL$ is a representable line bundle over $\cC$ (in the sense that the associated morphism $\cL\colon \cC \to B\Gm$ to the classifying stack is representable) with a given isomorphism $\cL^r \cong \omega^{\log}_{\cC/S}$ where $\omega^{\log}_{\cC/S}$ is the log cotangent bundle of the log smooth morphism $\cC \to S$.
\end{enumerate}
 The pullback of $r$-spin curves is defined in the usual sense. For simplicity, we may write $(\cC \to S, \cL)$ for an $r$-spin curve over $S$.
\end{definition}

\begin{notation}
For the purposes of this paper, we would like to view the family of curves $\cC \to S$ as a family of log curves equipped with the canonical log structure pulled-back from the stack of log curves as in Section \ref{sss:curve-stack}. This avoids adding extra underlines to both $\cC$ and $S$.
\end{notation}

\begin{notation}
Unlike the usual notation in logarithmic geometry, the log cotangent bundle of $\cC\to S$ in this paper is denoted by $\omega^{\log}_{\cC/S}$ rather than $\omega_{\cC/S}$. We reserve the notation $\omega_{\cC/S}$ for the dualizing line bundle of the family $\cC \to S$. This choice of notations is compatible with the commonly used notation in FJRW theory.
\end{notation}

\subsubsection{Monodromy representation along markings and nodes}

Consider an $r$-spin curve $(\cC \to S, \cL)$ and its $i$-th marking
$\sigma_i \subset \cC$ with the cyclic group $\mu_{r_i}$.
As the line bundle $\cL$ is representable, the action of $\mu_{r_i}$
on $\cL|_{\sigma_i}$ factors through a group homomorphism
\[
\gamma_i\colon \mu_{r_i} \hookrightarrow \Gm
\]
which is called the {\em monodromy representation} along $\sigma_{i}$.

In this paper, we use $\vgamma = (\gamma_i)_{i=1}^{n}$ to denote the collection of monodromy representations along the $n$ marked points. This is a discrete invariant of $r$-spin curves.

\bigskip

Consider a geometric point $q \to \cC$ which is a node. \'Etale locally around $q$, we have the model (\ref{equ:node-local}). Denote by $\cC_{q+}$ and $\cC_{q-}$ the two components intersecting at $q$ with respect to the two coordinates $x$ and $y$ respectively. We obtain two monodromy representations
\[
\gamma_{q\pm}\colon \mu_{r} \to \Gm
\]
of $\cL|_{q}$ at $q \in \cC_{q\pm}$ respectively. The representability of $\cL$ implies that both $\gamma_{+}$ and $\gamma_{-}$ are injective. The balanced condition of $\cC$ at the node $q$ implies that the composition
\[
\mu_r \stackrel{\gamma_{+}\times\gamma_{-}}{\longrightarrow} \Gm\times\Gm \longrightarrow \Gm
\]
is trivial, where the second arrow is the multiplication morphism.

\subsubsection{$r$-spin structure as twisted stable maps}
Given an $r$-spin curve $(\cC \to S, \cL)$ we obtain a unique commutative diagram:
\begin{equation}\label{diag:spin-map-correspondence}
\xymatrix{
 && (C, \omega^{\log}_{C/S})^{1/r} \ar[d] \\
\cC \ar[rru] \ar[rr] \ar[d] && C \ar[d] \\
S_1 \ar[rr] && S_2.
}
\end{equation}
where
\begin{enumerate}
\item $\cC \to C$ is the coarsification. Here we equip both $\cC \to S_1$ and $C \to S_2$ with their canonical log structures as a family of log curves. This is a log \'etale morphism. Furthermore, the bottom morphism $S_1 \to S_2$ induces the identity morphism of the underlying schemes $\underline{S_1} = \underline{S_2} = \underline{S}$, see \cite[Theorem~1.9]{Ol07}.

\item $(C, \omega^{\log}_{C/S})^{1/r} \to C$ is strict and \'etale with the underlying morphism given by taking the $r$-th root stack of $\omega^{\log}_{\cC/S}$ over $C$.

\item $\cC \to (C, \omega^{\log}_{C/S})^{1/r}$ is induced by the $r$-spin structure $\cL^r \cong \omega^{\log}_{\cC/S}$.
\end{enumerate}

Our description of the $r$-spin structure is similar to the case of \cite[Section~1.5]{AbJa03} except that we equip the two families of curves with their canonical log structure for later use.

Conversely, by pulling back the universal $r$-th root along $\cC \to (C, \omega^{\log}_{C/S})^{1/r}$ we obtain an $r$-spin bundle over $\cC$. To summarize, we have

\begin{lemma}\label{lem:spin-twist-stable-map}
The data of an $r$-spin curve $(\cC \to S, \cL)$ is equivalent to the diagram (\ref{diag:spin-map-correspondence}).
\end{lemma}

\subsubsection{The stack of $r$-spin structures}
Denote by $\fM_{g,\vgamma}^{1/r}$ the stack of genus $g$, $n$-marked, $r$-spin curves with monodromy data $\vgamma$ along markings. It can be viewed a fibered category over the category of usual schemes as the log structures on the curves are the canonical ones.

\begin{proposition}\label{prop:curve-stack}
  The stack $\fM_{g,\vgamma}^{1/r}$ is a smooth, log smooth algebraic
  stack locally of finite presentation.
  Furthermore, the tautological morphism removing the $r$-spin
  structures
  \[
    \fM_{g,\vgamma}^{1/r} \to \fM^{\tw}_{g, n}
  \]
  is locally of finite type, quasi-separated, strict, and (log) \'etale.
\end{proposition}
\begin{proof}
  Denote by $\pi\colon \fC \to \fM^{\tw}_{g, n}$ the universal curve,
  and $\fC \to C$ the universal coarse moduli morphism.
  Also, denote by $(\fC, \omega_{\fC/\fM^{\tw}_{g, n}})^{1/r}$ the
  root stack over $\fC$ parameterizing $r$-th roots of
  $\omega_{\fM^{\tw}_{g, n}}^{\log}$.
  As
  $\omega_{C/\fM^{\tw}_{g, n}}^{\log}|_{\fC} \cong
  \omega_{\fC/\fM^{\tw}_{g, n}}^{\log}$, we observe that
  $\tilde{\fC} := (C, \omega_{C/\fM^{\tw}_{g,
      n}}^{\log})^{1/r}\times_{C}\fC \cong (\fC,
  \omega_{\fC/\fM^{\tw}_{g, n}}^{\log})^{1/r}$ with an \'etale
  projection $\tilde{\fC} \to \fC$.

  Consider $S \to \fM^{\tw}_{g, n}$ with the pullback family
  $\tilde{\fC}_S \to \fC_S \to S$.
  By the description of \eqref{diag:spin-map-correspondence}, giving
  an $r$-spin bundle $\cL_S$ over $\fC_S$ is equivalent to giving a
  section $s$ of the projection $\tilde{\fC}_S \to \fC_S$ such that
  the composition
  $\fC_S \to \tilde{\fC}_S \to (C, \omega_{C/\fM^{\tw}_{g,
      n}}^{\log})^{1/r}_S$ is representable.
  Thus the stack $\fM_{g,\vgamma}^{1/r}$ is an open substack of the
  stack $\pi_*\tilde{\fC}$ parameterizing sections of the morphism
  $\tilde{\fC} \to \fC$ over $\fM^{\tw}_{g, n}$ with discrete data
  $\vgamma$.
  By \cite[Theorem~1.3]{HR14}, the stack $\fM_{g,\vgamma}^{1/r}$ is
  algebraic, and the tautological morphism
  $\fM_{g,\vgamma}^{1/r} \to \fM^{\tw}_{g, n}$ is locally of finite
  type and quasi-separated.

  As $\fM^{\tw}_{g, n}$ carries the canonical locally free log
  structure, it remains to show that the morphism
  $\fM_{g,\vgamma}^{1/r} \to \fM^{\tw}_{g, n}$ is \'etale in the usual
  sense.
  We check it using the infinitesimal lifting property.

Let $A \to B$ be a small extension of Artin rings, and consider the commutative diagram of solid arrows
\[
\xymatrix{
\spec B \ar[r] \ar[d] & \fM_{g,\vgamma}^{1/r} \ar[d] \\
\spec A \ar[r] \ar@{-->}[ru]&  \fM^{\tw}_{g, n}
}
\]
It suffices to show that there is a unique dashed arrow making the above diagram commutative. Pulling back the universal families, it remains to construct the section given by the dashed arrows fitting in the commutative diagram of solid arrows
\[
\xymatrix{
\tilde{\fC}_{\spec B} \ar[r] \ar[d] & \tilde{\fC}_{\spec A} \ar[d] \\
\fC_{\spec B} \ar[r] \ar@/^1pc/[u] & \fC_{\spec A} \ar@/_1pc/@{-->}[u]
}
\]
But since the vertical arrows are \'etale, by the infinitesimal lifting of \'etale morphisms, such a dashed arrow exists and is unique.
\end{proof}

The following is an analogue of \cite[Corollary 2.2.2]{AbJa03}

\begin{corollary}\label{cor:r-spin-finite}
The tautological morphism $\fM^{1/r}_{g,\vgamma} \to \fM_{g,n}$ is proper and quasi-finite.
\end{corollary}
\begin{proof}
By viewing $r$-spin curves as twisted stable maps, the properness follows from \cite[Theorem 1.4.1]{AV02}. Since the morphism $\fM_{g,\vgamma}^{1/r} \to \fM^{\tw}_{g, n}$ is \'etale and $\fM^{\tw}_{g, n}\to \fM_{g,n}$ has zero dimensional fibers, we conclude that the composition $\fM_{g,\vgamma}^{1/r} \to \fM^{\tw}_{g, n} \to  \fM_{g,n}$ is quasi-finite.
\end{proof}

\subsubsection{Log $r$-spin curves and their stacks}

\begin{definition}\label{def:log-rspin}
A {\em log $r$-spin curve} over a log scheme $S$ consists of
\[
(\cC \to S, \cL)
\]
where $\cC \to S$ is a log curve (not necessarily equipped with the canonical log structure), and $\cL$ is an $r$-spin structure over the underlying orbifold curve of $\cC \to S$. The {\em pullback} of the log $r$-spin curve is defined as usual using fiber products in the fine and saturated category.
\end{definition}

As every log curve is obtained by the unique pullback from the associated canonical log curve, we have:
\begin{corollary}
The log stack $\fM_{g,\vgamma}^{1/r}$ with its canonical log structure given by its universal curve represents the category of log $r$-spin curves fibered over the category of log schemes.
\end{corollary}

\subsection{Log fields and their moduli}

\subsubsection{Log fields}\label{sss:spin-field}
Given a log $r$-spin curve $(\cC \to S, \cL)$, consider the $\PP^1$-bundle
\[
\underline{\cP} := \PP(\cL\oplus\cO_{\cC}) \to \underline{\cC}.
\]
Denote by $0_{\cP}$ and $\infty_{\cP}$ the zero and infinity section of the above $\PP^1$-bundle with normal bundles $\cL$ and $\cL^{\vee}$ respectively. Let $\cM_{\infty_{\cP}}$ be the log structure over $\underline{\cP}$ associated to the Cartier divisor $\infty_{\cP}$. It is Deligne--Faltings type of rank one, see Section \ref{sss:rank-one}.

Denote by $\cP' = (\underline{\cP}, \cM_{\infty_{\cP}})$ and
$\cP = (\underline{\cP},
\cM_{\cC}|_{\underline{\cP}}\oplus_{\cO^*}\cM_{\infty_{\cP}})$ the
corresponding log stacks where $\cM_{\cC}|_{\underline{\cP}}$ is the
pullback of $\cM_{\cC}$.
There is a natural projection
\begin{equation}\label{equ:compact-r-spin}
\cP \to \cC.
\end{equation}

\begin{definition}\label{def:r-spin-field}
  A \emph{log field} over a log $r$-spin curve $(\cC \to S, \cL)$ over
  a scheme $S$ is a log map $f\colon \cC \to \cP$ which is a section
  of $\cP \to \cC$.
  The triple $(\cC \to S, \cL, f)$ is called an \emph{$r$-spin curve
    with a log field}.
  It is called \emph{stable} if
  $\omega^{\log}_{\cC/S}\otimes f^*\cO(0_{\cP})^{k}$ is positive for
  $k \gg 0$.
  The \emph{pullback} of an $r$-spin curve with a log field is
  defined as usual via the pullback of log curves.
\end{definition}

\subsubsection{Associated log map of a log $r$-spin field}
Note that giving a log field $f\colon \cC \to \cP$ is equivalent to giving an \emph{associated log map}
\begin{equation}\label{equ:reduce-target}
\cC \to \cP',
\end{equation}
which induces a section of $\ul{\cP} \to \ul{\cC}$.
In fact, the inclusion $\cM_{\infty_{\cP}} \to \cM_{\cC}|_{\underline{\cP}}\oplus_{\cO^*}\cM_{\infty_{\cP}}$ defines a natural morphism $\cP \to \cP'$. Thus (\ref{equ:reduce-target}) is given by the composition
$
\cC \to \cP \to \cP'.
$

On the other hand, given a morphism (\ref{equ:reduce-target}) we recover the log field $f$ via
$
\cC \to \cP'\times_{\underline{\cC}}\cC =: \cP.
$
For convenience, we may use $f$ for the corresponding log map (\ref{equ:reduce-target}) when there is no danger of confusion.

\begin{definition}\label{def:umd-spin-field}
A log field has {\em uniform maximal degeneracy} if its associated log map has {\em uniform maximal degeneracy}.

It is called {\em minimal (with uniform maximal degeneracy)} if the associated log map (\ref{equ:reduce-target}) is minimal (with uniform maximal degeneracy).
\end{definition}

\subsubsection{The discrete data of an $r$-spin curve with a log field}
The discrete data of an $r$-spin curve with a log field is given by
\begin{equation}\label{equ:spin-data}
\beta:=(g, \vgamma = (\gamma_i)_{i=1}^{n}, {\bf c} = (c_i)_{i=1}^{n})
\end{equation}
where
\begin{enumerate}
 \item $g$ is the genus.
 \item $\gamma_i$ is the monodromy representation at the $i$-th marking.
 \item $c_i$ is the contact order of the associated log map at the $i$-th marking.
\end{enumerate}

Compared to the discrete data \eqref{discretedata}, the above \eqref{equ:spin-data} does not specify the curve class.
However, since we only allow sections, the curve class is uniquely determined by the collection of contact orders $\bf{c}$:
\begin{equation}\label{equ:fields-class}
A = [0_{\cP}] + \sum_{i=1}^{n} c_i \cdot [\cP_{\sigma_i}]
\end{equation}
where $0_{\cP}$ is the zero section of the projection $\cP \to C$, and
$\cP_{\sigma_i}$ is the fiber over the $i$-th marking $\sigma_i$.
Indeed, \eqref{equ:fields-class} follows by decomposing $A$ according
to the irreducible components $\cZ$ of $\cC$,
\begin{equation*}
  A = \sum_{\cZ \subset \cC} A_\cZ,
\end{equation*}
and noting that if $\cZ$ is non-degenerate, we may scale the
section to zero to see that
\begin{equation*}
  A_\cZ = [0_{\cP}|_\cZ] + \sum c_j [\cP_{\tau_j}],
\end{equation*}
where $c_j$ denotes the contact order at the special point $\tau_j$ on
$\cZ$, and that if $\cZ$ is degenerate, the discussion in
\cite[Proposition 5.2.4]{Ch14} (see also
\eqref{equ:degree-contact-orders}) implies that
\begin{multline*}
  A_\cZ = [\infty_{\cP}|_\cZ]
  = [0_{\cP}|_\cZ] - c_1(\cL|_\cZ) \\
  = [0_{\cP}|_\cZ] + \sum_{j\text{ outgoing}} c_j [\cP_{\tau_j}] - \sum_{j\text{ incoming}} c_j [\cP_{\tau_j}] 
\end{multline*}
where the incoming and outgoing special points are as introduced in
Section~\ref{sss:partial-order}.

Finally, \eqref{equ:fields-class} is obtained by combining the two
cases, and observing the cancellation at the nodes.

\subsubsection{Automorphisms of minimal stable $r$-spin curves with a log field}
An automorphism of an $r$-spin curve with a log field can be defined similarly as in Section \ref{sss:finite-auto} by taking into account the automorphisms on the target $\cP$ induced by the automorphisms of the curve.

\begin{proposition}\label{prop:r-spin-field-finite-auto}
  Consider an $r$-spin curve $(\cC \to S, \cL)$ with a log field $f\colon \cC \to \cP$ over $S$ with $\ul{S}$ a geometric point. Suppose it is minimal (with uniform maximal degeneracy). Then its automorphism group is finite.
\end{proposition}
\begin{proof}
  By Proposition~\ref{prop:finite-auto} and
  \ref{prop:finite-auto-umd}, it suffices to show that the underlying
  structure
  $(\ul{\cC}, \cL, \underline{f}\colon \underline{\cC} \to
  \underline{\cP})$ has finite automorphisms.
  To simplify notation, we will abuse notation and write
  $(\cC, \cL, f\colon \cC \to \cP)$ instead of
  $(\ul{\cC}, \cL, \underline{f}\colon \underline{\cC} \to
  \underline{\cP})$ in this proof.

  The group of automorphisms of $(\cC, \cL, f\colon \cC \to \cP)$
  which fix the dual graph of $\cC$ is of finite index in the full
  automorphism group.
  Hence, it suffices to prove that $f_i\colon \cC_i \to \cP$ has
  finitely many automorphisms for any irreducible component
  $\cC_i \subset \cC$ with all special points of $\cC_i$ marked.
  Since $\cP \to \cC$ is representable, $f_i$ has finitely many
  automorphisms when $\cC_i$ is a stable curve.

  It remains to prove finiteness of automorphism when $\cC_i$ is
  unstable, and hence $\omega_{\cC_i}^{\log}$ and $\cL$ have
  non-positive degree.
  Stability hence implies that
  $\deg \cO(f_i^* 0_\cP) > -\deg(\cL) \ge 0$.
  In particular, $f_i$ cannot be the zero or infinity section.
  Since $f_i$ is a log-map, this implies that there must be a marking
  $\sigma$ where $f_i$ meets the infinity section.
  In particular, we are reduced to the case that $\cC_i$ is genus zero
  with one or two markings.
  In addition, there must be a point $q$ where $f_i$ meets the zero
  section.
  Automorphisms of $(\cC_i, f_i)$ must preserve $q$, hence the
  automorphism group must be a subgroup of the $\CC^*$ of
  automorphisms of $\cC_i$ fixing $\sigma$ and $q$.
  
  Let $\cC_i' = \cC_i \setminus \{q\}$.
  We note that $\cC_i'$ is of the form $[\CC/\mu_a]$ for some
  $a \in \ZZ$, and $\omega_{\cC_i}^{\log}|_{\cC_i'}$ is
  $\CC^*$-equivariantly trivial.
  Hence, we may view $f_i^r|_{\cC_i'}$ as a meromorphic
  function on $\cC_i'$, or equivalently, as a map
  $g_i\colon \cC_i' \to \PP^1$.
  Since we also know that $g_i$ has an isolated pole at $\sigma$, this
  implies that there are only finitely many $\CC^*$-automorphisms of
  $\cC_i'$ that fix $g_i$.
  In particular, $(\cC_i, f_i)$ has only finitely many automorphisms,
  as desired.
\end{proof}

\subsubsection{The stacks of $r$-spin curves with a log field}\label{sss:r-spin-field-stack}

Let $\SF_{\ddata}^{1/r}$ be the category of stable $r$-spin curves with a log field over the category of log schemes with the discrete data $\beta$. Let $\USF_{\ddata}^{1/r} \subset \SF_{\ddata}^{1/r}$ be the subcategory consisting of objects with uniform maximal degeneracy. Next we show that

\begin{theorem}\label{thm:spin-fields-moduli}
The two categories $\USF_{\ddata}^{1/r}$ and $\SF_{\ddata}^{1/r}$ are represented by proper log Deligne--Mumford stacks.
\end{theorem}

For later use, we introduce $\cS$ the stack over
$\fM_{g,\vgamma}^{1/r}$, which associates, to each strict morphism
$T \to \fM_{g,\vgamma}^{1/r}$, the category of sections
$\underline{f}$ of the underlying projective bundle
$\underline{\cP}_T := \PP(\cL_{T}\oplus\cO_{\underline{\cC}_T}) \to
\underline{\cC}_T$ with the curve class given by
\eqref{equ:fields-class}.
Here $(\underline{\cC}_T \to \underline{T}, \cL_{T})$ is the spin
structure given by $T \to \fM_{g,\vgamma}^{1/r}$.
We view $\cS$ as a log stack with the strict morphism to
$\fM_{g,\vgamma}^{1/r}$.

Note that $\cS$ is an open substack of the stack parameterizing
twisted stable maps with the family of targets
$\underline{\cP}_{\fM_{g,\vgamma}^{1/r}} \to \fM_{g,\vgamma}^{1/r}$.
Indeed, requiring $\underline{f}$ to be a section of
$\underline{\cP}_T \to \underline{\cC}_T$ amounts to requiring the composition $\underline{\cC}_T \to \underline{\cP}_T \to \underline{\cC}_T$ to be an isomorphism which is an open condition.
The following is a consequence of \cite[Theorem 1.4.1]{AV02}:

\begin{lemma}\label{lem:algebraicity-usual-section}
The stack $\cS$ is algebraic, locally of finite type.
\end{lemma}

\begin{proof}[Proof of Theorem \ref{thm:spin-fields-moduli}]
By Theorem \ref{thm:max-uni-moduli}, the tautological morphism
\[
\USF_{\ddata}^{1/r} \to \SF_{\ddata}^{1/r}
\]
is proper, log \'etale, and representable by algebraic spaces of finite type. Thus to prove Theorem \ref{thm:spin-fields-moduli}, it remains to prove the statements for $\SF_{\ddata}^{1/r}$ only. We first verify the representability.

Consider the tautological morphism that removes log structures
\[
\SF_{\ddata}^{1/r} \to \cS.
\]
By Proposition \ref{prop:uni-min-stack}, this morphism is represented
by an algebraic stack locally of finite type.
Therefore, Lemma~\ref{lem:algebraicity-usual-section} implies that the
stack $\SF_{\ddata}^{1/r}$ is also algebraic and locally of finite
type.
Proposition \ref{prop:r-spin-field-finite-auto} further implies that
$\SF_{\ddata}^{1/r}$ is a Deligne--Mumford stack.

It remains to prove the properness. We will divide this into two parts: the boundedness part will be proved in Section \ref{ss:r-spin-field-boundedness}, and the valuative criterion will be checked in Section \ref{ss:valuative}.
\end{proof}

\subsection{Boundedness}\label{ss:r-spin-field-boundedness}

We next prove the following result:

\begin{proposition}\label{prop:spin-field-boundedness}
The stack $\SF_{\ddata}^{1/r}$ is of finite type.
\end{proposition}

Consider the tautological morphism
\begin{equation}\label{equ:take-coarse-curve}
\SF^{1/r}_{\ddata} \to \fM_{g,n}
\end{equation}
by taking the corresponding coarse curves. Using the above morphism, the proof of Proposition \ref{prop:spin-field-boundedness} splits into the following two lemmas.

\begin{lemma}
The tautological morphism (\ref{equ:take-coarse-curve}) is of finite type.
\end{lemma}
\begin{proof}
Note that the morphism (\ref{equ:take-coarse-curve}) is given the composition
\[
\SF^{1/r}_{\ddata} \to \cS \to \fM_{g,n}^{1/r} \to  \fM_{g,n}
\]
where the middle arrow is of finite type by Lemma \ref{lem:algebraicity-usual-section}, and the right arrow  is of finite type by Corollary \ref{cor:r-spin-finite}. It remains to show that the morphism $\SF^{1/r}_{\ddata} \to \cS$ is of finite type.

Let $T \to \cS$ be any strict morphism from a log scheme $T$ of finite type, and write $\SF_T := \SF^{1/r}_{\ddata}\times_{\cS}T$. It suffices to show that $\SF_T$ is of finite type. By Proposition \ref{prop:boundedness}, it suffices to show that the discrete data $\beta$ is combinatorially finite over $T$, see Definition \ref{def:combinatorially-finite}. We prove this by applying the strategy similar to \cite[Proposition 5.3.1]{Ch14}.

Denote by $\underline{f}_T$ the universal section of
$\underline{\cP}_T \to \underline{\cC}_T$ over $\underline{T}$.
As $T$ is of finite type, there are finitely many dual graphs for
geometric fibers of the source curve
$\underline{\cC}_T \to \underline{T}$.
Let $\underline{G}$ be any such dual graph of $\underline{\cC}_t$ for
some geometric point $t \to T$.
It remains to show that the choices of log combinatorial types as in
(\ref{equ:combinatorial-type}) with the given dual graph
$\underline{G}$ is finite.

Note that the partition
$V(\underline{G}) = V^n(\underline{G}) \sqcup V^{d}(\underline{G})$ as
in (\ref{equ:combinatorial-type}) is uniquely determined by
$\underline{f}_t$.
Indeed, $V^{d}(\underline{G})$ consists of irreducible components
whose images via $\underline{f}_t$ are contained in the infinity
section of $\cP_{t}$.
The contact orders along the marked points are determined by $\beta$.
Since $\underline{G}$ is a finite graph, the number of partial
orderings $\poleq$ on $V(\underline{G})$ is also finite.
We fix one such choice, denoted again by $\poleq$.
It remains to show that the number of contact orders at the nodes are
finite.

Let $Z \subset \underline{\cC}_t$ be an irreducible component.
Recall the discussion of incoming and outgoing special points in
Section~\ref{sss:partial-order}.
The same discussion as in \cite[Proposition 5.2.4]{Ch14} implies that
\begin{equation}\label{equ:degree-contact-orders}
\deg(\underline{f}^*(\infty_{\cP})|_{Z}) = \sum_{q\text{ outgoing}} \frac{c_q}{r_q}  - \sum_{q\text{ incoming}} \frac{c_q}{r_q},
\end{equation}
where $c_q$ and $r_q$ are the contact order and order of isotropy
group at the special point $q$.

To bound the choices of contact orders at the nodes, we construct a partition:
\[
V(\underline{G}) = V_1 \sqcup V_2 \sqcup \cdots \sqcup V_k
\]
inductively as follows. First, we choose $V_1$ to be the collection of
largest elements in $V(\underline{G})$ with respect to $\poleq$.
Supposing that $V_1, \cdots, V_i$ are chosen, we choose
$V_{i+1} \subset V(\underline{G}) \setminus (\cup_{j=1}^{i}V_j)$ to be
the collection of largest elements with respect to $\poleq$.

By construction, a node $q$ joining component(s) in the same $V_i$
must have $c_q = 0$. Let $Z_1$ be any component corresponding to an
element in $V_1$. Then $Z_1$ has only incoming node(s).
By \eqref{equ:degree-contact-orders}, the choices of contact orders at
these nodes are finite, as contact orders are non-negative integers.
In particular, there are finitely many choices for the contact orders
of the outgoing nodes attached to components of $V_2$.

Now suppose the number of choices of contact orders at the outgoing
nodes attached to components of $V_i$ is finite.
Using \eqref{equ:degree-contact-orders} and the condition that contact
orders are non-negative integers, we conclude that the incoming
nodes of components of $V_i$, hence the outgoing nodes of components
of $V_{i+1}$ have finitely many choices of contact orders.
By induction, the number of choices of contact orders at each nodes is
finite.
This finishes the proof.
\end{proof}

\begin{lemma}
  The image of the morphism $\SF^{1/r}_{\ddata} \to \fM_{g,n}$ is
  contained in an open substack of finite type.
\end{lemma}
\begin{proof}
To bound the image of $\SF^{1/r}_{\ddata} \to \fM_{g,n}$, it suffices to show that the number of rational components of the fibers over $\SF^{1/r}_{\ddata}$ is bounded. For this, it suffices to show that the numbers of unstable components of the source curves are bounded.

Consider any geometric point $t \to \SF^{1/r}_{\ddata}$ with the fiber $f_t\colon \cC_t \to \cP_t$. By the stability as in Definition \ref{def:r-spin-field}, the line bundle $f_t^*(\cO(0_{\cP_t}))$ has non-negative degree along each component of $\cC_t$, and positive degree along each unstable component of $\cC_t$. Furthermore, since the spin bundle $\cL_t$ over $t$ is representable, the degree of $f_t^*(\cO(0_{\cP_t}))$ along each unstable component is at least $\frac{1}{r}$. Since $\deg f_t^*(\cO(0_{\cP_t})) = (2g-2 + n)/r$, the number of unstable components of $\cC_t$ is at most $(2g-2 + n)$.
\end{proof}

\subsection{Valuative criterion}\label{ss:valuative}

Let $R$ be a discrete valuation ring, $m_R \subset R$ be its maximal
ideal, and $K$ be its quotient field.
Let
$( \cC_{\eta} \to \eta, \cL_\eta, f_{\eta}\colon \cC_{\eta} \to
\cP_{\eta})$ be a minimal stable object over
$\eta = (\spec K, \cM_{\eta})$.
Possibly after a finite extension of $R$, we wish to uniquely extend
$f_{\eta}$ to a family $f\colon \cC \to \cP$ over
$S = (\spec R, \cM_{S})$.

\subsubsection{Outline}

The construction of the extension $f$ is rather involved.
Given Proposition \ref{prop:log-map-valuative}, it remains to extend the underlying structure, for which a main ingredient is the properness of the moduli space of twisted
stable maps \cite[Theorem~1.4.1]{AV02}.
However, the dependence of the target $\cP$ on $\omega^{\log}_{\cC}$
does not play well with the non-log-\'etale modifications (attaching
of new rational tails) of $\cC$ that arise when taking a twisted
stable maps limit.
To get around this, our strategy is to first introduce auxiliary
markings in such a way that there are no new rational tails under a
stable maps limit (Section~\ref{sss:add-aux-markings}).
With this, we obtain an extension with the auxiliary markings, which
we then need to remove (Section~\ref{sss:remove-aux-markings}).
In that process, we might introduce unstable components, which we then
need to contract (Section~\ref{sss:contract-unstable}).

One way to think of the auxiliary markings is a way to reduce to the
case of log stable maps to $\cP$ with logarithmic structure at both
$\infty$ and $0$ (similar to \cite{Gu16}).
In that situation, the construction of the extension is much simpler.

\subsubsection{Reduce to the case of nondegenerate irreducible generic fiber}

By Proposition \ref{prop:log-map-valuative}, it suffices to extend
$\underline{f}_{\eta}$ to a family of sections $\underline{f}$ over
$\spec R$.
Taking the normalization of $\underline{\cC}_{\eta}$ and labeling the
preimages of the nodes, it suffices to extend $\underline{f}_{\eta}$
over each component of the normalization.
Here we use that $r$-spin curves glue along their evaluation maps to
the \emph{proper} rigidified inertia stack which are defined by taking
the $r$-th root of $\cO_\Sigma \cong \omega_\cC^{\log}|_\Sigma$.
Thus, we may assume that $\underline{\cC}_{\eta}$ is smooth.

We may further assume that the image of $\underline{f}_{\eta}$ is not entirely contained in $0_{\cP_{\eta}}$ or $\infty_{\cP_{\eta}}$, as otherwise we may simply extend $\underline{f}_{\eta}$ as $0_{\cP_{\eta}}$ or $\infty_{\cP_{\eta}}$ respectively. Passing to a finite extension if necessary, we may assume that $\underline{f}_{\eta}$ intersects $0_{\cP_{\eta}}$ and $\infty_{\cP_{\eta}}$ properly along $\underline{\eta}$-points of $\underline{\cC}_{\eta}$.

As the log structures are irrelevant for extending the underlying structure, we will drop the underline in this section for simplicity, and all stacks are assumed to be underlying stacks unless otherwise specified. It remains to prove the following result.

\begin{proposition}\label{prop:valuative}
Let $(\cC_{\eta}, \cL_{\eta})$ be an irreducible $r$-spin curve, and $f_{\eta}$ be a section of $\cP_{\eta} := \PP(\cL_{\eta}\oplus\cO_{\cC_{\eta}}) \to \cC_{\eta}$ Denote by $0_{\cP_{\eta}}$ and $\infty_{\cP_{\eta}}$ the zero and infinity sections of $\cP_{\eta}$. Suppose that
\begin{enumerate}
 \item $f_{\eta}$ is neither the zero nor the infinity section.
 \item $f_{\eta}$ intersects the infinity section only along marked points.
 \item $\omega^{\log}_{\cC_{\eta}}\otimes f_{\eta}^{*}(\cO(0_{\cP_{\eta}}))^{k}$ is positive for $k\gg 0$.
\end{enumerate}
Possibly after a finite extension, there is a unique $r$-spin curve $(\cC, \cL)$ over $\spec R$ and a section $f$ of $\cP := \PP(\cL\oplus\cO_{\cC}) \to \cC$ extending the triple $(\cC_{\eta}, \cL_{\eta}, f_{\eta})$ such that $\omega^{\log}_{\cC}\otimes f^{*}(\cO(0_{\cP}))^{k}$ is positive for $k\gg 0$.
\end{proposition}

\begin{remark}
In the above proposition, marked points are allowed to be broad, namely the inertia group along the marking can be trivial.
\end{remark}

\begin{notation}
  In the following, we will consider various $r$-spin curves with log
  fields $(\cC_i, \cL_i, f_i)$.
  Their generic fibers over $\eta$ will be decorated by subscripts
  $\eta$.
\end{notation}

We state two useful tools:

\begin{lemma}\label{lem:map-extending}
  Consider an $r$-spin curve $(\cC_{\eta}, \cL_{\eta})$ with its
  coarse moduli $\cC_{\eta} \to C_{\eta}$.
  Let $C \to \spec R$ be a pre-stable curve extending $C_{\eta}$.
  Possibly after a finite base change, there is a unique $r$-spin
  curve $(\cC, \cL)$ with the coarse moduli $\cC \to C$ over $\spec R$
  extending $(\cC_{\eta}, \cL_{\eta})$.
  
  If we are given the further data of a log field
  $f_\eta\colon \cC_\eta \to \cP_\eta$, then, possibly after a finite
  base change, in addition to $(\cC, \cL)$ as above, there is a unique
  twisted stable map $f'\colon \cC' \to \cP := \PP(\cL\oplus \cO_\cC)$
  extending $f_\eta$.
\end{lemma}
\begin{proof}
To prove the statement, we apply properness of twisted stable maps twice, see \cite[Theorem~1.4.1]{AV02}. First, we extend the $r$-spin structure using the twisted stable map point of view as in (\ref{diag:spin-map-correspondence}). We then extend $f_{\eta}$ to $f$ as twisted stable maps.
\end{proof}

\begin{lemma}
  \label{lem:unique-map}
  Let $\cC$ be a normal and integral Deligne--Mumford stack, and $X$
  be a separated Deligne--Mumford stack.
  Consider two morphisms $f, g\colon \cC \to X$ which agree over an
  open dense substack $U \subset \cC$.
  Then $f$ and $g$ agree on all of $\cC$.
\end{lemma}
\begin{proof}
  Define $\cC_\Delta$ by the cartesian square
  \begin{equation*}
    \xymatrix{
      \cC_\Delta \ar[r] \ar[d] & \cC \ar[d]^{(f, g)} \\
      X \ar[r]^-{\Delta} & X \times X.
    }
  \end{equation*}
  Since $X$ is separated and Deligne--Mumford, the diagonal morphism
  $X \to X \times X$ and thus $\cC_\Delta \to \cC$ are finite and
  representable.

  Furthermore, notice that the map
  $(f, g)|_U \colon U \to X \times X$, factors through $\Delta$.
  Hence, we get a section $s\colon U \to \cC_\Delta \times_{\cC}U$ of
  the projection $\cC_\Delta \times_{\cC}U \to U$.
  Let $Y$ be the closure of $s(U)$ in $\cC_\Delta$ equipped with the
  reduced substack structure.
  By construction, $Y$ is integral.
  We have constructed a finite, birational and representable morphism
  $Y \to \cC$.
  Since $Y$ is integral and $\cC$ is integral and normal, we see that
  $Y \to \cC$ is an isomorphism \cite[Lemma 29.53.8]{stacks-project}.
  Note that the morphism $(f,g)$ is the same as the composition
  $\cC \cong Y \to \cC_{\Delta} \to X \stackrel{\Delta}{\to} X\times
  X$.
  In particular, $f$ and $g$ agree on all of $\cC$.
\end{proof}

\subsubsection{Construct an extension with auxiliary markings}
\label{sss:add-aux-markings}
Denote by $\Lambda$ the set of markings of $\cC_{\eta}$.
Taking a finite base change if necessary, we may assume that
$f_{\eta}$ intersects $0_{\cP_{\eta}}$ properly along $\eta$-points of
$\cC_{\eta}$.
Denote by $\Lambda_0$ the set of these intersection points which are
non-marked in $\cC_{\eta}$.
Let $\cC'_{\eta}$ be the marked curve given by $\cC_{\eta}$ together
with the set of markings $\Lambda \cup \Lambda_0$.

Let $\cC'_{\eta} \to C'_{\eta}$ be the coarse moduli morphism. Possibly after a finite base change, let
\begin{equation}\label{equ:any-extension}
C'_1 \to \spec R
\end{equation}
be any family of pre-stable curves with the set of markings
$\Lambda \cup \Lambda_0$ extending $C'_{\eta}$.
Let $C_1 \to \spec R$ be the family of pre-stable curves obtained by
removing the set of markings $\Lambda_0$ from $C'_1$.
By Lemma~\ref{lem:map-extending}, we obtain an $r$-spin curve
$(\cC_1, \cL_1) \to \spec R$ extending $(\cC_{\eta}, \cL_{\eta})$ with
the coarse moduli $\cC_1 \to C_1$.

Let $\tilde{\cC}_1 \to \cC_1$ be the $r$-th root stack along the
markings in $\Lambda_{0}$.
Then $\tilde{\cC}_1$ has the set of markings $\Lambda\cup \Lambda_0$.
Given point $x$ on $\cC_1$ in $\Lambda_0$, we use $\tilde x$ to denote
the corresponding marking of $\tilde{\cC}_1$.
Note that
\begin{equation*}
  \cO_{\cC_1}(x)|_{\tilde{\cC}_1} \cong \cO_{\tilde{\cC}_1}(\sum_{x\in\Lambda_0} r\tilde x), \qquad
  \omega^{\log}_{\cC_1/\spec R}|_{\tilde{\cC}_1} \otimes \cO_{\tilde{\cC}_1}(\sum_{x\in\Lambda_0} r\tilde x) \cong \omega^{\log}_{\tilde{\cC}_1/\spec R}.
\end{equation*}
Thus, the line bundle over $\tilde{\cC}_1$
\begin{equation}\label{equ:spin-extra-twist}
\tilde{\cL}_1 = \cL_{1}|_{\tilde{\cC}_1}\otimes\cO_{\tilde{\cC}_1}(\sum_{x\in\Lambda_0} \tilde x)
\end{equation}
satisfies
$(\tilde{\cL}_1)^r \cong \omega^{\log}_{\tilde{\cC}_1/\spec R}$, and
hence $(\tilde{\cC}_1, \tilde{\cL}_1)$ is an $r$-spin curve over
$\spec R$.

Define $\tilde{\cP}_1 := \PP(\tilde{\cL}_1\oplus\cO)$ and its restriction $\tilde{\cP}_{1,\eta}:= \tilde{\cP}_1\times_{\spec R}\eta$.
The section $f_{\eta}$ induces a section $\tilde{f}_{1,\eta}$ of $\tilde{\cP}_{1,\eta} \to \tilde{\cC}_{1,\eta}$ as follows. 

Let $\cC_{\eta}^{\circ} = \cC_{\eta}\setminus  \Lambda_0$. Observe that $\tilde{\cP}_{1,\eta}|_{\cC_{\eta}^{\circ}} = \cP_{\eta}|_{\cC_{\eta}^{\circ}}$, giving a section $\tilde{f}_{1,\eta}|_{\cC_{\eta}^{\circ}} \colon \cC_{\eta}^{\circ} \to \tilde{\cP}_{1,\eta}$ over $\cC_{\eta}^{\circ} \subset \tilde{\cC}_{1,\eta}$ induced by $f_{\eta}$. To see this section extends to the entire curve $\tilde{\cC}_{1,\eta}$, let $U \subset \tilde{\cC}_{1,\eta}$ be a neighborhood of a marking $\tilde{x} \subset \tilde{\cC}_{1,\eta}$ for $x \in \Lambda_0$. By \eqref{equ:spin-extra-twist}, we have a natural morphism of line bundles $\cL_1|_{U} \to \tilde{\cL}_{1}|_{U}$. Shrinking $U$, we may assume that $f_{\eta}|_{U}$ defines a section of $\cL_1|_{U}$, hence a section of $\tilde{\cL}_{1}|_{U}$ by composing with $\cL_1|_{U} \to \tilde{\cL}_{1}|_{U}$. This gives the desired section $\tilde{f}_{1,\eta}$ which is neither the zero nor the infinity section by construction.

\begin{lemma}\label{lem:ext-with-extra-marking}
With notation as above, possibly after a finite extension, there is an $r$-spin curve with a log field $(\tilde{\cC}_2, \tilde{\cL}_2,\tilde{f}_2)$ extending $(\tilde{\cC}_{1,\eta}, \tilde{\cL}_{1,\eta}, \tilde{f}_{1,\eta})$.
\end{lemma}
\begin{proof}
  Observe that $\tilde{f}_{1,\eta}$ intersects the zero and infinity
  section of $\tilde{\cP}_{1,\eta}$ only along markings in
  $\Lambda\cup\Lambda_0$.
  By \cite[Theorem~1.4.1]{AV02}, possibly after a further finite base
  change, we obtain a stable map
  $\tilde{f}_2\colon \tilde{\cC}_2 \to \tilde{\cP}_1$ extending
  $\tilde{f}_{1, \eta}$.

We claim that the composition $\tilde{\cC}_2 \to \tilde{\cP}_1 \to \tilde{\cC}_1$ contracts only rational components with precisely two special points. Let $Z \subset \tilde{\cC}_2$ be a contracted component. Then $Z$ cannot be a rational tail:
Otherwise, $\tilde{f}_2|_{Z}$ surjects onto a fiber of
$\tilde{\cP}_1 \to \tilde{\cC}_1$.
As over the generic point all intersections with zero and infinity
section are marked, $Z$ contains at least two special points.

Suppose $Z$ has at least three special points. Then two of the special points are either a marked point or a node joining $Z$ with a tree of rational components contracting to a point of $\tilde{\cC}_1$. The above discussion implies that such a tree contains at least one marked point. This is impossible since  $\tilde{\cC}_2 \to \tilde{\cC}_1$  preserves marked points.

Since $\tilde{\cC}_2 \to \tilde{\cC}_1$ contracts only rational
bridges and is compatible with markings, we have
$\omega^{\log}_{\tilde{\cC}_2/\spec R} =
\omega^{\log}_{\tilde{\cC}_1/\spec R}|_{\tilde{\cC}_2}$.
We check that
$(\tilde{\cC}_2, \tilde{\cL}_2 := \tilde{\cL}_1|_{\tilde{\cC}_{2}})$
is an $r$-spin curve over $\spec R$, hence
$\tilde{\cP}_1|_{\tilde{\cC}_2} = \tilde{\cP}_2 :=
\PP(\tilde{\cL}_2\oplus\cO)$.
Thus $\tilde{f}_1$ pulls back to a section $\tilde{f}_2$ of
$\tilde{\cP}_2 \to \tilde{\cC}_2$ as needed.
\end{proof}

\subsubsection{Remove auxiliary markings}
\label{sss:remove-aux-markings}

\begin{lemma}\label{lem:remove-extra-marking}
  Let $(\tilde{\cC}_2, \tilde{\cL}_2,\tilde{f}_2)$ be as in
  Lemma~\ref{lem:ext-with-extra-marking}.
  Let $\tilde{\cC}_2 \to \cC_2$ be obtained by first rigidifying along
  markings in $\Lambda_0$, then removing $\Lambda_0$ from the set of
  markings.
  Then there is an $r$-spin curve with a log field $(\cC_2, \cL_2, f_2)$
  extending $(\cC_{\eta}, \cL_{\eta}, f_{\eta})$ such that
\begin{enumerate}
 \item $\tilde{\cL}_2 = \cL_2|_{\tilde{\cC}_2}\otimes\cO_{\tilde{\cC}_2}(\sum_{x\in\Lambda_0}\tilde x)$,
 \item $f_{2}$ and $\tilde{f}_2$ are isomorphic away from the sections in $\Lambda_0$,
 \item $f_{2}$ sends sections in $\Lambda_0$ to the zero section of $\cP_2 := \PP(\cL_2\oplus \cO)$.
\end{enumerate}
\end{lemma}
\begin{proof}
  We first construct the spin bundle $\cL_2$. Let $\cC_2 \to C_2$ be
  the coarse moduli morphism.
  Then $C_2$ over $\spec R$ extends $C_{\eta}$ as a family of
  pre-stable curves with the set of markings $\Lambda$.
  By Lemma~\ref{lem:map-extending}, we obtain an $r$-spin curve
  $(\cC_3, \cL_3)$ over $\spec R$ extending $(C_{\eta}, L_{\eta})$.

  Let $\tilde{\cC}_3 \to \cC_3$ be the $r$th root construction along
  sections in $\Lambda_0$, and view $\tilde{\cC}_3$ as a family of
  pre-stable curves with markings $\Lambda\cup\Lambda_0$.
  Consider the line bundle
  $ \tilde{\cL}_3 :=
  \cL_{3}|_{\tilde{\cC}_3}\otimes\cO_{\tilde{\cC}_3}(\sum_{x \in
    \Lambda_0}\tilde x)$ over $\tilde{\cC}_3$.
  We check that $\tilde{\cL}_3$ is an $r$-spin bundle over
  $\tilde{\cC}_3$.
  Since
  $\tilde{\cC}_{3,\eta} = \tilde{\cC}_{1,\eta} =
  \tilde{\cC}_{2,\eta}$, the $r$-spin structure
  $(\tilde{\cC}_3, \tilde{\cL}_3)$ over $\spec R$ extends
  $(\tilde{\cC}_{2,\eta}, \tilde{\cL}_{2,\eta} =
  \tilde{\cL}_{1,\eta})$.
  As both $\tilde{\cC}_2$ and $\tilde{\cC}_3$ have the same coarse
  curve $C_2$, by the uniqueness of Lemma~\ref{lem:map-extending}, we
  conclude that
  $(\tilde{\cC}_{3},\tilde{\cL}_3) \cong
  (\tilde{\cC}_{2},\tilde{\cL}_2)$, hence $\cC_3 = \cC_2$ and
  $\cL_2 := \cL_3$.
  This proves (1).

We now construct the section $f_2$. Possibly after a finite extension, $f_{\eta}$ extends to a twisted stable map $\hat{f}_2\colon \hat{\cC}_2 \to \cP_2$. Consider the commutative diagram
\[
\xymatrix{
\tilde{\cP}_2 \ar@{-->}[rr] \ar[d] && \cP_2 \ar[d] && \\
\tilde{\cC}_2 \ar@/^1pc/[u]^{\tilde{f}_2} \ar[rr] && \cC_2  && \hat{\cC}_2 \ar[ll]_p \ar[llu]_{\hat{f}_2}
}
\]
where $\tilde{\cC}_2 \to \cC_2$ is the rigidification along
$\Lambda_0$ as shown in the previous paragraph, hence
$V := \tilde{\cC}_2\setminus \Lambda_0 \cong \cC_2\setminus
\Lambda_0$, and the dashed arrow is a rational map which is a
well-defined isomorphism over $V$.
Let $W = p^{-1}(V)$.
Then, since $(\tilde f_2 \circ p)|_W$ and $\hat f_2|_W$ agree over the
generic fiber, they also agree over all of $W$ by
Lemma~\ref{lem:unique-map}.
Hence, since $\hat{f}_2$ is a stable map limit,
Lemma~\ref{lem:map-extending} implies that $W \cong V$, and that
$\hat{\cC}_2 \to \cC_2$ is a contraction of rational components over
$\Lambda_0$ in the closed fiber.

Let $Z_x \subset \hat{\cC}_2$ be the preimage of a point
$x \in \cC_{2,\spec R/m_R}$ in $\Lambda_0$.
Suppose $Z_x$ is not a point.
Then $\hat{f}_2|_{Z_x}$ surjects onto a fiber of $\cP_2 \to \cC_2$,
hence intersects the infinity section non-trivially. We argue that this is not possible as follows. 

First $Z_x$ contains no markings, as otherwise it is in contradiction to the fact that $\Lambda$
is disjoint from $\Lambda_0$. By Proposition~\ref{prop:log-map-valuative}, we may lift $\hat{f}_2$ to 
a log map over the log scheme $S$, denoted again by $\hat{f}_2$, where the target $\cP_{2}$ is equipped with the log structure given by its infinity divisor $\infty_{\cP_2}$.

Let $Z' \subset Z_x$ be an irreducible component not contracted by
$\hat{f}_2$. Then $\hat{f}_2|_{Z'}$ must intersect properly with the infinity
section along at least one point, say $z \in Z'$. Observe that in absence of marking, $z$ is necessarily a node of $\hat{C}_2$. Otherwise, assume $z$ is a smooth unmarked point.   
Consider the morphism of characteristic monoids 
\[
\bar{\hat{f}}^{\flat}_2|_z \colon (\hat{f}^*\oM_{\cP_2})|_z \cong \NN \longrightarrow \oM_{\hat{C}_2}|_{z} \cong Q := \oM_{S}|_{\spec R/m_R}
\]
induced by the log map $\hat{f}_2$. The assumption $\hat{f}_2(z) \in \infty_{\cP_2}$ implies that $\bar{\hat{f}}^{\flat}_2|_z(1) \neq 0$. Note that the sheaf $\oM_{\hat{C}_2}|_{(Z')^{\circ}}$ is globally constant over the locus of smooth unmarked points $(Z')^{\circ} \subset Z'$. Hence  $\bar{\hat{f}}^{\flat}_2|_z(1) \in Q$ is the degeneracy of the component $Z'$. This is in contradiction to $Z'$ being non-degenerate, see \S \ref{sss:map-at-generic-pt}. Therefore,  $\hat{f}_2|_{Z'}$ necessarily intersects with the infinity section along a node of $Z_x$ joining $Z'$ and a component $Z'_1 \subset Z_x$.
Since this node is an incoming node of $Z'_1$,
 $\hat{f}_2|_{Z'_{1}}$ is a contraction to the infinity section and
$\deg \hat{f}^*_2(\infty_{\cP_2})|_{Z'_{1}} = 0$.
Applying \eqref{equ:degree-contact-orders} to the component $Z'_{1}$
and noting that $Z'_1$ has no markings, we observe that $Z'_{1}$ must
contain an outgoing node joining $Z'_1$ and $Z'_2$ with $Z'_2$
contracted to infinity. 
Indeed, the outgoing node of $Z'_1$ is an incoming node of $Z'_2$.

Applying the same argument inductively to $Z'_i$, we obtain an
infinite chain of rational components
$(Z'_1 \cup Z'_2 \cup \cdots )\subset Z_x$, which is not possible.
This proves (2).

The third statement follows since $f_{2,\eta} = f_{\eta}$ sends sections in $\Lambda_0$ to the zero section of $\cP_{2} \to \cC_2$.
\end{proof}

\subsubsection{Contract unstable components}
\label{sss:contract-unstable}

Let $\cC_2 \to C_2$ be the coarse moduli morphism where $C_2$ is a
family of pre-stable curves over $\spec R$ with the set of markings
$\Lambda$.
An irreducible component $\cZ \subset \cC_2$ is \emph{unstable} if
$\omega^{\log}_{\cC_2}\otimes f^{*}(\cO(0_{\cP_2}))^{k}$ fails to be
positive on $\cZ$ for $k\gg 0$.
Let $Z \subset C_2$ be the image of $\cZ$.
Then $Z$ is \emph{unstable} if $\cZ$ is so.
Note that all unstable components are over the closed point
$\spec R/m_R$, and are rational components with at most two markings.

\begin{lemma}\label{lem:stabilize-bridge}
  Let $C_2 \to C_3$ be a contraction of an unstable component $Z$ with
  two special points. Possibly after a further finite base change, we
  obtain an $r$-spin curve with a log field $(\cC_3, \cL_3,f_3)$
  extending $(\cC_{\eta}, \cL_{\eta}, f_{\eta})$ such that
\begin{enumerate}
 \item $\cC_3 \to C_3$ is the coarse moduli morphism.
 \item $\cC_2 \to \cC_3$ contracts $\cZ \subset \cC_2$ to a point.
 \item $\cP_{2} = \cP_{3}\times_{\cC_3}\cC_2$ and $f_{2}$ is the pullback of $f_3$.
\end{enumerate}
Here we do not require that $f_2$ is of the form in Lemma \ref{lem:remove-extra-marking}.
\end{lemma}

\begin{proof}
By Lemma \ref{lem:map-extending}, we obtain an $r$-spin curve $(\cC_3, \cL_{3})$ over $\spec R$ extending $(\cC_{\eta}, \cL_{\eta})$ with the coarse moduli morphism $\cC_3 \to C_3$. Consider the cartesian diagram of solid arrows
\[
\xymatrix{
\cC_2 \ar@{-->}[r] \ar@{-->}[rd] & \hat{\cC}_2 \ar[r] \ar[d] & (C_2, \omega^{\log}_{C_2/\spec R})^{1/r} \ar[d] \ar[r] & C_2 \ar[d] \\
 & \cC_3 \ar[r] & (C_3, \omega^{\log}_{C_3/\spec R})^{1/r} \ar[r] & C_3
}
\]
The square on the right is cartesian as
$\omega^{\log}_{C_3/\spec R} = \omega^{\log}_{C_2/\spec R}|_{C_3}$.

Let $\cC''_2 \to \hat{\cC}_2$ be the twisted stable map extending the
isomorphism $\cC_{2,\eta} \to \hat{\cC}_{2,\eta}$.
By pulling back the universal $r$-th root via
$\cC''_2 \to \hat{\cC}_2 \to (C_2, \omega^{\log}_{C_2/\spec
  R})^{1/r}$, we obtain an $r$-spin bundle $\cL''_2$ over $\cC''_2$
extending $\cL_{\eta}$. By uniqueness of Lemma
\ref{lem:map-extending}, we obtain
$(\cC''_2, \cL''_2) = (\cC_2,\cL_2)$ hence the dashed horizontal arrow
as above. The skew dashed arrow is then the composition
$\cC_2 \to \hat{\cC}_2 \to \cC_3$. Thus (2) follows.

It follows from the above construction that $\cL_2 = \cL_{3}|_{\cC_2}$. Hence $\cP_{2}$ is the pullback of $\cP_{3}$ via $\cC_2 \to \cC_3$, and $f_2$ is the pullback of $f_2$.  This proves (3).
\end{proof}

We next remove unstable rational tails. We first prove

\begin{lemma}\label{lem:unstable-tail-vanishing}
Choosing the extension (\ref{equ:any-extension}) appropriately, we may assume that $\cC_2$ in Lemma \ref{lem:remove-extra-marking} has unstable rational tails contained in the zero section of $\cP_2$ via $f_2$.
\end{lemma}
\begin{proof}
Let $\cZ \subset \cC_2$ be an unstable rational tail not contained in the zero section of $\cP_2$. Then $\cZ$ contains no section from $\Lambda_0$. Since $\cZ$ is from a component $\tilde{\cZ} \subset \tilde{\cC}_2$, and $\tilde{f}_2$ is a twisted stable map, $\cZ$ maps to a rational tail $Z \subset C_1$. Blowing down $Z$, we obtain another extension (\ref{equ:any-extension}) together with the same set of sections $\Lambda\cup\Lambda_0$.
\end{proof}

We then contract the unstable rational tails inductively as follows.

\begin{lemma}\label{lem:stabilize-tail}
  Let $(\cC_2, \cL_2,f_2)$ be an $r$-spin curve with a log field
  extending $(\cC_{\eta},\cL_{\eta}, f_{\eta})$.
  Suppose all unstable rational tails of $\cC_2$ are contained in the
  zero section of $\cP_{2} \to \cC_2$ via $f_2$.
  Let $\cC_2 \to C_2$ be the coarse moduli morphism, and $C_2 \to C_3$
  be the contraction of an unstable rational tail $Z$.
  Then there is a triple $(\cC_3, \cL_3, f_3)$ extending
  $(\cC_{\eta},\cL_{\eta}, f_{\eta})$ with coarse moduli morphism
  $\cC_3 \to C_3$ such that all unstable rational tails of $\cC_3$ are
  contained in the zero section of $\cP_{3} = \PP(\cL_3 \oplus \cO)$
  via $f_3$.
\end{lemma}
\begin{proof}
  By Lemma \ref{lem:map-extending}, we obtain the $r$-spin curve
  $(\cC_3, \cL_3)$ extending $(\cC_{\eta}, \cL_{\eta})$ with the
  coarse moduli morphism $\cC_3 \to C_3$, and a twisted stable map
  $f_4\colon \cC_4 \to \cP_3$ over $\spec R$ extending $f_{\eta}$.
  Let $x$ be the image of $Z \to C_3$.
  As the pairs $(\cC_2, \cL_2)$ and $(\cC_3, \cL_3)$ are isomorphic
  away from the preimage of $Z$ and $x$, $f_4$ and $f_2$ are
  isomorphic away from the preimage of $Z$ and $x$.
  Thus the composition $\cC_4 \to \cP_3 \to \cC_3$ is a contraction of
  rational components $\cZ_x$ over $x \in \cC_3$. Then the same argument as for Lemma \ref{lem:remove-extra-marking} (2) implies that $\cZ_x$ has to be a point. 
\end{proof}

We start with an $r$-spin curve with a log field as in Lemma
\ref{lem:remove-extra-marking} and \ref{lem:unstable-tail-vanishing},
and inductively apply Lemma \ref{lem:stabilize-bridge} and
\ref{lem:stabilize-tail} by contracting unstable components.
After finitely many steps, we obtain $(\cC, \cL, f)$ as in Proposition
\ref{prop:valuative}.

\subsubsection{Separatedness}

Consider stable extensions $(\cC_i, \cL_i, f_i)$ of
$(\cC_{\eta}, \cL_{\eta}, f_{\eta})$ for $i=1,2$.
Let $\cC_i \to C_i$ be the coarse moduli for $i=1,2$.
By Lemma \ref{lem:map-extending} or Lemma~\ref{lem:unique-map}, it
suffices to show that there is an isomorphism $C_1 \cong C_2$
extending the one over $\eta$.

Let $C_3$ be a family of prestable curves over $\spec R$ extending
$C_{\eta}$ with dominant morphisms $C_3 \to C_i$ for $i=1,2$.
We may assume $C_3$ has no rational components with at most two
special points contracted in both $C_1$ and $C_2$ by further
contracting these components.

Let $\cC_3' \to \cC_1\times \cC_2\times C_3$ be the family of twisted stable maps over $\spec R$ extending the obvious one $\cC_\eta \to \cC_1\times \cC_2\times C_3$. Observe that the composition $\cC_3' \to \cC_1\times \cC_2\times C_3 \to C_3$ is the coarse moduli morphism. Indeed, if there is a component of $\cC_3'$ contracted in $C_3$, then it will be contracted in both $\cC_1$ and $\cC_2$ as well.

Let $\cC_3 \to (\cC_3',\omega^{\log}_{\cC_3'/\spec R})^{1/r}$ be the twisted stable map extending the spin structure over $\eta$. Then we obtain a (not necessarily representable) spin bundle $\cL_3$ over $\cC_3$. We next compare $C_3$ and $C_i$ for $i=1,2$.

For $i = 1, 2$, set $U_i$ be the complement in $C_3$ of all trees of
rational components contracted by $C_3 \to C_i$.
More explicitly, set $U_i^{(0)} = C_3$.
Let $U_i^{(k+1)}$ be obtained by removing from $U_i^{(k)}$ the
rational components with precisely one special point in $U_i^{(k)}$
and that are contracted in $C_i$.
We observe that this process must stop after finitely many steps.
Denote by $U_i \subset C_3$ the resulting open subset.

\begin{lemma}\label{lem:cover}
\begin{enumerate}
\item $U_1 \cup U_2 = C_3$.
\item $C_3 \setminus U_1 \subset U_2$ and $C_3 \setminus U_2 \subset U_1$.
\end{enumerate}
\end{lemma}
\begin{proof}
Suppose $z \in C_3 \setminus (U_1\cup U_2) \neq \emptyset$. Then there is a tree of rational curves in $C_3$ attached to $z$ and contracted in both $C_1$ and $C_2$. This contradicts the assumption on $C_3$. Statement (2) follows from (1).
\end{proof}

Consider the coarse moduli morphism $\cC_3 \to C_3$.
Define $\cU_i := \cC_3\times_{C_3}U_i$ for $i=1,2$.
Since $U_{i} \to C_i$ contracts only rational components with two
special points in $U_i$, we have
$\omega^{\log}_{C_i/\spec R}|_{U_i} = \omega^{\log}_{U_i/\spec R}$
which further pulls back to
$\omega^{\log}_{\cC_i/\spec R}|_{\cU_i} = \omega^{\log}_{\cU_i/\spec
  R}$.
Thus the pullback $\cL_{i}|_{\cU_i}$ is an $r$-th root of
$\omega^{\log}_{\cU_i/\spec R}$.
Recall the $r$-spin bundle $\cL_{3}|_{\cU_i}$.
Note that $\cU_{i,\eta} = \cC_{\eta}$ and
$\cL_{3}|_{\cU_{i,\eta}} \cong \cL_{i,\eta}$.
Using Lemma~\ref{lem:unique-map}, we see that this isomorphism extends
uniquely to $\cL_{3}|_{\cU_{i}} \cong \cL_{i}|_{\cU_i}$ for $i=1,2$.
This allows us to glue the pullbacks $f_{i}|_{\cU_i}$ to a field
$f_3: \cC_3 \to \cP_3$.

In the following, indices ``$s$'' denote base-change to the central
fiber.
\begin{lemma}\label{lem:degree-comparison}
  $\deg(f_{3,s}^*\cO(0_{\cP_3})|_{\overline{\cU_{i,s}}}) \geq
  \deg(f_{i,s}^*\cO(0_{\cP_i}))$, for $i=1,2$.
\end{lemma}
\begin{proof}
  Note that on the components $\cZ$ contracted by $U_i \to C_i$, we
  have $\deg(f_{3,s}^*\cO(0_{\cP_3}))|_{\cZ} = 0$.
  Hence, it suffices to prove the inequality component-wise, and we
  may assume that $\overline{\cU_{i,s}}$ is irreducible.
  
  Note that
  $\omega^{\log}_{\cC_3/\spec R}|_{\overline{\cU_{i,s}}} =
  \omega^{\log}_{\cC_i/\spec R}|_{\overline{\cU_{i,s}}}(D')$ for some
  effective divisor $D'$ supported on the special points
  $\overline{\cU_{i,s}} \setminus \cU_{i,s}$ of $\cC_{3, s}$.
  Further note that $\cL_{3}$ and $\cL_{i}$ are the $r$-th roots of
  $\omega^{\log}_{\cC_3/\spec R}$ and $\omega^{\log}_{\cC_i/\spec R}$
  respectively and $\cL_{3}|_{\cU_{i}} = \cL_{i}|_{\cU_i}$.
  Thus there exists an effective divisor $D$ supported on
  $\overline{\cU_{i,s}} \setminus \cU_{i,s}$ such that
  $\cL_3|_{\overline{\cU_{i,s}}} \cong
  \cL_i|_{\overline{\cU_{i,s}}}(D)$.

  We may assume without loss of generality that $D \neq 0$.
  If one, or equivalently, both $f_{3,s}|_{\overline{\cU_{i,s}}}$ and
  $f_{i,s}|_{\overline{\cU_{i,s}}}$ map into the infinity section, we
  are also done because then
  $f_{3,s}^*\cO(0_{\cP_3})|_{\overline{\cU_{i,s}}}$ and
  $f_{i,s}^*\cO(0_{\cP_i})|_{\overline{\cU_{i,s}}}$ are trivial.
  We may thus assume that $f_{3,s}|_{\overline{\cU_{i,s}}}$ and
  $f_{i,s}|_{\overline{\cU_{i,s}}}$ are not the infinity section,
  and may thus be viewed as rational sections of
  $\cL_3|_{\overline{\cU_{i,s}}}$ and $\cL_i|_{\overline{\cU_{i,s}}}$,
  respectively.
  Since they agree on $\cU_{i,s}$ and
  $\cL_3|_{\overline{\cU_{i,s}}} \cong
  \cL_i|_{\overline{\cU_{i,s}}}(D)$ for an effective $D$, we have
  \begin{equation*}
    \deg f_{3,s}^*\cO(0_{\cP_3})|_{\overline{\cU_{i,s}}} - \deg f_{i,s}^*\cO(0_{\cP_i})|_{\overline{\cU_{i,s}}} \geq 0.
  \end{equation*}
\end{proof}

Suppose $C_1 \neq C_2$. Then we have $U_i \neq C_i$ for some $i$, say
$i=1$.
By construction each connected component of $C_3 \setminus U_1$ is a
tree of proper rational curves in $U_2$ with no marked point, and by
Lemma~\ref{lem:cover} (2),
$\cT := (\cC_3 \setminus \cU_1) \subset \cU_2$.

By construction, the composition $\cT \to \cC_3 \to \cC_2$ is a closed
immersion and $f_{3}|_{\cT} = f_{2}|_{\cT}$.
Since $\deg \omega^{\log}_{\cC_3/\spec R}|_{\cT} < 0$ (unless $\cT = \emptyset$), the stability of $f_2$ implies
\begin{equation*}
\deg f_3^*\cO(0_{\cP_3})|_{\cT} =  \deg f_2^*\cO(0_{\cP_2})|_{\cT} > 0.
\end{equation*}
Using Lemma  \ref{lem:degree-comparison}, we calculate
\[
\begin{aligned}
\deg f_{3,s}^*\cO(0_{\cP_3}) &= \deg f_{3,s}^*\cO(0_{\cP_3})|_{\overline{\cU_{1,s}}} + \deg f_{3,s}^*\cO(0_{\cP_3})|_\cT \\
&\geq  \deg f_{1,s}^*\cO(0_{\cP_1}) + \deg f_{3,s}^*\cO(0_{\cP_3})|_\cT.
 \end{aligned}
\]
Combining this with the fact that both $f_1$ and $f_3$ extend
$f_\eta$, so that
$\deg f_{3,s}^*\cO(0_{\cP_3}) = \deg f_{1,s}^*\cO(0_{\cP_1})$, we see
that $\cT = \cC_3 \setminus \cU_1 = \emptyset$.

Observe that $C_3 = U_1 \to C_1$ contracts proper rational components
with precisely two special points.
Let $Z \subset C_3$ be such a component, and let
$\cZ = Z \times_{C_3}\cC_3$.
Since $f_3|_{\cC_3 = \cU_1}$ is the pullback of $f_1$, we have
\begin{equation}\label{equ:bridge-trivial-degree}
\deg f_3^*\cO(0_{\cP_3})|_{\cZ} = 0.
\end{equation}

On the other hand, since $Z$ has two special points in $C_3$ and is
contracted in $C_1$, it is not contracted in $C_2$.
Denote by $\cZ' \subset \cC_2$ the component dominating
$Z \subset C_2$.
Then $\cZ'$ has precisely two special points.
Furthermore $f_2|_{\cZ'}$ and $f_3|_{\cZ}$ coincide away from the two
special points.
Using \eqref{equ:bridge-trivial-degree}, we observe that
$\deg f_2^*\cO(0_{\cP_2})|_{\cZ'} = 0$, which contradicts
the stability of $f_2$.
Thus $C_3 \to C_1$ is an isomorphism.

This completes the proof of Proposition \ref{prop:valuative}. \qed

\subsubsection{Failure of properness without log structure along $\infty_{\cP}$}\label{sss:properness-failure}

As our target has the non-trivial log structure $\cM_{\infty_{\cP}}$
along the infinity section (see Section~\ref{sss:spin-field}), a
non-degenerate component can only intersect $\infty_{\cP}$ along nodes
or markings.
Hence we have condition (2) in Proposition~\ref{prop:valuative}.
It turns out that this condition is necessary for proving the weak
valuative criterion for the moduli of meromorphic sections of the spin
bundle.
We exhibit this necessity using the following example.

Consider the case that $r = 1$.
Let $C = \PP^1$ with three marked points $z = 1, 2, \infty$ where
$z = u/v$ for a fixed homogeneous coordinates $[u:v]$ of $\PP^1$.
Consider a family of meromorphic differentials $f_t = t\frac{dz}z$
over $C$ where $t$ is the parameter over a punctured disc
$\Delta \setminus 0$.
Observe that $f_t$ intersects the infinity section transversally at a
single non-marked point $z=0$.
We claim that the limit as $t \to 0$ does not exist as a section of
$\PP(\omega^{\log}\oplus\cO)$ with finite automorphisms.

Suppose possibly after a finite base change, the family $f_t$ extends
to a family $f$ of sections of
$\cP := \PP(\omega^{\log}_{C_{\Delta}/\Delta}\oplus\cO) \to
C_{\Delta}$ over $\Delta$.

Consider $f_0$ as a section of $\cP|_{C_0}$ over $C_0$.
By semi-stable reduction, there is a contraction morphism
$C_\Delta \to C \times \Delta$.
We may then write $C_0 = Z_1 \cup Z_2$ where $Z_2$ is the pre-image of
$[0:1] \in C$ in $C_0 \subset C_\Delta$, and where $Z_1$ is the closure of
its complement in $C_0$.
Note that we may extend $f_t$ by zero in the central fiber away from
$Z_2$.
Hence, using Lemma~\ref{lem:unique-map}, we see that
$f(Z_1) \subset 0_\cP$.

Let us write $\cL = f_0^* \cO(\infty_\cP)$.
We note that $\deg(\cL) = 1$ and $\deg(\cL|_{Z_1}) = 0$, so that
$\deg(\cL|_{Z_2}) = 1$.
Let $Z$ be the union of components of $Z_2$ mapped to $\infty_\cP$.
Then, we have $\deg(\cL|_Z) = -\deg(\omega^{\log}|_Z) = 2n - m$, where
$n$ is the number of connected components of $Z$, and where $m$ is the
number of nodes on $Z$ that connect $Z$ to other components.
Note that $\deg(\cL|_{\overline{Z_2 \setminus Z}}) \ge m$, so that $Z$
must be empty.
Similarly, $f_0$ cannot map any node to $\infty_\cP$.
It follows that we may view $f_0$ as a meromorphic section of
$\omega^{\log}$ with a unique pole at a non-special point $q$.
Further, letting $p = Z_1 \cap Z_2$, which is mapped to zero by $f_0$,
we may view $f_0$ as a section $\alpha \in H^0(\omega_{Z_2}(p+(q-p)))$
whose image under the residue map
$\alpha \mapsto \alpha|_q \in \omega_{Z_2}(q)|_q$ is non-zero.
However, in the residue exact sequence
\begin{equation*}
  0 \to H^0(\omega_{Z_2}) \to H^0(\omega_{Z_2}(q)) \to \omega_{Z_2}(q)|_q \to H^1(\omega_{Z_2}) \to H^1(\omega_{Z_2}(q)) = 0,
\end{equation*}
the connecting homomorphism must be a an isomorphism, and $\alpha|_q$
must therefore be zero.
We have arrived at a contradiction.


\section{Cosections and the reduced virtual cycle}
\label{sec:virtual}


\subsection{The logarithmic perfect obstruction theory}
The perfect obstruction theory of stable log maps has been formulated in \cite{GS13, AW} in different but equivalent ways using the log cotangent complexes of \cite{LogCot}. Here we will follow the method of \cite{AW}.

\subsubsection{The canonical perfect relative obstruction theory}\label{ss:log-obs}
Let $\SF:=\SF_{\ddata}^{1/r}$ be the stack of stable $r$-spin curves with a log field with the discrete data $\beta$ as in (\ref{equ:spin-data}).
Let $\beta'$ be the reduced discrete data and $\fM(\cA,\beta')$ be the universal stack, see Section \ref{ss:log-moduli}.
Consider the fiber product in the fine and saturated category
\begin{equation}\label{equ:log-base-stack}
\fM:=\fM(\cA,\beta')\times_{\fM^{\tw}_{g, n}}\fM^{1/r}_{g,\vgamma}
\end{equation}
where the two arrows to $\fM^{\tw}_{g, n}$ are the tautological
morphisms. By Propositions \ref{prop:uni-stack-smooth} and
\ref{prop:curve-stack}, $\fM$ is log smooth and equi-dimensional.

By \eqref{equ:forgetgeometry}, we have the tautological morphism
\[
\SF \to \fM
\]
induced by the associated log maps (\ref{equ:reduce-target}) and the spin structures. This leads to the commutative diagram
\[
\xymatrix{
\cC_{\SF} \ar@/^1pc/[rrd]^{=} \ar[rd]^{f_{\SF}} \ar@/_1pc/[ddr]_{h}&&& \\
& \cP_{\SF} \ar[r] \ar[d] & \cC_{\SF} \ar[r]^{\pi_{\SF}} \ar[d] & \SF \ar[d] \\
& \cP_{\fM} \ar[r] & \cC_{\fM} \ar[r]_{\pi_{\fM}} & \fM
}
\]
where $f_{\SF}\colon \cC_{\SF} \to \cP_{\SF}$ is the universal log
field, $\pi_{\SF}\colon \cC_{\SF} \to \SF$ and
$\pi_{\fM}\colon \cC_{\fM} \to \fM$ are the universal log curves.
Note that the two squares are both cartesian with strict vertical
arrows.

\begin{notation}
We reserve the letter $\LL$ for the log cotangent complexes of
\cite{LogCot}, and the letter $\TT$ for its dual.
For what follows, without further decoration all functors such as $f^*$ and $\pi_*$ are
automatically in the derived sense to simplify the notation.
\end{notation}

Observe that the left and right Cartesian squares imply
\[
f^*_{\SF}\LL_{\cP_{\SF}/\cC_{\SF}} \cong h^*\LL_{\cP_{\fM}/\cC_{\fM}}, \ \ \ \mbox{and} \ \ \ \pi^*_{\SF}\LL_{\SF/\fM} \cong  \LL_{\cC_{\SF}/\cC_{\fM}}
\]
respectively.
Note that there is a map between cotangent complexes induced by $h$:
\[
h^*\LL_{\cP_{\fM}/\cC_{\fM}} \to \LL_{\cC_{\SF}/\cC_{\fM}},
\]
By the commutativity of arrows to $\cC_{\fM}$, we thus obtain the
morphism
\[
f^*_{\SF}\LL_{\cP_{\SF}/\cC_{\SF}} \to \pi^*_{\SF}\LL_{\SF/\fM}.
\]
Since $\cP_{\SF} \to \cC_{\SF}$ is log smooth and integral, we have
$\LL_{\cP_{\SF}/\cC_{\SF}} \cong \Omega_{\cP_{\SF}/\cC_{\SF}}$ the log
cotangent bundle.
Tensoring with the dualizing complex
$\omega_{\cC_\SF/\SF}^\bullet \cong \omega_{\cC_\SF/\SF}[1]$,
we obtain the morphism
\begin{equation*}
  f^*_{\SF}\LL_{\cP_{\SF}/\cC_{\SF}} \otimes \omega_{\cC_\SF/\SF}^\bullet \to \pi^!_{\SF}\LL_{\SF/\fM}.
\end{equation*}
Pushing forward along $\pi_{\SF}$, and using the adjunction morphism
$\pi_{\SF,*} \pi^!_{\SF} \LL_{\SF/\fM} \to \LL_{\SF/\fM}$, we obtain
\begin{equation}\label{equ:log-perfect-obs-cotangent}
\pi_{\SF,*} (f_{\SF}^*\Omega_{\cP_{\SF}/\cC_{\SF}} \otimes \omega_{\cC_\SF/\SF}^\bullet)\to \LL_{\SF/\fM}.
\end{equation}
Since the morphism $\SF \to \fM$ is strict, we have
\[
\LL_{\SF/\fM} \cong \LL_{\underline{\SF}/\underline{\fM}}
\]
where $\LL_{\underline{\SF}/\underline{\fM}}$ is the cotangent complex in the usual sense.

\begin{proposition}
The morphism (\ref{equ:log-perfect-obs-cotangent}) defines a perfect obstruction theory for $\SF \to \fM$ in the sense of Behrend-Fantechi \cite[Section 7]{BF97}.
\end{proposition}
\begin{proof}
  Note that $\pi_{\SF,*}(f_{\SF}^*\Omega_{\cP_{\SF}/\cC_{\SF}} \otimes \omega_{\cC_\SF/\SF}^\bullet)$ is a two-term complex perfect in $[-1, 0]$. It suffices to show that (\ref{equ:log-perfect-obs-cotangent}) is an obstruction theory.

Let $S_0 \to S$ be a strict closed embedding induced by a square-zero ideal. Given a commutative diagram of solid arrows
\[
\xymatrix{
S_0  \ar[r] \ar[d] & \SF \ar[d] \\
S \ar[r] \ar@{-->}[ur] & \fM
}
\]
we want to study the dashed arrow lifting the bottom arrow. Using the
associated log map \eqref{equ:reduce-target}, the above diagram of
solid arrows translates to the left square of the following commutative diagram of solid
arrows, and the dashed arrow translates to the dashed arrow below:
\[
\xymatrix{
\cC_{S_0} \ar[r] \ar[d] & \cP_{\cC_{S}} \ar[d] \ar@/^1pc/[rd] & \\
\cC_{S} \ar[r] \ar@{-->}[ur] & \cA\times\underline{\cC}_{S} \ar[r] & \Log\times\underline{\cC}_{S}
}
\]
Here $\Log$ is Olsson's stack parameterizing log structures \cite[Section 1]{Ol03}. Note that all the arrows in the left triangle are strict and smooth. Furthermore, the left bottom arrow is \'etale \cite[Corollary 5.24]{Ol03}. Thus, we have $T_{\cP_{\cC_S}/\cC_S} = T_{\cP_{\cC_S}/\Log\times\underline{\cC}_{S}} \cong T_{\cP_{\cC_{S}}/\cA\times\underline{\cC}_S}$. Now the statement follows from the same argument as in \cite[Section 6.1]{AW} using \cite{Wi11}.
\end{proof}

Observe that the above construction of perfect obstruction theory is compatible with arbitrary base changes.

\begin{lemma}\label{lem:obs-base-change}
For any morphism $S \to \fM$, consider $\SF_{S} := S\times_{\fM}\SF$ with the pullback $f_{\SF_S}\colon \cC_{\SF_S} \to \cP_{\SF_S}$. Then the perfect obstruction theory (\ref{equ:log-perfect-obs-cotangent}) of $\SF \to \fM$ pulls back to a perfect obstruction theory
\[
\pi_{\SF_S,*} (f_{\SF_S}^*\Omega_{\cP_{\SF_S}/\cC_{\SF_S}} \otimes \omega_{\cC_{\SF_S}/\SF_S}^\bullet) \to \LL_{\SF_S/S}.
\]
of the strict morphism $\SF_S \to S$.
\end{lemma}

\subsubsection{The case of maps with uniform maximal degeneracy}

Replacing $\fM(\cA,\beta')$ by $\fU(\cA,\beta')$ in (\ref{equ:log-base-stack}), we obtain
\begin{equation}\label{equ:log-base-max-deg}
\fU:=\fU(\cA,\beta')\times_{\fM^{\tw}_{g, n}}\fM^{1/r}_{g,\vgamma}
\end{equation}
By Theorem \ref{prop:uni-stack-smooth}, the natural projection
$
\fU \to \fM
$
is log \'etale. Thus $\fU$ is log smooth and equi-dimensional.

Now consider $\USF := \USF_{\ddata}^{1/r}$ and the universal log field $f_{\USF}\colon \cC_{\USF} \to \cP_{\USF}$ over $\USF$. Since $\USF = \fU\times_{\fM}\SF$, applying Lemma \ref{lem:obs-base-change}, we obtain a relative perfect obstruction theory
\[
\pi_{\USF,*} (f_{\USF}^*\Omega_{\cP_{\USF}/\cC_{\USF}} \otimes \omega_{\cC_\USF/\USF}^\bullet) \to \LL_{\USF/\fU}.
\]

Taking the dual of the above morphism and using $T_{\cP_{\USF}/\cC_{\USF}} = \Omega^{\vee}_{\cP_{\USF}/\cC_{\USF}}$, we obtain
\begin{equation}\label{equ:log-perfect-obs}
\TT_{\USF/\fU} \to \pi_{\USF,*}f_{\USF}^*T_{\cP_{\USF}/\cC_{\USF}} =: \EE_{\USF/\fU}.
\end{equation}
The following result will be useful for later calculation.

\begin{lemma}\label{lem:pull-back-log-tangent}
  Let $(\cC \to S, \cL, f\colon \cC \to \cP)$ be an $r$-spin curve
  with a log field over a log scheme $S$.
  Then we have
\[
f^*T_{\cP/\cC} \cong f^*(\cO_{\cP}(0_{\cP}))  \cong \cL\otimes f^*(\cO_{\cP}(\infty_{\cP})).
\]
\end{lemma}
\begin{proof}
  Note that the usual tangent bundle is
  $T_{\underline{\cP}/\underline{\cC}} = \cO_{\cP}(0_{\cP} +
  \infty_{\cP})$.
  The log tangent bundle
  $T_{\cP/\cC} \subset T_{\underline{\cP}/\underline{\cC}}$ is the
  subsheaf consisting of vector fields vanishing along
  $\infty_{\cP}$. Thus we have $T_{\cP/\cC} = \cO_{\cP}(0_{\cP})$
  which proves the first equality.

  The second equality follows from the observation
  that
  \begin{equation}
    \label{eq:hirzebruch}
    \cO_{\cP}(0_{\cP} - \infty_{\cP}) \cong \cL|_{\cP},
  \end{equation}
  where $\cL|_{\cP}$ is the pullback of $\cL$ along the morphism
  $\cP \to \cC$.
\end{proof}

\subsection{The relative cosection}\label{ss:extending-cosection}

Consider the universal log field
\[
f_{\USF}\colon \cC_{\USF} \to  \cP_{\USF}
\]
over $\USF := \USF_{\ddata}^{1/r}$. Denote by $\pi_{\USF}\colon \cC_{\USF} \to \USF$ the projection, and $\cL_{\USF}$ the universal $r$-spin bundle over $\cC_{\USF}$.

Throughout the rest of Section~\ref{sec:virtual}, we impose the
following condition on the discrete data, which is necessary for the
cosection construction.
\begin{assumption}\label{ass:narrow}
  All marked points are narrow and have zero contact order.
\end{assumption}

\begin{notation}
  For a locally free sheaf $V$ over a log stack $X$, write
  $\vb(V) = \spec(\mathrm{Sym}^\bullet V^\vee)$ to be the geometric
  vector bundle associated to $V$ with the strict morphism
  $\vb(V) \to X$.
  For any morphism $Y \to X$, denote by $V|_{Y}$ and $\vb(V)|_{Y}$ the
  pullbacks of $V$ and $\vb(V)$ respectively.
\end{notation}

\subsubsection{The twisted spin section}

Consider the canonical inclusion
\begin{equation}\label{equ:P-zero-section}
\iota\colon \cO_{\cP_{\USF}} \to \cO_{\cP_{\USF}}(0_{\cP}).
\end{equation}
Using the isomorphism
\[
\cO_{\cP_{\USF}}(0_{\cP}) \cong \cL_{\USF}|_{\cP_{\USF}} \otimes \cO_{\cP_{\USF}}(\infty_{\cP})
\]
from \eqref{eq:hirzebruch}, we obtain
\begin{equation}\label{equ:underlying-twist}
f^*_{\USF}\iota\colon \cO_{\cC_{\USF}} \to \cL_{\USF}\otimes f^*_{\USF}\cO_{\cP_{\infty}}(\infty_{\cP}).
\end{equation}

By Assumption \ref{ass:narrow}, we may pull back \eqref{equ:log-twist} via $\USF \to \fU$ and obtain
\begin{equation}\label{equ:log-twist-spin}
\tf_{\USF}\colon \pi^*_{\USF}\bL_{\max}\otimes f^*_{\USF}\cO(\infty_{\cP}) \to \cO_{\cC_\USF}.
\end{equation}
By abuse of notation, $\bL_{\max}$ denotes the pullback of the
corresponding line bundle over $\fU$.
Using Lemma~\ref{lem:pull-back-log-tangent} we obtain a morphism
\begin{equation}\label{equ:tangent-twist}
f_{\USF}^*T_{\cP_{\USF}/\cC_{\USF}} \cong \cL_{\USF}\otimes f^*_{\USF} \cO(\infty_{\cP}) \stackrel{\otimes\tf^{\vee}_{\USF}}{\longrightarrow} \cL_{\USF} \otimes \pi^*_{\USF} \bL_{\max}^{\vee}.
\end{equation}
Composing with (\ref{equ:underlying-twist}), we obtain the {\em twisted spin section}
\begin{equation}\label{equ:twisted-spin-section}
\bs_{\USF} := (\otimes\tf^{\vee}_{\USF})\circ (f^*_{\USF}\iota)\colon  \cO_{\cC_{\USF}} \to \cL_{\USF}\otimes \pi^*_{\USF}\bL_{\max}^{\vee},
\end{equation}
or equivalently a morphism $\bs_{\USF}\colon \cC_{\USF} \to \vb(\cL_{\USF}\otimes \pi^*_{\USF}\bL_{\max}^{\vee}).$

\subsubsection{The twisted superpotential and its differentiation}
Write for simplicity
\[
\omega_{\log,\USF} := \omega^{\log}_{\cC_{\USF}/\USF} \ \ \ \mbox{and} \ \ \   \omega_{\USF} := \omega_{\cC_{\USF}/\USF}.
\]

The $r$-spin structure $\cL_{\USF}^{r}\cong \omega_{\log,\USF}$ defines an isomorphism
\[
(\cL_{\USF}\otimes \pi^*_{\USF}\bL_{\max}^{\vee})^{r} \cong \omega_{\log,\USF}\otimes \pi^*_{\USF}\bL_{\max}^{-r}
\]
hence a non-linear morphism
\[
W\colon \vb(\cL_{\USF}\otimes \pi^*_{\USF}\bL^{\vee}_{\max}) \to \vb(\omega_{\log,\USF}\otimes \pi^*_{\USF}\bL_{\max}^{-r})
\]
called the \emph{twisted superpotential}. For convenience, we may equip both the source and the target of $W$ with the log structures pulled back from $\cC_{\USF}$. In particular, $W$ is a strict morphism. Differentiating $W$, we have
\[
\diff W\colon T_{\vb(\cL_{\USF}\otimes \pi^*_{\USF}\bL^{\vee}_{\max})/\cC_{\USF}} \to W^* T_{\vb(\omega_{\log,\USF}\otimes \pi^*_{\USF}\bL_{\max}^{-r})/\cC_{\USF}}.
\]
Pulling back $\diff W$ via the twisted spin section \eqref{equ:twisted-spin-section} gives
\begin{equation}\label{equ:before-twist}
\bs_{\USF}^*(\diff W) \colon \cL_{\USF}\otimes \pi^*_{\USF}\bL_{\max}^{\vee} \to \omega_{\log,\USF}\otimes \pi^*_{\USF}\bL_{\max}^{-r}.
\end{equation}
Pushing forward along $\pi_{\USF}$ and applying the projection formula, we have
\begin{equation}\label{equ:push-diff-potential}
\pi_{\USF,*}\bs_{\USF}^*(\diff W)\colon \pi_{\USF,*}(\cL_{\USF})\otimes \bL_{\max}^{\vee} \to \pi_{\USF,*}(\omega_{\log,\USF})\otimes \bL_{\max}^{-r}.
\end{equation}

Denote by $\Sigma := \sum_i \sigma_i$ the sum of marked points of
$\cC_{\USF}$. Since all the markings are narrow, recall from \cite[Lemma 3.2]{CLL15}
that the push-forward of the natural inclusion
$\cL_{\USF}(-\Sigma) \hookrightarrow \cL_{\USF}$ is an isomorphism
\begin{equation}\label{equ:narrow-iso}
\pi_{\USF,*}\cL_{\USF}(-\Sigma) \stackrel{\cong}{\longrightarrow} \pi_{\USF,*}\cL_{\USF}.
\end{equation}

Twisting down (\ref{equ:before-twist}) by the markings and pushing forward, we have
\begin{equation}\label{equ:after-twist}
\pi_{\USF,*}\bs_{\USF}^*(\diff W)(-\Sigma) \colon \pi_{\USF,*}\big(\cL_{\USF}(-\Sigma)\big)\otimes\bL_{\max}^{\vee} \to \pi_{\USF,*}\omega_{\USF}\otimes\bL_{\max}^{-r}.
\end{equation}
The two morphisms (\ref{equ:push-diff-potential}) and (\ref{equ:after-twist}) fit in a commutative diagram
\begin{equation}\label{diag:differentiate-potential}
\xymatrix{
\pi_{\USF,*}\big(\cL_{\USF}(-\Sigma)\big)\otimes\bL_{\max}^{\vee} \ar[d]_{\cong} \ar[rrr]^{\pi_{\USF,*}\bs_{\USF}^*(\diff W)(-\Sigma)} &&& \pi_{\USF,*}\omega_{\USF}\otimes\bL_{\max}^{-r} \ar[d] \\
\pi_{\USF,*}(\cL_{\USF})\otimes \bL_{\max}^{\vee} \ar[rrr]^{\pi_{\USF,*}\bs_{\USF}^*(\diff W)} &&& \pi_{\USF,*}(\omega_{\log,\USF})\otimes \bL_{\max}^{-r}
}
\end{equation}
where the vertical arrows are induced by twisting down by $\Sigma$, and the left vertical arrow is the isomorphism (\ref{equ:narrow-iso}).

\bigskip

We give a point-wise description of $R^1\pi_{\USF,*}\bs_{\USF}^*(\diff W)(-\Sigma)$. Consider a geometric point $w \to \USF$ with the pullback $(\cC_w/w, \cL_{w}, f_w)$. Using Serre duality and the spin structure $\cL_{w}^{r} \cong \omega_{\log, w}$, we have
\begin{align*}
  H^1\big(\cL_{w}(-\Sigma)\big)\otimes{\bL_{\max,w}^{\vee}} &\cong H^{0}\big( \cL^{\vee}_{w}(\Sigma)\otimes\omega_{w}\big)^\vee \otimes {\bL_{\max,w}^{\vee}}\\
                                                            &\cong H^{0}\big( \cL^{\vee}_{w}\otimes\omega_{\log,w}\big)^\vee \otimes {\bL_{\max,w}^{\vee}} \\
                                                            &\cong H^0(\cL_{w}^{r-1})^{\vee}\otimes\bL^{\vee}_{\max,w}.
\end{align*}
The fiber $R^1\pi_{\USF,*}\bs_{\USF}^*(\diff W)(-\Sigma)_w$ is then given by
\begin{equation}\label{equ:fiber-diff-W}
H^0(\cL_{w}^{r-1})^{\vee}\otimes {\bL}^{\vee}_{\max, w} \to {\bL}^{-r}_{\max, w}, \ \ \  {\dot{s}} \mapsto {r} \bs_{w}^{r-1} \cdot {\dot{s}}.
\end{equation}
where $\bs_w$ is the fiber of (\ref{equ:twisted-spin-section}) at $w$,
hence
${\bs}_{w}^{r-1} \in H^0(\cL_{w}^{r-1})\otimes {\bL}^{1 - r}_{\max, w}$.

\subsubsection{The relative cosection}

Pushing forward (\ref{equ:tangent-twist}), we obtain
\begin{equation}\label{equ:push-tangent-twist}
\EE_{\USF/\fU} \to \pi_{\USF,*}(\cL_{\USF}) \otimes \bL_{\max}^{\vee}.
\end{equation}
Composing with the left and top arrows of (\ref{diag:differentiate-potential}), we obtain
\begin{equation}\label{equ:derived-cosection}
\EE_{\USF/\fU} \to \pi_{\USF,*}\omega_{\USF}\otimes\bL_{\max}^{-r},
\end{equation}
whose $H^1$ defines the {\em relative cosection}
\begin{equation}\label{equ:relative-cosection}
\sigma_{\USF/\fU}\colon \obs_{\USF/\fU} := H^1(\EE_{\USF/\fU}) \to R^1\pi_{\USF,*}\omega_{\USF}\otimes\bL_{\max}^{-r} \cong \bL_{\max}^{-r}.
\end{equation}

By abuse of notation, denote by $\Delta_{\max} \subset \USF$ the pre-image of $\Delta_{\max}\subset \fU$. Let $\USF^{\circ} := \USF \setminus \Delta_{\max}$. Then $\USF^{\circ}$ is the stack parameterizing sections of the $r$-spin bundle. Note that $\USF^{\circ}$ is the stack $X$ as in \cite[Section 3]{CLL15}.

\begin{lemma}\label{lem:old-compatible}
In the $r$-spin case, the restriction of \eqref{equ:log-perfect-obs} to $\USF^{\circ}$ is the perfect obstruction theory in \cite[(3.2)]{CLL15}, and $\sigma_{\USF/\fU}|_{\fU^{\circ}}$ is the relative cosection in \cite[(3.5)]{CLL15}.
\end{lemma}
\begin{proof}
By assumption, we have that
\[
f^*_{\USF^{\circ}}\cO(\infty_{\cP}) = \cO_{\cC_{\USF^{\circ}}} \ \ \ \mbox{and} \ \ \ \bL_{\max}|_{\USF^{\circ}} = \cO_{\USF^{\circ}}.
\]
Then the statement follows from the construction of $\sigma_{\USF/\fU}$.
\end{proof}

\subsubsection{Surjectivity of $\sigma_{\USF/\fU}$ along the boundary}

\begin{lemma}\label{lem:narrow-marking-vanishing}
Suppose a narrow marking has the trivial contact order. Then its image via $f_{\USF}$ is contained in the zero section $0_{\cP} \subset \cP_{\USF}$.
\end{lemma}
\begin{proof}
  We first show that the images of narrow markings avoid the
  infinity section.
  Since this is an open condition, it suffices to check this over a
  geometric fiber $w \to \USF$ with the log twisted field
  $f_w\colon \cC_w \to \cP_w$.

  Suppose $f_w(\sigma_i) \in \infty_{\cP}$ for a narrow marking
  $\sigma_i \in \cC_w$.
  Then there is an irreducible Zariski neighborhood $V \subset \cC_w$
  of $p$ such that
  $\oM_{\cC_w} \cong \pi_w^*\oM_{w}\oplus \sigma_{i,*}\NN_{\sigma_i}$,
  see Section~\ref{sss:log-twisted-curve}.
  Since the contact order of $\sigma_i$ is trivial, $f^{\flat}_w|_{V}$
  induces a morphism of $\cO^*$-torsors over $V$ of the form
  $ f^{\flat}_w|_{V}\colon f^*_{w}\cT_{\infty_{\cP}}|_{V} \to
  \cT_{e_V} $ where $e_V \in \pi_w^*\oM_{w}$ is the degeneracy of
  $f_w$ along $V$,
  $\cT_{e_V} = e_V \times_{\oM_{\cC_w}} \cM_{\cC_w}$, and
  $\cT_{\infty_{\cP}} \subset \cM_{\cP_{w}}$ is the preimage of the
  torsor $\cT_{\infty}$ as in \eqref{equ:target-torsor} via
  $\cP_{w} \to \cA$.
  Taking the corresponding line bundles, we obtain a morphism
  $f^*\cO(-\infty_{\cP})|_{V} \to \cO_{V}$ whose dual
  $\cO_V \to f^*\cO(\infty_{\cP})|_{V}$ is non-vanishing at $p$ since
  the contact order at $p$ is trivial.
  Since $f^*\cO(\infty_{\cP})|_{V} \cong \cL^{\vee}_w|_{V}$, we
  obtain a local section of $\cL^{\vee}_w$ non-vanishing at a
  narrow marking.
  But by \cite[Lemma 3.2]{CLL15} and \cite[Proposition 3.0.3]{AbJa03}
  such a local section vanishes at $\sigma_i$.
  This is a contradiction.

Since $f_{\USF}(\sigma_i)$ avoids $\infty_{\cP}$, locally $f_{\USF}$ is a section of $\cL_{\USF}$  around $\sigma_{i}$, hence vanishes along $\sigma_{i}$. This completes the proof.
\end{proof}

We next prove the surjectivity of $\sigma_{\USF/\fU}$ along $\Delta_{\max}$.

\begin{proposition}\label{prop:cosection-boundary-surjective}
The vanishing locus  $(\sigma_{\USF/\fU} = 0) \subset \USF$ is given by the locus along which $f_{\USF}$ is the zero section.
\end{proposition}
\begin{proof}
By Lemma \ref{lem:old-compatible} and \cite[Lemma 3.6]{CLL15}, $\sigma_{\USF/\fU}|_{\USF^{\circ}}$ vanishes along the locus where $f_{\USF}$ is the zero section. It remains to show that $\sigma_{\USF/\fU}$ is surjective along $\Delta_{\max}$. Since $\bL_{\max}^{-r}$ is a line bundle, the image of $\sigma_{\USF/\fU}$ is a torsion-free sub-sheaf of $\bL_{\max}^{-r}$. Thus it suffices to show that $\sigma_{\USF/\fU}$ is surjective at each geometric point of $\Delta_{\max}$.

Let $w \in \Delta_{\max}$ be a geometric point with $(\cC_w/w, \cL_{w}, f_w)$. Taking $H^1$ of (\ref{equ:tangent-twist}) over $w$, we have
\begin{equation}\label{equ:H1-tangent-twist}
H^{1}\big(\cL_{w}\otimes f^*_{w} \cO(\infty_{\cP_w})\big) \to H^{1}(\cL_{w}) \otimes {\bL}_{\max, {w}}^{\vee}.
\end{equation}
By construction $\sigma_{\USF/\fU,w}$ is the composition
\begin{align*}
H^{1}\big(\cL_{w}\otimes f^*_{w} \cO(\infty_{\cP_w})\big) & \stackrel{}{\longrightarrow}  H^{1}(\cL_{w}) \otimes {\bL}_{\max, {w}}^{\vee} \\
 \big(\mbox{by \eqref{equ:narrow-iso}} \big)       \ \ \ \ \        & \stackrel{\cong}{\longleftarrow}  H^1\big(\cL_{w}(-\Sigma)\big)\otimes{\bL_{\max,w}^{\vee}} \\
   \big( \mbox{by \eqref{equ:after-twist}} \big)       \ \ \ \ \          & \stackrel{}{\longrightarrow} {\bL}^{-r}_{\max, w}
\end{align*}
where the first arrow is (\ref{equ:H1-tangent-twist}). Applying Serre duality and taking the dual, we have $\sigma^{\vee}_{\USF/\fU,w}$:
\begin{align*}
H^0\big(\cL^{\vee}_{w}\otimes f^*_{w} \cO(-\infty_{\cP_w})\otimes\omega_w\big) &\longleftarrow H^{0}(\cL^{\vee}_{w}\otimes\omega_{w})\otimes {\bL}_{\max, {w}}\\
&\stackrel{\cong}{\longrightarrow} H^0(\cL_{w}^{r-1})\otimes\bL_{\max,w} \\
&\stackrel{}{\longleftarrow} \bL^r_{\max, w}.
\end{align*}
where the first and last arrows are given by the duals of (\ref{equ:H1-tangent-twist}) and (\ref{equ:fiber-diff-W}) respectively. We describe $\sigma^{\vee}_{\USF/\fU,w}$ via the above composition as follows.

Suppose $v_0 \in \bL^r_{\max, w}$ is a non-zero vector. Applying the dual of (\ref{equ:fiber-diff-W}), we obtain a vector 
\[
v_1:=(r \bs_{w}^{r-1})^{\vee}(v_0) \in H^0(\cL_{w}^{r-1}\otimes\bL_{\max,w}) = H^0(\cL^{\vee}_{w}\otimes\omega^{\log}_{w}\otimes {\bL}_{\max, {w}}).
\] 
We observe that $v_1$ is non-trivial along the sub-curve $Z \subset \cC_w$ consisting of maximally degenerate components, and vanishes along $\cC_w \setminus Z$. Indeed, $\bs_{w}$ is the fiber of $\bs_{\USF}$ in \eqref{equ:twisted-spin-section} which is defined as the composition of a section $(f^*_{\USF}\iota)$ vanishing only along components of $\cC_w$ with images in $0_{\cP_w}$ by \eqref{equ:P-zero-section}, followed by tensoring the section $\tf_{\USF}^{\vee}$ in \eqref{equ:log-twist-spin} vanishing only along the closure of $\cC_w \setminus Z$ in $\cC_w$ by Proposition \ref{prop:log-twist}. The observation then follows since $Z$ has image entirely in $\infty_{\cP_w}$.

We further observe that $Z$ contains no markings by Lemma \ref{lem:narrow-marking-vanishing}. Thus the above paragraph implies that $v_1 \in H^0\big(\cL^{\vee}_{w}\otimes\omega_{w}\otimes {\bL}_{\max, {w}}\big)$ as $v_1$ vanishes along all markings.

Finally observe that the dual of  (\ref{equ:H1-tangent-twist}) is given by
\[
\cL^{\vee}_{w}\otimes\omega_{w}\otimes {\bL}_{\max, {w}} \stackrel{\otimes\tf_{\USF}}{\longrightarrow} \cL^{\vee}_{w}\otimes f^*_{w} \cO(-\infty_{\cP_w})
\]
hence $\sigma^{\vee}_{\USF/\fU,w}(v_0) = v_1\otimes\tf_{\USF}$. By Proposition \ref{prop:log-twist} again, $\tf_{\USF}$ hence $v_1\otimes\tf_{\USF}$ is non-trivial along $Z$. In particular, $\sigma^{\vee}_{\USF/\fU,w}(v_0) \neq 0$.

The above analysis implies that $\sigma^{\vee}_{\USF/\fU,w}$ is injective, hence $\sigma_{\USF/\fU,w}$ is surjective. This completes the proof.
\end{proof}

\subsection{Factorization of the relative obstruction}\label{ss:cosection-factorization}

\subsubsection{An auxiliary twist}\label{sss:auxiliary-twist}

Denote by $\cL_{\fU,-}:= \cL_{\fU}(-\Sigma)$ for simplicity, where
$\Sigma$ still denotes the sum of markings of $\cC_\fU$.
Similar to the construction of $\cP_{\fU}$ in Section \ref{sss:spin-field}, we formulate the stack $\cP_{\fU,-}$ with $\cL_{\fU}$ replaced by $\cL_{\fU,-}$.
The log structure on $\cP_{\fU,-}$ is defined as
\[
\cM_{\cP_{\fU,-}} := \cM_{\cC_{\fU}}|_{\cP}\oplus_{\cO^*}\cM_{\infty_{\cP_{-}}}
\]
where $\infty_{\cP_{-}} \subset \cP_{\fU,-}$ is the corresponding
infinity section.
The natural morphism $\vb(\cL_{\fU,-}) \to \vb(\cL_{\fU})$ induces a
birational map of log stacks
$\xymatrix{\cP_{\fU,-} \ar@{-->}[r] & \cP_{\fU}}$ which is an
isomorphism away from fibers over marked points.
Denote by $\cP_{\fU,reg} \subset \cP_{\fU,-}$ the open sub-stack where
the above rational map is well-defined.
Let $\bt\colon \cP_{\fU,reg} \to \cP_{\fU}$ be the corresponding
morphism. Denote by $\cP_{\USF,reg}$ the pullback of $\cP_{\fU,reg}$
with the corresponding morphism
$\bt\colon \cP_{\USF,reg} \to \cP_{\USF}$.

\begin{lemma}
There is a canonical factorization
\[
\xymatrix{
\cC_{\USF} \ar[rd]_{f_{\USF,-}} \ar[rr]^{f_{\USF}} && \cP_{\USF} \\
 &\cP_{\USF,reg} \ar[ru]_{\bt}&
}
\]
\end{lemma}
\begin{proof}
Note that $\cP_{\fU,reg} \subset \cP_{\fU,-}$ is obtained by removing the fiber of  $\infty_{\cP_{-}}$ over marked points. The statement follows from Lemma \ref{lem:narrow-marking-vanishing}.
\end{proof}

Denote by $\cP_{\fU,-} \to \cA$ the morphism of log stacks such that $\cM_{\infty_{\cP_{-}}}$ is the pullback of $\cM_{\cA}$. Consider the natural morphism
$
\USF \to \fU
$
induced by the composition of $f_{\USF,-}$ with $\cP_{\fU,-} \to \cA$. The above lemma implies that $\USF$ can be viewed as the log stack parameterizing log twisted sections $f_{T,-}\colon \cC_{T} \to \cP_{T,-}$ for any  $T \to \fU$. The same construction in Section \ref{ss:log-obs} provides a perfect obstruction theory of $\USF \to \fU$:
\begin{equation}\label{equ:-perfect-obs}
\TT_{\USF/\fU} \to \pi_{\USF,*}f^*_{\USF,-}T_{\cP_{\USF,reg}/\cC_{\USF}} \cong \pi_{\USF,*}f^*_{\USF,-}T_{\cP_{\USF,-}/\cC_{\USF}} =: \EE_{\USF/\fU, -}.
\end{equation}

On the other hand since $f^{*}_{\USF}\cO(\infty_{\cP}) \cong f^*_{\USF,-}\cO(\infty_{\cP,-})$, we calculate
\begin{align}
f^*_{\USF,-}T_{\cP_{\USF,-}/\cC_{\USF}} &\cong \cL_{\USF,-} \otimes f^*_{\USF,-}\cO(\infty_{\cP,-}) \nonumber \\
     &\cong \cL_{\USF}(-\Sigma)\otimes f^{*}_{\USF}\cO(\infty_{\cP}) \label{equ:twist-log-tangent} \\
     &\cong f^*_{\USF}T_{\cP_{\USF}/\cC_{\USF}}(-\Sigma).\nonumber
\end{align}
Using (\ref{equ:narrow-iso}) and Lemma \ref{lem:narrow-marking-vanishing}, we have
\[
\pi_{\USF,*}f^*_{\USF,-}T_{\cP_{\USF,reg}/\cC_{\USF}} \cong \pi_{\USF, *}f^*_{\USF}T_{\cP_{\USF}/\cC_{\USF}}.
\]
To summarize:
\begin{lemma}\label{lem:twist-obs}
The two perfect obstruction theories (\ref{equ:log-perfect-obs}) and (\ref{equ:-perfect-obs}) are identical.
\end{lemma}
We now view $\USF$ with the universal family $f_{\USF,-}\colon \cC_{\USF} \to \cP_{\USF,-}$.

\subsubsection{Partial expansion and contraction}

The morphism $\fm\colon \fU(\cA,\beta') \to \cA_{\max}$ from Section \ref{ss:partial-expansion} induces a morphism $\fU \to \cA_{\max}$ which will again be denoted by $\fm$ by abuse of notation. Consider the cartesian diagram of fine log stacks
\begin{equation}\label{equ:bundle-partial-expansion}
\xymatrix{
\cP^{e}_{\fU,-} \ar[r] \ar[d]_{\fb} & \cA^{e} \ar[d]^{\fb} \\
\cP_{\fU,-} \ar[r] & \cA \times \cA_{\max}
}
\end{equation}
where the bottom is the product of $\fm$ and $\cP_{\fU,-} \to \cA$.

By construction, one checks that the bottom arrow satisfies the flatness conditions in \cite[Proposition (4.1)]{KKato}, hence is integral in the sense of \cite[Definition (4.3)]{KKato}. In particular, the underlying structure of the above cartesian diagram is a cartesian diagram of the underlying algebraic stacks. We remark that the above diagram is indeed cartesian in the fine and saturated category. Since the saturation plays no role in the following discussion, we omit the details here.

In the above diagram, since the right vertical arrow is log \'etale,
the left vertical arrow is again log \'etale.
By abuse of notation, we denote both vertical arrows by $\fb$.
Let $\infty_{\cP^{e}_{-}} \subset \cP^{e}_{\fU,-}$ be the pre-image of
$\infty_{\cA^e} \subset \cA^e$, and write
$\cP^{e}_{\fU,-} := \cP^{e,\circ}_{\fU,-}\setminus
\infty_{\cP^{e}_{-}}$.
Denote by $\cE_{\fb} \subset \cP^{e}_{\fU,-}$ the exceptional divisor
contracted by $\fb$.
In the following, we view the (relative) normal bundle
$\cN_{\infty_{\cP^e_-}}$ of $\infty_{\cP^e_-}$ as
a line bundle over $\cC_{\fU}$:

\begin{lemma}
$
\cN^{\vee}_{\infty_{\cP^e_{-}}} \cong \cL_{\fU,-}\otimes \pi^*_{\fU}\bL^{\vee}_{\max}.
$
\end{lemma}
\begin{proof}
  Observe that
  $\fb^*[\infty_{\cP_{-}}] = [\infty_{\cP^e_{-}}] + [\cE_{\fb}]$ where
  $[*]$ denotes the corresponding divisor class.
  Pulling back to $\cC_{\fU}$ via the identification
  $\cC_{\fU} \cong \infty_{\cP_{-}} \cong \infty_{\cP^e_{-}}$, we
  obtain
  \[
    \cO(\infty_{\cP_{-}})|_{\infty_{\cP_{-}}} \cong \big(\cO(\infty_{\cP^e_{-}})\otimes \cO(\cE_{\fb})\big)|_{\infty_{\cP^e_{-}}}.
  \]
  Using
  $\cO(\infty_{\cP_{-}})|_{\infty_{\cP_{-}}} \cong
  \cN_{\infty_{\cP_{-}}} \cong \cL^{\vee}_{\fU,-}$ and
  $\cO(\infty_{\cP^e_{-}})|_{\infty_{\cP^e_{-}}} \cong
  \cN_{\infty_{\cP^e_{-}}}$, we obtain
  \[
    \cL^{\vee}_{\fU,-} \cong \cN_{\infty_{\cP^e_{-}}}\otimes \cO(\cE_\fb)|_{\infty_{\cP^e_{-}}}.
  \]

  Finally, observe that
  $\cO(\cE_{\fb})|_{\infty_{\cP^e_{-}}} \cong
  \pi^*_{\fU}\bL_{\max}^{\vee}$, which leads to the desired
  isomorphism.
\end{proof}

\begin{lemma}\label{lem:contraction}
There is a commutative diagram of log stacks
\[
\xymatrix{
\cP_{\fU,-}^{e,\circ} \ar[rr]^-{\fc} \ar[rd] && \vb(\cL_{\fU,-}\otimes \pi^*_{\fU}\bL^{\vee}_{\max}) \ar[ld] \\
&\cC_{\fU}&
}
\]
where $\fc$ is a birational morphism contracting $\cE_{\fc}$, the
proper transform of $\cP_{\fU,-}\times_{\fU}\Delta_{\max}$, to
the zero section of
$\vb(\cL_{\fU,-}\otimes \pi^*_{\fU}\bL^{\vee}_{\max})$.
\end{lemma}
\begin{proof}
Note that once the underlying morphism of $\fc$ is defined, the morphism on the level of log structures is automatically obtained since the right skew arrow is strict.
We may assume for simplicity that all stacks in the rest of this proof have the trivial log structure.

Note that $[\infty_{\cP^e_{-}}]$ is a relative nef divisor of the
family of nodal rational curves $\cP^e_{\fU,-} \to \cC_{\fU}$.
Let $\bar{\fc}\colon \cP^e_{\fU,-} \to \cP^{c}_{\fU,-}$ be the induced
contraction, and $\cE_{\fc} \subset \cP^e_{\fU,-}$ be the exceptional
locus contracted by $\bar{\fc}$.
Then $\cE_{\fc}$ is the proper transform of
$\cP_{\fU,-}\times_{\fU}\Delta_{\max}$.
Observe that the resulting projection $\cP^{c}_{\fU,-} \to \cC_{\fU}$
is again a smooth $\PP^1$-fibration since the contracted locus
consists of a family of $(-1)$-curves over $\cC_{\fU}$.

Furthermore, note that $\bar{\fc}$ induces an embedding $\cN_{\infty_{\cP^e_{-}}} \to \cP^{c}_{\fU}$ over $\cC_{\fU}$ with complement $0_{\cP^{c}_{-}} = \cP^{c}_{\fU,-}\setminus \cN_{\infty_{\cP^e_{-}}}$ given by the image of the zero section $0_{\cP^e_{-}} \subset \cP^e_{\fU,-}$. We thus obtain
\[\cP^c_{\fU,-} \cong \PP(\cN_{\infty_{\cP^e_{-}}}\oplus \cO) \cong \PP(\cO \oplus \cN^{\vee}_{\infty_{\cP^e_{-}}}).\]
Thus $\fc$ is obtained from $\bar{\fc}$ by removing $\infty_{\cP^e_{-}}$ and its image in $\cP^c_{\fU,-}$.
\end{proof}

Consider the canonical morphism induced by the divisor $\cE_{\fc}$
\begin{equation}\label{equ:C-twist}
\iota_{\cE_{\fc}}\colon \cO_{\cP^e_{\fU,-}} \to \cO_{\cP^e_{\fU,-}}(\cE_{\fc}) \cong \bL_{\max}^{\vee}(-\cE_{\fb}),
\end{equation}
and the morphism of log tangent bundles
\[
\diff\fc\colon T_{\cP^{e,\circ}_{\fU,-}/\cC_{\fU}} \to \fc^*T_{\vb(\cL_{\fU,-}\otimes \pi^*_{\fU}\bL^{\vee}_{\max})/\cC_{\fU}}.
\]

\begin{lemma}\label{lem:diff-c}
$\diff\fc = \otimes\iota_{\cE_{\fc}}$.
\end{lemma}
\begin{proof}
Consider the morphism of log cotangent bundles
\[
(\diff\fc)^\vee\colon \fc^* \Omega_{\vb(\cL_{\fU,-}\otimes \pi^*_{\fU}\bL^{\vee}_{\max})/\cC_{\fU}} \to \Omega_{\cP^{e,\circ}_{\fU,-}/\cC_{\fU}}.
\]
Note that $\fc$ is an isomorphism away from the divisor $\cE_{\fc}$.
Furthermore, the contraction $\fc$ is the blow-up of the zero section
of $\vb(\cL_{\fU,-}\otimes \pi^*_{\fU}\bL^{\vee}_{\max})$.
A local coordinate calculation shows that $(\diff\fc)^\vee$ is given by the composition
\[
\fc^* \Omega_{\vb(\cL_{\fU,-}\otimes \pi^*_{\fU}\bL^{\vee}_{\max})/\cC_{\fU}} \stackrel{\otimes\iota^{\vee}_{\cE_{\fc}}}{\longrightarrow} \fc^*\cL_{\fU,-}^{\vee}(-\cE_{\fb}) \stackrel{\cong}{\longrightarrow} \fb^*\Omega_{\cP_{\fU,-}/\cC_{\fU}}(-\cE_{\fb}) \hookrightarrow  \fb^*\Omega_{\cP_{\fU,-}/\cC_{\fU}}
\]
where the first arrow follows from \eqref{equ:C-twist}.
Note that
$\fb^*\Omega_{\cP_{\fU,-}/\cC_{\fU}} \cong
\Omega_{\cP^{e,\circ}_{\fU,-}/\cC_{\fU}}$ since $\fb$ is log \'etale
over $\cC_{\fU}$.
This means that $(\diff\fc)^\vee = \otimes \iota_{\cE_{\fc}}^{\vee}$.
Taking the dual, we obtain the desired equality.
\end{proof}

\subsubsection{Twisted spin section via partial expansion}

Consider the commutative diagram of solid arrows
\begin{equation}\label{diag:lift-log-fields}
\xymatrix{
\cC_{\USF} \ar@/^/[rrd] \ar@/_/[ddr]_{f_{\USF,-}} \ar@{.>}[rd]|-{f^{e}_{\USF,-}} && \\
&\cP^{e,\circ}_{\USF,-} \ar[r] \ar[d]^{\fb} & \cA^{e,\circ} \ar[d] \\
&\cP_{\USF,-} \ar[r] & \cA\times\cA_{\max}
}
\end{equation}
where the square is the pullback of  (\ref{equ:bundle-partial-expansion}) via $\USF \to \fU$.  We then obtain the dashed arrow $f^e_{\USF,-}$.
Consider the composition
\begin{equation}\label{equ:minus-section}
\bs_{\USF,-}\colon C_{\USF} \stackrel{f^e_{\USF,-}}{\longrightarrow} \cP^{e,\circ}_{\USF,-} \stackrel{\fc_\USF}{\longrightarrow} \vb(\cL_{\USF,-}\otimes \pi^*_{\USF}\bL^{\vee}_{\max})
\end{equation}
where $\fc_\USF$ is the pullback of the contraction $\fc$ as in Lemma \ref{lem:contraction}.

\begin{lemma}\label{lem:twisted-spin-marking-vanishing}
The section $\bs_{\USF}$ in (\ref{equ:twisted-spin-section}) is given by the composition
\[
C_{\USF} \stackrel{\bs_{\USF,-}}{\longrightarrow} \vb(\cL_{\USF,-}\otimes \pi^*_{\USF}\bL^{\vee}_{\max}) \stackrel{}{\longrightarrow} \vb(\cL_{\USF}\otimes \pi^*_{\USF}\bL^{\vee}_{\max}).
\]
\end{lemma}
\begin{proof}
Since $\bt\colon \cP_{\USF,reg} \to \cP_{\USF}$ is well-defined along the zero section, pulling back (\ref{equ:P-zero-section}) we have
\[
\bt^*\iota\colon \cO_{\cP_{\USF,reg}} \to \cO_{\cP_{\USF,reg}}(0_{\cP_{-}})\otimes \cO_{\cC_{\USF}}(\Sigma)|_{\cP_{\USF,reg}}.
\]
Since $\cO_{\cP_{\USF,reg}}(0_{\cP_{-}}) \cong \cL_{\USF,-}|_{\cP_{\USF,reg}}\otimes \cO_{\cP_{\USF,reg}}(\infty_{\cP_{-}})$, further pulling back to $\cP_{\USF,reg}^e = \fb^{-1}(\cP_{\USF,reg})$, we have
\[
\fb^*\bt^*\iota\colon \cO_{\cP^e_{\USF,reg}} \to \cL_{\USF,-}|_{\cP^e_{\USF,reg}}\otimes \cO_{\cP^e_{\USF,reg}}(\infty_{\cP^e_{-}}+\cE_{\fb}) \otimes \cO_{\cC_{\USF}}(\Sigma)|_{\cP^e_{\USF,reg}}
\]
which is naturally the restriction of
\[
\fb^*\bt^*\iota\colon \cO_{\cP^e_{\USF,-}} \to \cL_{\USF,-}|_{\cP^e_{\USF,-}}\otimes \cO_{\cP^e_{\USF,-}}(\infty_{\cP^e_{-}}+\cE_{\fb}) \otimes \cO_{\cC_{\USF}}(\Sigma)|_{\cP^e_{\USF,-}}.
\]
Since $f_{\USF}$ factors through $f^e_{\USF,-}$, we have $(f^e_{\USF,-})^*(\fb^*\bt^*\iota) = f^*_{\USF}\iota$.

By Lemma \ref{lem:compare-universal-log-twist}, we have $(f^e_{\USF,-})^*(\otimes\iota_{\cE_{\fc}}) = (\otimes\tf^{\vee}_{\USF})$ in (\ref{equ:tangent-twist}). Putting things together, we have
\begin{align*}
\bs_{\USF} = (\otimes\tf^{\vee}_{\USF}) \circ f^*_{\USF}\iota &= (f^e_{\USF,-})^*(\otimes\iota_{\cE_{\fc}}) \circ (f^e_{\USF,-})^*(\fb^*\bt^*\iota) \\
&= (f^e_{\USF,-})^*\big((\otimes\iota_{\cE_{\fc}}) \circ (\fb^*\bt^*\iota) \big).
\end{align*}
Note that $(\otimes\iota_{\cE_{\fc}}) \circ (\fb^*\bt^*\iota)$ is the morphism
\[
 \cO_{\cP^e_{\USF,-}} \to \cL_{\USF,-}|_{\cP^e_{\USF,-}}\otimes \cO_{\cP^e_{\USF,-}}(\infty_{\cP^e_{-}})\otimes\bL_{\max}^{\vee}\otimes\cO_{\cC_{\USF}}(\Sigma)|_{\cP^e_{\USF,-}}.
\]
which factors through the natural morphism
\begin{equation}\label{equ:twisted-spin-Pe}
 \cO_{\cP^e_{\USF,-}} \to \cL_{\USF,-}|_{\cP^e_{\USF,-}}\otimes \cO_{\cP^e_{\USF,-}}(\infty_{\cP^e_{-}})\otimes\bL_{\max}^{\vee}.
\end{equation}

Write $V = \vb(\cL_{\USF,-}\otimes \pi^*_{\USF}\bL^{\vee}_{\max})$ for simplicity.
The section $\bs_{\USF,-}$ is the pullback of the canonical morphism via itself
\[
\iota_{-}\colon \cO_{V} \to \cO_{V}(0_{V}) \cong (\cL_{\USF,-}\otimes \pi^*_{\USF}\bL^{\vee}_{\max})|_{V}.
\]
This pulls back to
\[
\fc^* \iota_{-}\colon \cO_{\cP^{e,\circ}_{\USF,-}} \to \cL_{\USF,-}|_{\cP^{e,\circ}_{\USF,-}} \otimes\bL_{\max}^{\vee},
\]
which is the restriction of (\ref{equ:twisted-spin-Pe}). Since $\bs_{\USF}$ factors through $f^{e}_{\USF,-}$, the section (\ref{equ:twisted-spin-Pe}) pulls back to $\bs_{\USF,-}$ via $f^{e}_{\USF,-}$. This finishes the proof.
\end{proof}

\subsubsection{Relative cosection via partial expansion}
For simplicity, write
\[
\tilde{\cL}_{\USF,-} := \cL_{\USF,-}\otimes \pi^*_{\USF}\bL^{\vee}_{\max} \ \ \mbox{and} \ \ \tilde{\omega}_{\USF} := \omega_{\USF} \otimes \pi^*_{\USF}\bL^{-r}_{\max}
\]
Consider the composition
\begin{equation}\label{equ:twist-composition}
\xymatrix{
\cP^{e,\circ}_{\USF,-} \ar[r]^{\fc \ \ \ \ } & \vb(\tilde{\cL}_{\USF,-}) \ar[r]^{W\ \ \ \ } \ar@/_1.5pc/[rr]_{W_{-}} & \vb\big(\tilde{\omega}_{\USF}((1-r)\Sigma)\big) \ar[r] & \vb(\tilde{\omega}_{\USF}).
}
\end{equation}
Take the differentiation
\[
T_{\cP^{e,\circ}_{\USF,-}/\cC_{\USF}} \stackrel{\diff \fc}{\longrightarrow} \fc^*T_{\vb(\tilde{\cL}_{\USF,-})/\cC_{\USF}} \stackrel{\diff W_{-}}{\longrightarrow} (W_{-}\circ\fc)^*\vb(\tilde{\omega}_{\USF}).
\]
Using \eqref{equ:twist-log-tangent} and pulling back to $\cC_{\USF}$, we have
\[
\cL_{\USF,-}\otimes f^{*}_{\USF}\cO(\infty_{\cP})  \stackrel{(f^e_{\USF,-})^*\diff \fc}{\longrightarrow} \tilde{\cL}_{\USF,-} \stackrel{\bs_{\USF,-}^*\diff W_{-}}{\longrightarrow} \tilde{\omega}_{\USF}.
\]
Further pushing forward, we obtain
\begin{equation}\label{equ:relative-cosection-via-expansion}
\begin{aligned}
\EE_{\USF/\fU} := \pi_{\USF,*}\big(\cL_{\USF,-}\otimes f^{*}_{\USF}\cO(\infty_{\cP})\big) &\stackrel{\pi_{\USF,*}(f^e_{\USF,-})^*\diff \fc}{\longrightarrow} \pi_{\USF,*}\tilde{\cL}_{\USF,-} \\
&\stackrel{\pi_{\USF,*}\bs_{\USF,-}^*\diff W_{-}}{\longrightarrow} \pi_{\USF,*}\tilde{\omega}_{\USF}
\end{aligned}
\end{equation}

\begin{proposition}\label{prop:relative-cosection-comparison}
The composition \eqref{equ:relative-cosection-via-expansion} is \eqref{equ:derived-cosection}. In particular, the relative cosection \eqref{equ:relative-cosection} is the $H^1$ of \eqref{equ:relative-cosection-via-expansion}.

\end{proposition}
\begin{proof}
By Lemma \ref{lem:compare-universal-log-twist}, we have $(f^e_{\USF,-})^*(\otimes\iota_{\cE_{\fc}}) = (\otimes\tf^{\vee}_{\USF})$, where $\iota_{\cE_{\fc}}$ is as in \eqref{equ:C-twist}.  By Lemma \ref{lem:diff-c}, $(f^e_{\USF,-})^*\diff \fc$ is obtained by tensoring \eqref{equ:tangent-twist} by $\cO_{\cC_{\USF}}(-\Sigma)$. By \eqref{equ:narrow-iso}, the arrow $\pi_{\USF,*}(f^e_{\USF,-})^*\diff \fc$ is \eqref{equ:push-tangent-twist}. Further observe that the arrow $\pi_{\USF,*}\bs_{\USF,-}^*\diff W_{-}$ is \eqref{equ:after-twist}. This proves the statement.
\end{proof}

\subsubsection{The twisted Hodge bundle}
Denote by $\tilde{\omega}_{\fU} := \omega_{\fU} \otimes \pi^*_{\fU}\bL^{-r}_{\max}$. Consider the direct image cone
$
\bC(\pi_{\fU,*}\tilde{\omega}_{\fU})
$
as in \cite[Definition 2.1]{CL12}.
This is an algebraic stack over $\fU$ parameterizing sections of $\tilde{\omega}_{\fU}$, see \cite[Proposition 2.2]{CL12}.
We further equip it with the log structure pulled back from $\fU$.
For simplicity, we write $\fH := \bC(\pi_{\fU,*}\tilde{\omega}_{\fU})$, and denote by $\bs_{\fH}\colon C_{\fH} \to \vb(\tilde{\omega}_{\fH})$ the universal section over $\fH$.

By \cite[Proposition 2.5]{CL12}, the strict morphism $\fH \to \fU$ has a perfect obstruction theory
\begin{equation}\label{equ:Hodge-perfect-obs}
\TT_{\fH/\fU} \to \EE_{\fH/\fU} := \pi_{\fH,*}\tilde{\omega}_{\fH}.
\end{equation}
By the projection formula, we have
\begin{equation}\label{equ:fake-obs}
R^1\pi_{\fU,*} \tilde{\omega}_{\fU} = (R^1\pi_{\fU, *}\omega_{\fU}\otimes \bL^{-r}_{\max}) \cong \bL^{-r}_{\max}.
\end{equation}
Therefore
$R^0\pi_{\fU,*} \tilde{\omega}_{\fU} \cong R^0\pi_{\fU,
  *}\omega_{\fU}\otimes \bL^{-r}_{\max}$ is indeed a vector
bundle whose associated geometric vector bundle is $\fH$.
In particular, the morphism $\fH \to \fU$ is strict and smooth.
Thus $\TT_{\fH/\fU}$ is a vector bundle over $\fH$ concentrated in
degree zero, and the following morphism is trivial:
\[
0 = H^1(\TT_{\fH/\fU}) \stackrel{0}{\longrightarrow} R^1\pi_{\fH,*} \tilde{\omega}_{\fH} \cong  \bL^{-r}_{\max}.
\]

The section $\bs_{\USF}$ as in \eqref{equ:twisted-spin-section} defines a section
\[
\bs_{\USF}^r\colon \cC_{\USF} \to \vb(\cL_{\USF}^{r}\otimes \pi^*_{\USF}\bL^{-r}_{\max}) \cong \vb(\omega_{\log,\USF}\otimes \pi^*_{\USF}\bL^{-r}_{\max}).
\]
By Lemma \ref{lem:twisted-spin-marking-vanishing}, $\bs_{\USF}$ is a global section of $\cL_{\USF}(-\Sigma)\otimes \pi^*_{\USF}\bL_{\max}$. Thus $\bs_{\USF}^r$ factors through a section
\[
\cC_{\USF} \to \vb(\omega_{\USF}\otimes \pi^*_{\USF}\bL^{-r}_{\max}),
\]
which is again denoted by $\bs_{\USF}^r$. This induces a morphism
\[\USF \to \fH\]
such that  $\bs_{\USF}^r$ is the pullback of $\bs_{\fH}$.

\subsubsection{Obstruction factorization}

\begin{lemma}\label{lem:obs-commute}
There is a canonical commutative diagram
\begin{equation}\label{diag:rel-obs-commute}
\xymatrix{
\TT_{\USF/\fU} \ar[r] \ar[d] & \TT_{\fH/\fU}|_{\USF} \ar[d] \\
\EE_{\USF/\fU} \ar[r] & \EE_{\fH/\fU}|_{\USF}
}
\end{equation}
where the bottom arrow is \eqref{equ:relative-cosection-via-expansion}, and the left and right vertical arrows are the perfect obstruction theories (\ref{equ:log-perfect-obs}) and (\ref{equ:Hodge-perfect-obs}) respectively.
\end{lemma}

\begin{proof}
  Consider the commutative diagram over $\cC_{\fU}$:
\[
\xymatrix{
 \cC_\USF \ar[rr] \ar[d]_{f_{\USF,-}^e} && \cC_\fH \ar[d]^{s_\fH} \\
 \cP^{e,\circ}_{\fU,-}   \ar[rr]^{(W_{-}) \circ \fc} && \vb(\tilde{\omega}_{\fU})
      }
\]
By abuse of notation, $s_\fH$ is the composition $ \cC_\fH \to \vb(\tilde{\omega}_{\fH}) \to \vb(\tilde{\omega}_{\fU})$. We obtain a commutative diagram of log tangent complexes
  \begin{equation*}
    \xymatrix{
      \pi_{\USF}^* \TT_{\USF/\fU} \cong \TT_{\cC_\USF/\cC_\fU} \ar[rr] \ar[d] &&       \TT_{\cC_\fH/\cC_\fU}|_{\cC_{\USF}} \ar[d]\\
      (f_{\USF,-}^e)^* \TT_{\cP^{e,\circ}_{\fU,-}/\cC_\fU} \ar[rr]^{(f_{\USF,-}^e)^*\diff (W_{-}) \circ \fc}  && (s_\fH)^* \TT_{\vb(\tilde{\omega}_{\fU})/\cC_\fU}|_{\cC_\USF}
      }
  \end{equation*}
  Since $\fb$ in \eqref{diag:lift-log-fields} is log \'etale, we have
  $(f_{\USF,-}^e)^* \TT_{\cP^{e,\circ}_{\fU,-}/\cC_\fU} \cong
  (f_{\USF,-})^* \TT_{\cP_{\fU,-}/\cC_\fU}$.
  Applying $\pi_{\USF,*}$ and using adjunction we obtain
  \begin{equation*}
    \xymatrix{
      \TT_{\USF/\fU} \ar[rrr] \ar[d] &&&       \TT_{\fH/\fU}|_\USF \ar[d] \\
       \pi_{\USF,*}(f_{\USF,-})^* \TT_{\cP_{\fU,-}/\cC_\fU} \ar[rrr]^{\pi_{\USF,*} (f_{\USF,-}^e)^*\diff (W_{-}) \circ \fc} &&&  \pi_{\USF,*}(s_\fH)^* \TT_{\vb(\tilde{\omega}_{\fU})/\cC_\fU}|_{\cC_\USF}
      }
  \end{equation*}
which is \eqref{diag:rel-obs-commute}.
\end{proof}

\begin{proposition}\label{prop:rel-obs-factorization}
The injection $H^1(\TT_{\USF/\fU}) \to \obs_{\USF/\fU}$ factors through the kernel of the relative cosection $\sigma_{\USF/\fU}$ in (\ref{equ:relative-cosection}).
\end{proposition}
\begin{proof}
By Lemma \ref{lem:obs-commute}, taking $H^1$ of (\ref{diag:rel-obs-commute}), we obtain a commutative diagram
\[
\xymatrix{
H^1(\TT_{\USF/\fU}) \ar[r] \ar[d] & H^1(\TT_{\fH/\fU}) = 0 \ar[d] \\
\obs_{\USF/\fU} \ar[r]^{\sigma_{\USF/\fU}} & \bL_{\max}^{-r}
}
\]
where $H^1(\TT_{\fH/\fU}) = 0$ follows from the smoothness of $\fH \to \fU$.
\end{proof}

\subsection{The reduced relative perfect obstruction theory}\label{ss:relative-reduced-POT}
The dual of \eqref{equ:boundary-complex} induces a complex with amplitude $[0,1]$ over $\fU$:
\[
\FF := \cO_{\fH} \stackrel{\epsilon}{\longrightarrow} \bL_{\max}^{-r}.
\]
Since $\fH \to \fU$ is log smooth, $\epsilon$ is injective. Consider the cokernel $\cok \epsilon$. Then $\FF = \cok\epsilon[-1]$ in the derived category.  The composition
\[
\EE_{\fH/\fU} \to H^1(\EE_{\fH/\fU})[-1] \cong \bL^{-r}_{\max}[-1]  \twoheadrightarrow \cok\epsilon[-1]
\]
defines a morphism of complexes
$
\EE_{\fH/\fU} \to \FF|_{\fH},
$
and hence a triangle
\begin{equation}\label{tri:reduce-Y}
\EE_{\fH/\fU}^{\red} \longrightarrow \EE_{\fH/\fU} \longrightarrow \FF|_{\fH} \stackrel{[1]}{\longrightarrow}
\end{equation}
where the notation $|_{*}$ stands for derived pullback to $*$.

\begin{lemma}\label{lem:red-O}
$H^1(\EE^{\red}_{\fH/\fU}) = \cO_\fH$.
\end{lemma}
\begin{proof}
Taking the long exact sequence of (\ref{tri:reduce-Y}) and using $H^0(\FF) = 0$, we have an exact sequence
\[
0 \to H^1(\EE_{\fH/\fU}^{\red}) \to H^1(\EE_{\fH/\fU}) \to H^1(\FF|_{\fH})\to H^2(\EE_{\fH/\fU}^{\red}).
\]
Since $H^1(\EE_{\fH/\fU}) \to H^1(\FF|_{\fH})$ is precisely the morphism $\bL^{-r}_{\max}  \twoheadrightarrow \cok\epsilon$, it follows that $H^1(\EE^{\red}_{\fH/\fU}) = \cO_\fH$.
\end{proof}

The composition
\begin{equation}\label{equ:composing-to-moduli}
\EE_{\USF/\fU} \to \EE_{\fH/\fU}|_{\USF} \to \FF|_{\USF},
\end{equation}
yields a triangle
\begin{equation}\label{tri:reduce-X}
\EE_{\USF/\fU}^{\red} \longrightarrow \EE_{\USF/\fU} \longrightarrow \FF|_{\USF} \stackrel{[1]}{\longrightarrow}
\end{equation}

\begin{lemma}\label{lem:red-obs-factor}
  The obstruction theories $\TT_{\fH/\fU} \to \EE_{\fH/\fU}$ and
  $\TT_{\USF/\fU} \to \EE_{\USF/\fU}$ factor through
  $\TT_{\fH/\fU} \to \EE_{\fH/\fU}^{\red}$ and
  $\TT_{\USF/\fU} \to \EE_{\USF/\fU}^{\red}$ respectively.
  Furthermore, they fit in a commutative diagram
\begin{equation}\label{diag:red-obs-commute}
\xymatrix{
\TT_{\USF/\fU} \ar[r] \ar[d] & \TT_{\fH/\fU}|_{\USF} \ar[d] \\
\EE^{\red}_{\USF/\fU} \ar[r] & \EE^{\red}_{\fH/\fU}|_{\USF}
}
\end{equation}

\end{lemma}
\begin{proof}
By Lemma \ref{lem:obs-commute}, we have a commutative diagram of solid arrows:
\begin{equation}\label{diag:reduce-X-Y}
\xymatrix{
\TT_{\USF/\fU} \ar@/^1pc/[rrd] \ar@{-->}[rd]  \ar[dd] &&&& \\
 & \EE^{\red}_{\USF/\fU} \ar[dd] \ar[r] & \EE_{\USF/\fU} \ar[dd] \ar[r] & \FF|_{\USF} \ar[r]^{[1]} \ar@{=}[dd] & \\
\TT_{\fH/\fU}|_{\USF} \ar@/^1pc/[rrd]|{\ \ \ \hole \ \   } \ar@{-->}[rd] &&&& \\
 & \EE^{\red}_{\fH/\fU}|_{\USF} \ar[r] & \EE_{\fH/\fU}|_{\USF} \ar[r] & \FF|_{\USF} \ar[r]^{[1]} &
}
\end{equation}
where the two horizontal lines are triangles (\ref{tri:reduce-X}) and
(\ref{tri:reduce-Y}), and the two curved arrows are the corresponding
obstruction theories.

Since $\fH \to \fU$ is representable and smooth, the complex $\TT_{\fH/\fU}$ is represented by the relative tangent bundle $T_{\fH/\fU}$. Thus the composition $\TT_{\fH/\fU} \to \EE_{\fH/\fU} \to \FF|_{\fH}$ is the zero morphism. This yields the lower dashed arrow $\TT_{\fH/\fU} \to \EE_{\fH/\fU}$.

Now by the commutativity, the composition $\TT_{\USF/\fU} \to \EE_{\USF/\fU} \to \FF|_{\USF}$ is the same as $\TT_{\USF/\fU} \to \TT_{\fH/\fU}|_{\USF} \to \FF|_{\USF}$, hence is trivial. Thus, we obtain the top dashed arrow $\TT_{\USF/\fU} \to \EE_{\USF/\fU}$.
\end{proof}

\begin{lemma}\label{lem:red-perfect}
The two complexes $\EE^{\red}_{\USF/\fU}$ and $\EE^{\red}_{\fH/\fU}$ are perfect with tor-amplitude in $[0,1]$.
\end{lemma}
\begin{proof}
Since $\EE_{\fH/\fU}$, $\EE_{\USF/\fU}$ and $\FF$ are perfect in $[0,1]$, the complexes $\EE^{\red}_{\fH/\fU}$ and $\EE^{\red}_{\USF/\fU}$ are at least perfect in $[0,2]$. It suffices to show that $H^2(\EE^{\red}_{\fH/\fU}) = 0$ and $H^2(\EE^{\red}_{\USF/\fU}) = 0$.

Taking the long exact sequence of (\ref{tri:reduce-Y}), we have an exact sequence
\[
H^{1}(\EE_{\fH/\fU}) \to H^1(\FF|_{\fH}) \to H^2(\EE^{\red}_{\fH/\fU}) \to 0
\]
Since the left arrow is $\bL^{-r}_{\max}  \twoheadrightarrow \cok\epsilon$, we have $H^2(\EE^{\red}_{\fH/\fU}) = 0$.

Similarly using (\ref{tri:reduce-X}), we have an exact sequence
\[
H^{1}(\EE_{\USF/\fU}) \to H^1(\FF|_{\USF}) \to H^2(\EE^{\red}_{\USF/\fU}) \to 0.
\]
By \eqref{equ:composing-to-moduli}, the left arrow is the composition
\[
H^{1}(\EE_{\USF/\fU}) \to H^{1}(\EE_{\fH/\fU}|_{\USF}) \twoheadrightarrow  H^1(\FF|_{\USF}).
\]
By construction, $\FF|_{\USF\setminus\Delta_{\max}} = 0$ is the
zero complex.
It suffices to show that the above composition is surjective along a
neighborhood of $\Delta_{\max}$.
This follows from Proposition~\ref{prop:cosection-boundary-surjective}
given that Proposition~\ref{prop:relative-cosection-comparison}
identifies the morphism
$H^{1}(\EE_{\USF/\fU}) \to H^{1}(\EE_{\fH/\fU}|_{\USF})$ with the
relative cosection $\sigma_{\USF/\fU}$.
\end{proof}

\begin{lemma}\label{lem:red-obs}
The two arrows $\TT_{\fH/\fU} \to \EE_{\fH/\fU}^{\red}$ and $\TT_{\USF/\fU} \to \EE_{\USF/\fU}^{\red}$ define perfect obstruction theories of $\fH \to \fU$ and $\USF \to \fU$ respectively.
\end{lemma}
\begin{proof}
We verify the case of $\TT_{\USF/\fU} \to \EE_{\USF/\fU}^{\red}$.
The other case is similar.
By the triangle \eqref{tri:reduce-X} and the factorization of Lemma \ref{lem:red-obs-factor}, we have a surjection $H^0(\TT_{\USF/\fU}) \twoheadrightarrow H^0(\EE^{\red}_{\USF/\fU})$ and an injection $H^1(\TT_{\USF/\fU}) \hookrightarrow H^1(\EE^{\red}_{\USF/\fU})$.
Since $\FF$ is perfect in $[0,1]$, (\ref{tri:reduce-X}) implies that $H^0(\TT_{\USF/\fU}) \to H^0(\EE^{\red}_{\USF/\fU})$ is also injective, and hence an isomorphism.
\end{proof}

The proof of the above lemma leads to the following

\begin{corollary}\label{cor:red-cos-surj}
\begin{enumerate}
 \item $H^0(\EE_{\USF/\fU}^{\red}) = H^0(\EE_{\USF/\fU})$.
 \item Diagram (\ref{diag:reduce-X-Y}) induces a morphism between long exact sequences
 \[
 \xymatrix{
 0 \ar[r] & H^0(\FF|_{\USF}) \ar[r] \ar[d]^{\cong} & H^1(\EE^{\red}_{\USF/\fU}) \ar[r] \ar[d]^{\sigma^{\red}_{\USF/\fU}} & H^1(\EE_{\USF/\fU}) \ar[r] \ar[d]^{\sigma_{\USF/\fU}} & H^1(\FF|_{\USF}) \ar[r]  \ar[d]^{\cong}& 0 \\
 0 \ar[r] & H^0(\FF|_{\USF}) \ar[r]  & H^1(\EE^{\red}_{\fH/\fU}|_{\USF}) \ar[r]  & H^1(\EE_{\fH/\fU}|_{\USF}) \ar[r]  & H^1(\FF|_{\USF}) \ar[r]  & 0
 }
 \]
where the morphism $\sigma^{\red}_{\USF/\fU}$ is surjective along $\Delta_{\max}$.
\end{enumerate}
\end{corollary}
\begin{proof}
It remains to verify the surjectivity of $\sigma^{\red}_{\USF/\fU}$ along $\Delta_{\max}$. This follows from the surjectivity of $\sigma_{\USF/\fU}$ along $\Delta_{\max}$ by Proposition \ref{prop:cosection-boundary-surjective}.
\end{proof}

We summarizes our construction below.

\begin{proposition}\label{prop:reduced-obs}
The morphism $\USF \to \fU$ admits a {\em reduced perfect obstruction theory}
\begin{equation}\label{equ:reduced-relative-obs}
\TT_{\USF/\fU} \to \EE^{\red}_{\USF/\fU},
\end{equation}
and a {\em reduced relative cosection}
\begin{equation}\label{equ:reduced-relative-cosection}
\sigma^{\red}_{\USF/\fU}\colon \obs^{\red}_{\USF/\fU} := H^1(\EE^{\red}_{\USF/\fU}) \to \cO_{\USF}
\end{equation}
with the following properties
\begin{enumerate}
 \item $\EE^{\red}_{\USF/\fU}|_{\USF\setminus\Delta_{\max}} = \EE_{\USF/\fU}|_{\USF\setminus\Delta_{\max}}$.

 \item $\sigma^{\red}_{\USF/\fU}|_{\USF\setminus\Delta_{\max}} = \sigma_{\USF/\fU}|_{\USF\setminus\Delta_{\max}}$

 \item $\sigma^{\red}_{\USF/\fU}$ is surjective along $\Delta_{\max}$.
\end{enumerate}
In particular, $\sigma^{\red}_{\USF/\fU}$ and $\sigma_{\USF/\fU}$ have the same degeneracy loci.
\end{proposition}
\begin{proof}
The perfect obstruction theory has been verified in Lemma \ref{lem:red-perfect} and \ref{lem:red-obs}. The formation of $\sigma^{\red}_{\USF/\fU}$ and its surjectivity along $\Delta_{\max}$ follows from Corollary \ref{cor:red-cos-surj}.

Finally (1) follows from the observation $\FF|_{\USF \setminus \Delta_{\max}} = 0$. Statement (2) follows from (1) and \eqref{diag:reduce-X-Y}.
\end{proof}

Comparing \eqref{equ:reduced-relative-cosection} and
\eqref{equ:relative-cosection}, we observe that the reduced and
canonical cosections only differ along the boundary $\Delta_{\max}$,
and are related by Corollary \ref{cor:red-cos-surj} (2).

\begin{notation}\label{not:red-virtual-cycle}
Since $\fU$ is equi-dimensional, denote by $[\USF]^{\red}$ the virtual fundamental class of $\USF$ defined by the relative perfect obstruction theory \eqref{equ:reduced-relative-obs}, see \cite[Section~7]{BF97}.
\end{notation}

\subsection{The reduced absolute perfect obstruction theory}\label{ss:absolute-reduced-POT}

Our last goal is to compare the cosection localized virtual cycle and
the reduced virtual cycle as in Theorem \ref{thm:main}.
Since the cosection localized virtual cycle is defined using the
absolute theory \cite{CLL15}, we need to descend the relative reduced
theory in Proposition \ref{prop:reduced-obs} to an absolute one.
However, the log smooth base stack $\fU$ can have toroidal
singularities.
Hence the standard method constructing an absolute theory from a
relative one as explained in \cite[Proposition A.1]{BL00} does not
directly apply.
To fix this, we first construct a (non-canonical) resolution of $\fU$
in \S \ref{sss:resolution} leaving $\fU\setminus \Delta_{\max}$
untouched.
Pulling back along the resolution, we can then descend the relative
reduced theory to an absolute one in \S \ref{sss:absolute-reduced}.
Finally in \S \ref{sss:final-proof}, we compare the cosection
localized virtual cycle with the absolute reduced virtual cycle, and
further argue that the result is independent of the choice of
resolutions. 

\subsubsection{Resolution of the base}\label{sss:resolution}

\begin{lemma}\label{lem:base-resolution}
Let $\fV \subset \fU$ be a finite type open substack, and write $\Delta_{\max, \fV} = \Delta_{\max} \cap \fV$. Then there exists a birational, log \'etale,  projective morphism of log stacks
$
\tilde{\phi} \colon \tilde{\fV} \to \fV
$
such that
\begin{enumerate}
 \item $\tilde{\phi}|_{\fV\setminus \Delta_{\max,\fV}}$ is an isomorphism onto $\fV\setminus \Delta_{\max,\fV}$.
 \item The log structure of $\tilde{\fV}$ is locally free. In particular, the underlying stack of $\tilde{\fV}$ is smooth.
\end{enumerate}
\end{lemma}
\begin{proof}
  Recall from Corollary~\ref{cor:separate-non-distinct-log-U} that
  there is a canonical splitting
  $\cM_{\fV} = \cM'_{\fV}\oplus_{\cO^*}\cM''_{\fV}$ where
  $\oM''_{\fV,s} = \NN^d$ is the factor corresponding to nodes with
  the trivial contact order for each geometric point $s \to \fV$.
  Indeed, given a node over $s$, if it has the trivial contact order,
  then it can either be smoothed out or remain the trivial contact
  order in a neighborhood of $s$.
  Observe that $\cM'_{\fV}$ is trivial along
  $\fV\setminus \Delta_{\max,\fV}$ as the curves have no degenerate
  components away from $\Delta_{\max}$.

  Consider the Artin fans $\cA'_{\fV}$ and $\cA''_{\fV}$ associated
  to $\cM'_{\fV}$ and $\cM''_{\fV}$ respectively,
  see \cite[Proposition 3.1.1]{ACMW17}.
  By Theorem \ref{thm:max-uni-moduli}, we have a strict, smooth
  morphism of log stacks $\fV \to \cA'_{\fV}\times \cA''_{\fV}$.
  Let $\cY \to \cA'_{\fV}$ be the projective sub-division of
  \cite[Theorem 4.4.2]{ACMW17}.
  It is projective and log \'etale, and $\cM_{\cY}$ is
  locally free.
  This induces a projective, log \'etale morphism
  \[
    \tilde{\phi} \colon \tilde{\fV} := \fV\times_{\cA'_{\fV}\times\cA''_{\fV}}(\cY \times \cA''_{\fV}) \to \fV.
  \]
  It is an isomorphism over $\fV\setminus \Delta_{\max,\fV}$, over which $\cM'_{\fV}$ is trivial.
\end{proof}

Let $\fV \subset \fU$ be a finite type open substack containing the
image of $\USF$.
We fix a resolution $\tilde{\phi} \colon \tilde{\fV} \to \fV$ as in
Lemma~\ref{lem:base-resolution}.
Consider the fiber products
\[
 \tilde{\fH} := \fH\times_{\fU}\tilde{\fV} \ \ \ \mbox{and} \ \ \ \tilde{\USF} := \USF\times_{\fU}\tilde{\fV}.
\]
The perfect obstruction theories $\TT_{\fH/\fU} \to \EE_{\fH/\fU}^{\red}$ and $\TT_{\USF/\fU} \to \EE_{\USF/\fU}^{\red}$ in Lemma \ref{lem:red-obs} pull back to perfect obstruction theories
\[
 \TT_{\tilde{\fH}/\tilde{\fV}} \to \EE_{\tilde{\fH}/\tilde{\fV}}^{\red} \  \  \ \mbox{and} \ \ \ \TT_{\tilde{\USF}/\tilde{\fV}} \to \EE_{\tilde{\USF}/\tilde{\fV}}^{\red}.
\]
Since $\tilde{\fV}$ is equi-dimensional, let $[\tilde{\USF}]^{\red}$
be the virtual cycle of $\tilde{\USF}$ defined by the above perfect
obstruction theory as in \cite[Section~7]{BF97}.
By Lemma \ref{lem:base-resolution} and the virtual push-forward of
\cite{Co06, Ma12}, we obtain:

\begin{lemma}\label{lem:push-along-resolution}
$\tilde{\phi}_{*}[\tilde{\USF}]^{\red} = [\USF]^{\red}$
\end{lemma}

\subsubsection{The absolute reduced theory and cosection}\label{sss:absolute-reduced}

We define $\EE^{\red}_{\tilde{\fH}}$ to be a cone fitting into the
following morphism of distinguished triangles:
\begin{equation}\label{diag:abs-obs-H}
\xymatrix{
\TT_{\tilde{\fH}/\tilde{\fV}} \ar[r] \ar[d] & \TT_{\tilde{\fH}} \ar[r] \ar[d] & \TT_{\tilde{\fV}}|_{\tilde{\fH}} \ar[r]^{[1]} \ar[d]^{\cong} & \\
\EE^{\red}_{\tilde{\fH}/\tilde{\fV}} \ar[r]  & \EE^{\red}_{\tilde{\fH}} \ar[r]  & \TT_{\tilde{\fV}}|_{\tilde{\fH}} \ar[r]^{[1]}  &
}
\end{equation}

\begin{lemma}\label{lem:lift-cosection}
The induced morphism $H^1(\EE^{\red}_{\tilde{\fH}/\tilde{\fV}}) \to H^1(\EE^{\red}_{\tilde{\fH}})$ is an isomorphism and $H^1(\EE^{\red}_{\tilde{\fH}}) \cong \cO_{\tilde{\fH}}$.
\end{lemma}
\begin{proof}
Since $\tilde{\fV}$ is smooth, we have $H^1(\TT_{\tilde{\fV}}) = 0$. Consider the induced morphism between long exact sequences
\[
\xymatrix{
H^{0}(\TT_{\tilde{\fH}}) \ar[r] \ar[d]^{\cong} & H^{0}(\TT_{\tilde{\fV}}|_{\tilde{\fH}}) \ar[r] \ar[d]^{\cong} & H^{1}(\TT_{\tilde{\fH}/\tilde{\fV}}) \ar[r] \ar[d] & H^{1}(\TT_{\tilde{\fH}}) \ar[r] \ar[d] & 0 \\
H^{0}(\EE^{\red}_{\tilde{\fH}}) \ar[r] & H^{0}(\TT_{\tilde{\fV}}|_{\tilde{\fH}}) \ar[r] & H^{1}(\EE^{\red}_{\tilde{\fH}/\tilde{\fV}}) \ar[r] & H^{1}(\EE^{\red}_{\tilde{\fH}}) \ar[r] & 0
}
\]
Since $\tilde{\fH} \to \tilde{\fV}$ is smooth, $H^{0}(\TT_{\tilde{\fH}}) \to H^{0}(\TT_{\tilde{\fV}}|_{\tilde{\fH}})$ and $H^{0}(\EE^{\red}_{\tilde{\fH}}) \to H^{0}(\TT_{\tilde{\fV}}|_{\tilde{\fH}})$ are both surjective. Thus $H^1(\EE^{\red}_{\tilde{\fH}/\tilde{\fV}}) \to H^1(\EE^{\red}_{\tilde{\fH}})$ is an isomorphism. Lemma \ref{lem:red-O} implies that $H^1(\EE^{\red}_{\tilde{\fH}}) \cong \cO_{\tilde{\fH}}$.
\end{proof}

Now consider the morphism of triangles:
\begin{equation}\label{diag:abs-obs-X}
\xymatrix{
\TT_{\tilde{\USF}/\tilde{\fV}} \ar[r] \ar[d] & \TT_{\tilde{\USF}} \ar[r] \ar[d]^{\varphi_{\tilde{\USF}}} & \TT_{\tilde{\fV}} \ar[r]^{[1]} \ar[d]^{\cong} & \\
\EE^{\red}_{\tilde{\USF}/\tilde{\fV}} \ar[r]  & \EE^{\red}_{\tilde{\USF}} \ar[r]  & \TT_{\tilde{\fV}} \ar[r]^{[1]}  &
}
\end{equation}
By \cite[Proposition A.1. (1)]{BL00}, we obtain a perfect obstruction theory $\TT_{\tilde{\USF}} \to \EE^{\red}_{\tilde{\USF}}$ of $\tilde{\USF}$ with the corresponding virtual cycle $[\tilde{\USF}]^{\red}$.

The bottom morphism in (\ref{diag:red-obs-commute}) induces a morphism of triangles
\[
\xymatrix{
\EE^{\red}_{\tilde{\USF}/\tilde{\fV}}|_{\tilde{\USF}} \ar[r] \ar[d]  & \EE^{\red}_{\tilde{\USF}}|_{\tilde{\USF}} \ar[r] \ar[d] & \TT_{\tilde{\fV}}|_{\tilde{\USF}} \ar[r]^{[1]}  \ar[d] & \\
\EE^{\red}_{\tilde{\fH}/\tilde{\fV}} \ar[r]  & \EE^{\red}_{\tilde{\fH}} \ar[r]  & \TT_{\tilde{\fV}}|_{\tilde{\USF}} \ar[r]^{[1]}  &
}
\]
Taking $H^1$ and applying Lemma \ref{lem:lift-cosection}, we have a commutative diagram
\[
\xymatrix{
H^1(\EE^{\red}_{\tilde{\USF}/\tilde{\fV}}) \ar@{->>}[r] \ar[d]_{\sigma^{\red}_{\tilde{\USF}/\tilde{\fV}}} & H^1(\EE^{\red}_{\tilde{\USF}}) \ar[d]^{\sigma^{\red}_{\tilde{\USF}}} \\
\cO \ar[r]^{=} &  \cO.
}
\]
Observe that $\sigma^{\red}_{\tilde{\USF}/\tilde{\fV}}$ is the pullback of $\sigma^{\red}_{\USF/\fU}$ in \eqref{equ:reduced-relative-cosection}. We call $\sigma^{\red}_{\tilde{\USF}}$ the {\em absolute reduced cosection}.

\subsubsection{Proof of Theorem \ref{thm:main}}\label{sss:final-proof}

Denote by
$\tilde{\Delta}_{\max} := \tilde{\USF}\times_{\USF}\Delta_{\max}$. By
Lemma \ref{lem:base-resolution} (1), we have the identity
$\USF^{\circ} := \tilde{\USF} \setminus \tilde{\Delta}_{\max} = \USF
\setminus \Delta_{\max}$.
Consider the open embedding
$\iota\colon \USF^{\circ} \hookrightarrow \tilde{\USF}$ with the
trivial perfect obstruction theory.
Thus the virtual pullback $\iota^{!}$ in the sense of \cite{Ma12} is
just the flat pullback (see \cite[Remark~3.10]{Ma12}).

Denote by $\sigma_{\USF^{\circ}} = \sigma_{\tilde{\USF}}^{\red}|_{\USF^{\circ}}$. By Lemma \ref{lem:base-resolution}, \ref{lem:old-compatible}, and Proposition \ref{prop:reduced-obs} (2), the morphism $\sigma_{\USF^{\circ}}$ is the absolute cosection in \cite[Proposition 3.4]{CLL15} in the $r$-spin case. We then obtain:

\begin{lemma}
$[\USF^{\circ}]_{\sigma_{\USF^{\circ}}}$ is the {\em Witten's top Chern class} as in \cite[Definition-Proposition 3.9]{CLL15}.
\end{lemma}

On the other hand, let $\tilde{\USF}(\sigma_{\tilde{\USF}}^{\red})$ (respectively $\USF^{\circ}(\sigma_{\USF^{\circ}})$) be the degeneracy loci of $\sigma_{\tilde{\USF}}^{\red}$ (respectively $\sigma_{\USF^{\circ}}$). Since $\sigma^{\red}_{\tilde{\USF}/\tilde{\fV}}$ is the pullback of $\sigma^{\red}_{\USF/\fU}$, Proposition \ref{prop:reduced-obs} (3) implies that $\sigma_{\tilde{\USF}}^{\red}$ is surjective along $\tilde{\Delta}_{\max}$, hence $\tilde{\USF}(\sigma_{\tilde{\USF}}^{\red}) = \USF^{\circ}(\sigma_{\tilde{\USF}}^{\red})$.

Let $\iota^{!}_{\sigma_{\tilde{\USF}}^{\red}}$ be the cosection localized virtual pullback as in \cite[Section~2.1]{CKL17}. Since $\iota^{!}_{\sigma_{\tilde{\USF}}^{\red}} = \iota^{!}$ and $\tilde{\USF}(\sigma_{\tilde{\USF}}^{\red}) = \USF^{\circ}(\sigma_{\USF^{\circ}})$, applying \cite[Theorem 2.6]{CKL17} we have the following equalities in $A_*(\USF^{\circ}(\sigma_{\USF^{\circ}}))$:
\[
[\tilde{\USF}]^{\red}_{\sigma_{\tilde{\USF}}^{\red}} = \iota^{!}_{\sigma_{\tilde{\USF}}^{\red}}[\tilde{\USF}]^{\red}_{\sigma_{\tilde{\USF}}^{\red}} = [\USF^{\circ}]_{\sigma_{\USF^{\circ}}}.
\]
Let $\tilde{i}\colon \USF^{\circ}(\sigma_{\USF^{\circ}}) \to \tilde{\USF}$ be the closed embedding.
By \cite[Theorem~1.1]{KiLi13}, we have:

\begin{lemma}
$\tilde{i}_{*}[\USF^{\circ}]_{\sigma_{\USF^{\circ}}} = [\tilde{\USF}]^{\red}$.
\end{lemma}

Finally, let $i = \tilde{\phi}\circ \tilde{i}\colon \USF^{\circ}(\sigma_{\USF^{\circ}}) \to \USF$ be the closed embedding. Applying Lemma \ref{lem:push-along-resolution}, we have:
\begin{proposition}\label{prop:comparison}
$i_{*}[\USF^{\circ}]_{\sigma_{\USF^{\circ}}} = [\USF]^{\red}$.
\end{proposition}
This completes the proof of Theorem \ref{thm:main}.

\bibliographystyle{abbrv}
\bibliography{rspin}

\end{document}